\documentclass[onefignum,onetabnum]{siamart171218}

\usepackage{amsfonts}
\usepackage{algorithmic}
\usepackage{amsmath,amssymb} 
\usepackage{url,cite}
\usepackage{graphicx}
\usepackage{tikz,pgfplots}
\newcommand{\blue}{\color{black}}
\newcommand{\bx}{\boldsymbol{x}}

\newcommand{\by}{\boldsymbol{y}}
\newcommand{\bz}{\boldsymbol{z}}

\newcommand{\bxi}{\boldsymbol{\xi}}
\newcommand{\bbeta}{\boldsymbol{\eta}}

\newcommand{\bzeta}{\boldsymbol{\zeta}}

\newcommand{\bq}{\boldsymbol{q}}
\newcommand{\bQ}{\boldsymbol{Q}}
\newcommand{\bS}{\boldsymbol{S}}

\newcommand{\bp}{\boldsymbol{p}}
\newcommand{\bff}{\boldsymbol{f}}
\newcommand{\bc}{\boldsymbol{c}}

\newcommand{\bA}{\boldsymbol{A}}
\newcommand{\bL}{\boldsymbol{L}}

\newcommand{\bB}{\boldsymbol{B}}

\newcommand{\bH}{\boldsymbol{H}}
\newcommand{\bR}{\boldsymbol{R}}
\newcommand{\bn}{\boldsymbol{n}}
\newcommand{\bm}{\boldsymbol{m}}
\newcommand{\br}{\boldsymbol{r}}
\newcommand{\bl}{\boldsymbol{l}}
\newcommand{\bI}{\boldsymbol{I}}
\newcommand{\olambda}{\overline{\lambda}}
\newcommand{\ulambda}{\underline{\lambda}}
\newcommand{\osigma}{\overline{\sigma}}
\newcommand{\usigma}{\underline{\sigma}}
\newcommand{\oosigma}{\overline{\osigma}}
\newcommand{\uusigma}{\underline{\usigma}}
\newcommand{\bzero}{\boldsymbol{0}}
\newcommand{\st}{\mathop{\text{\normalfont s.t.}}}

\newcommand{\bbZ}{\mathbb{Z}}
\newcommand{\cG}{\mathcal{G}}
\newcommand{\cV}{\mathcal{V}}

\newcommand{\cA}{\mathcal{A}}
\newcommand{\cB}{\mathcal{B}}
\newcommand{\cL}{\mathcal{L}}
\newcommand{\cE}{\mathcal{E}}
\newcommand{\cP}{\mathcal{P}}

\newcommand{\cK}{\mathcal{K}}

\newcommand{\cN}{\mathcal{N}}
\newcommand{\cI}{\mathcal{I}}
\newcommand{\cJ}{\mathcal{J}}

\allowbreak
\allowdisplaybreaks
\newsiamremark{example}{Example}
\newsiamthm{assumption}{Assumption}
\newsiamremark{remark}{Remark}

\ifpdf
  \DeclareGraphicsExtensions{.eps,.pdf,.png,.jpg}
\else
  \DeclareGraphicsExtensions{.eps}
\fi

\title{Exponential Decay of Sensitivity in \\  Graph-Structured Nonlinear Programs}

\author{
  Sungho Shin\thanks{Department of Chemical and Biological Engineering, University of Wisconsin-Madison, Madison, WI (\email{sungho.shin@wisc.edu})},
  \and Mihai Anitescu\thanks{Mathematics and Computer Science Division, Argonne National Laboratory, Argonne, IL and Department of Statistics, University of Chicago, Chicago, IL(\email{anitescu@mcs.anl.gov})},
  \and Victor M. Zavala\thanks{Department of Chemical and Biological Engineering, University of Wisconsin-Madison, Madison, WI and Mathematics and Computer Science Division, Argonne National Laboratory, Argonne, IL (\email{victor.zavala@wisc.edu})}}

\usepackage{amsopn}

\ifpdf
\hypersetup{
  pdftitle={Exponential Decay of Sensitivity in Graph-Structured Nonlinear Programs},
  pdfauthor={S. Shin, M. Anitescu, and V. M. Zavala}
}
\fi

\begin{document}

\maketitle

\begin{abstract}
  We study solution sensitivity for nonlinear programs (NLPs) whose structures are induced by graphs. These NLPs arise in many applications such as dynamic optimization, stochastic optimization, optimization with partial differential equations, and network optimization. We show that for a given pair of nodes, the sensitivity of the primal-dual solution at one node against a data perturbation at the other node decays exponentially with respect to the distance between these two nodes on the graph. In other words, the solution sensitivity decays as one moves away from the perturbation point. This result, which we call exponential decay of sensitivity, holds under the strong second-order sufficiency condition and the linear independence constraint qualification. We also present conditions under which the decay rate remains uniformly bounded; this allows us to characterize the sensitivity behavior of NLPs defined over subgraphs of infinite graphs. The theoretical developments are illustrated with numerical examples.
  
\end{abstract}



\section{Introduction}\label{sec:intro}
{Many decision-making problems in science and engineering are formulated as {\it structured} NLPs whose algebraic structures are induced by graphs. Examples include dynamic optimization (the graph is a time horizon; e.g., optimal control, long-term planning, and state estimation) \cite{betts2010practical,biegler2010nonlinear}, multi-stage stochastic programs (the graph is a scenario tree) \cite{shapiro2014lectures,pereira1991multi}, optimization problems constrained by discretized partial differential equations (the graph is a discretization mesh) \cite{biegler2003large}, and network optimization (the graph is a physical network; e.g., energy networks, supply chains)  \cite{dembo1989or,jalving2019graph,coffrin2018powermodels,zlotnik2015optimal}. Typical graphs associated with such problems are illustrated in Figure \ref{fig:graphs}. This observation motivates us to consider a {\it unifying abstraction} for structured optimization problems that we call {\it graph-structured nonlinear programs} (gsNLPs). The abstraction allows us to study seemingly different problem classes under a common perspective and allows us to establish fundamental properties they share. This paper studies the solution sensitivity of the gsNLPs.}

\begin{figure}[t]
  \centering
  \setlength{\tabcolsep}{0.25em}\footnotesize
  \begin{tabular}{@{}cccc@{}}
    Dynamic Optimization&Stochastic Optimization&PDE Optimization&Network Optimization\\[.05in]
    \includegraphics[width=.23\textwidth]{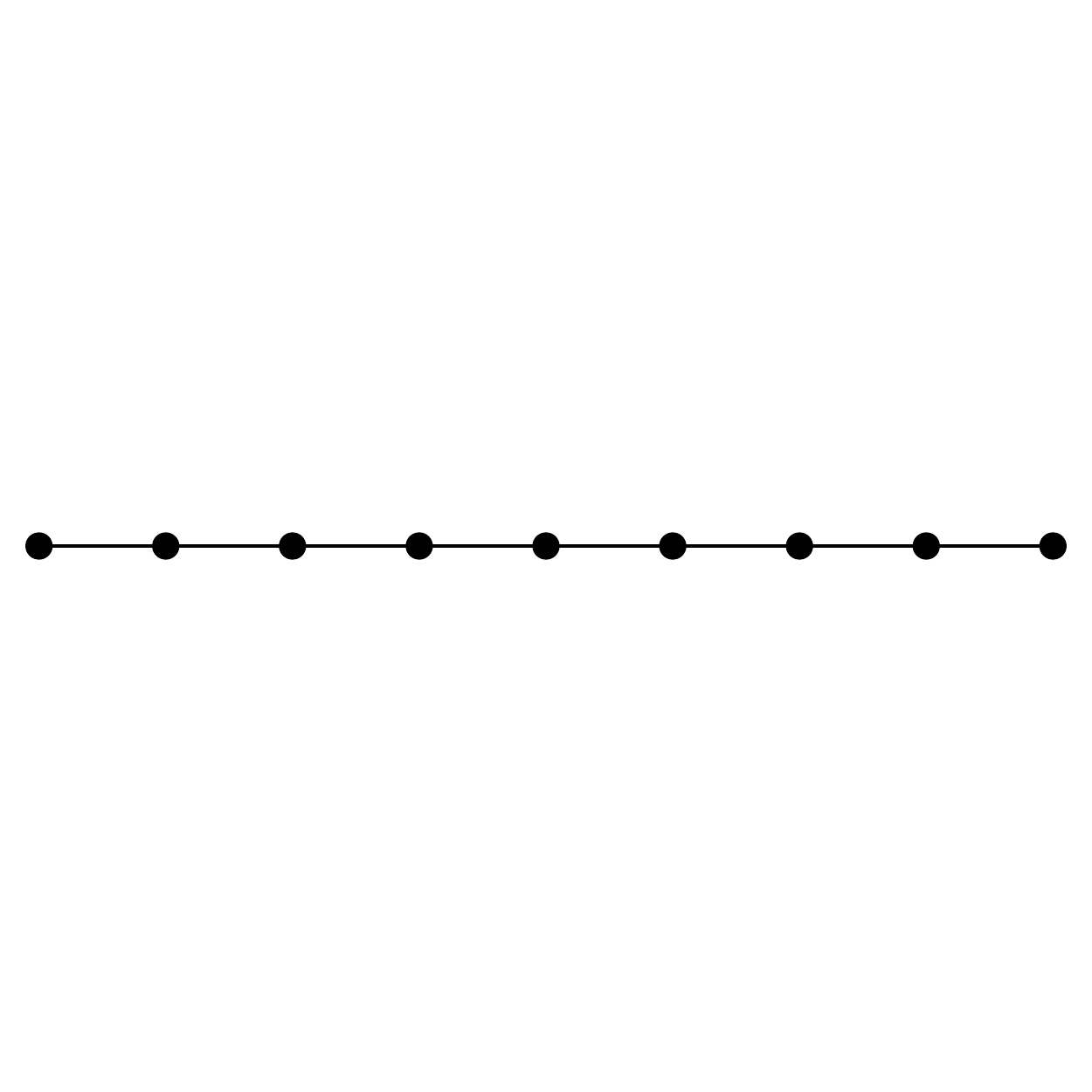}&\includegraphics[width=.23\textwidth]{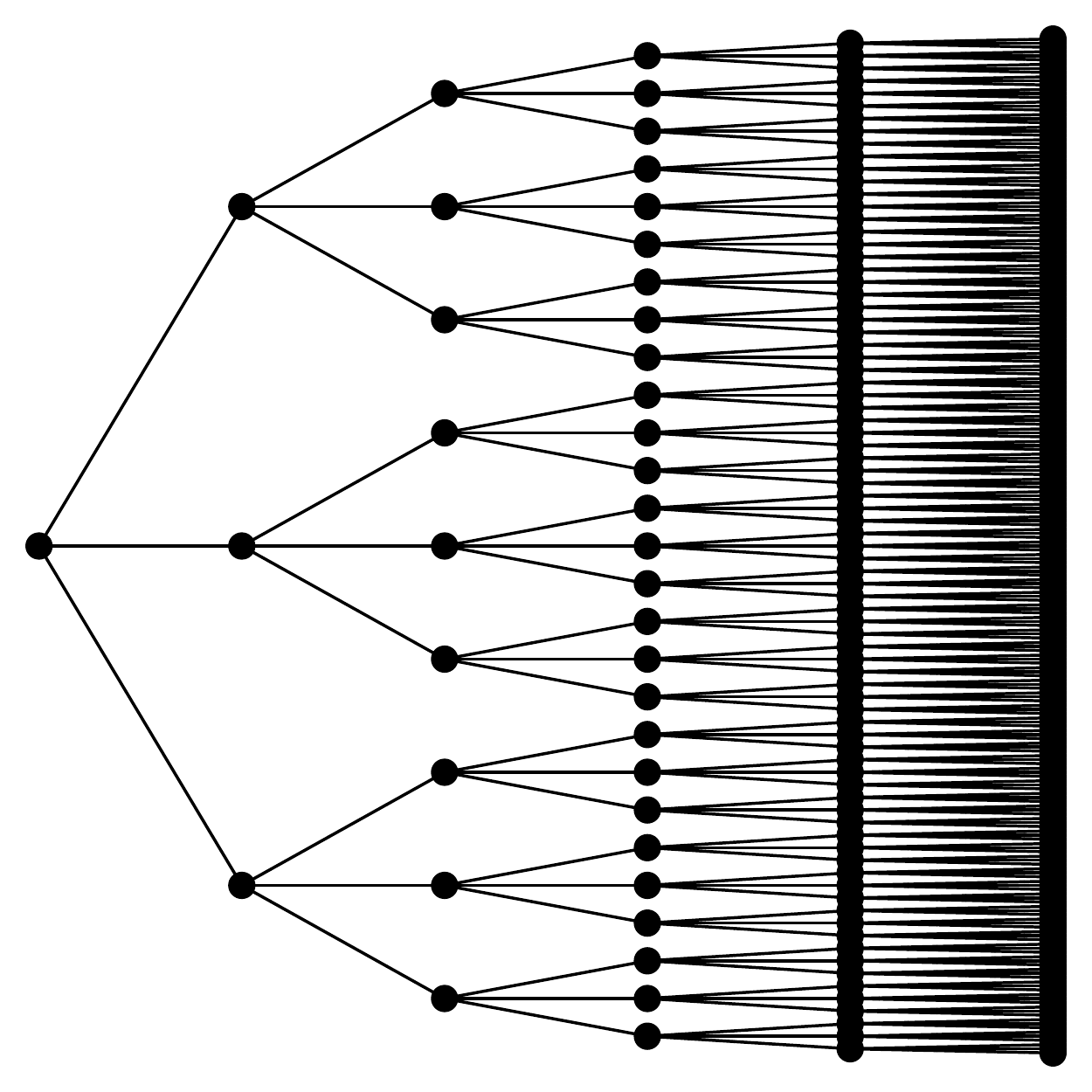}&\includegraphics[width=.23\textwidth]{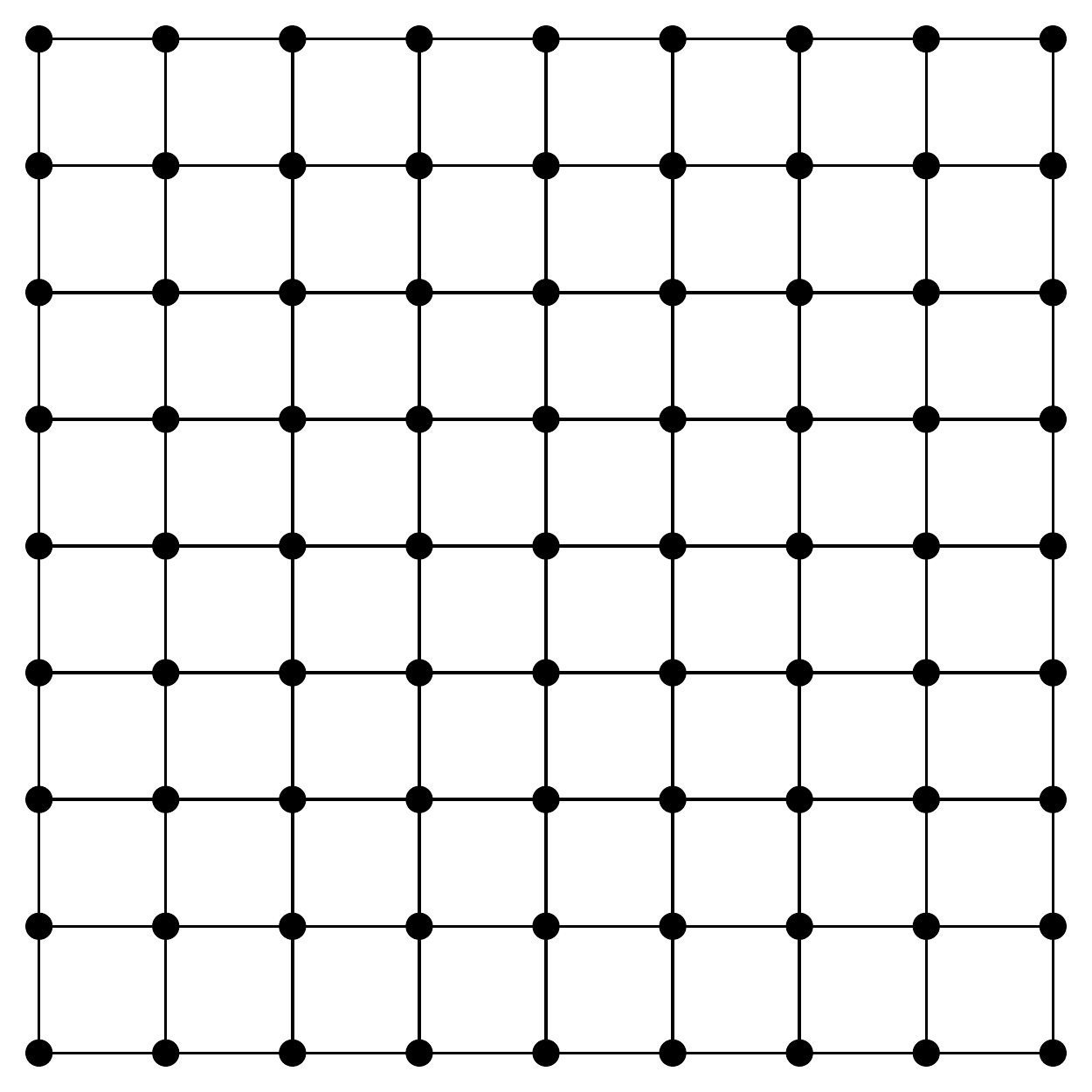}&\includegraphics[width=.23\textwidth]{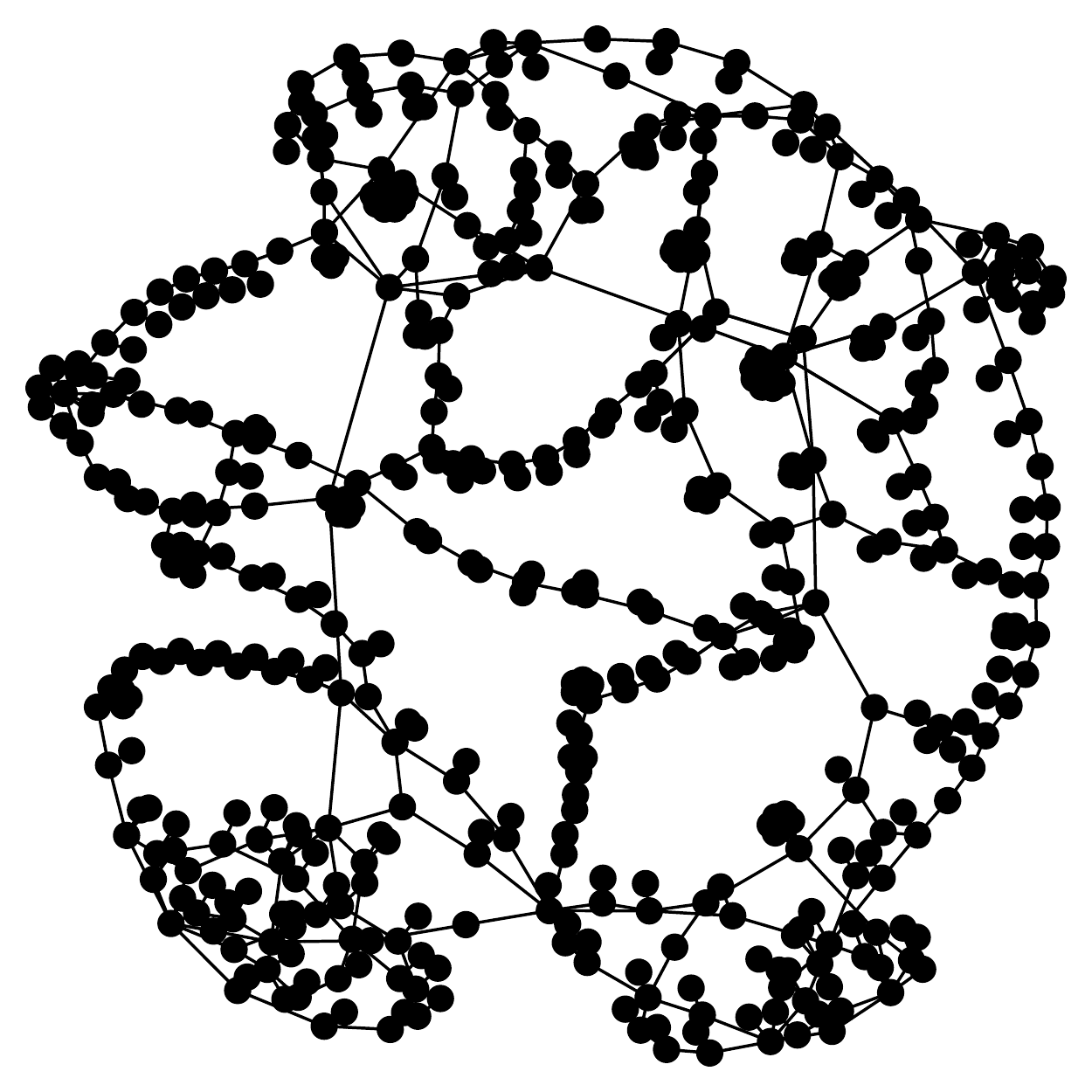}
  \end{tabular}
  \caption[Illustration of graphs associated with gsNLPs.]{Illustration of graphs associated with various graph-structured NLPs.} \label{fig:graphs}
\end{figure}

Our work is motivated by the following question:
\begin{align}\label{question}
  \begin{gathered}
    \text{\it How does the primal-dual solution at node $i\in \cV$ change}\\
    \text{\it when the data at node $j\in \cV$ is perturbed?}
  \end{gathered}
  \tag{Q1}
\end{align}
{\blue Here, $\cV$ is the set of nodes}. We provide an answer to \eqref{question} by identifying conditions under which there exist {\it nodal} sensitivity coefficients $\{C_{ij}\in\mathbb{R}_{\geq 0}\}_{i,j\in \cV}$ satisfying:
\begin{align}\label{eqn:sens}
  \Vert z^\dag_i(\bp)- z^\dag_i(\bp')\Vert \leq \sum_{j\in \cV} C_{ij} \Vert p_j - p'_j\Vert,\quad i\in \cV ,
\end{align}
where $z_i^\dag(\bp)$ and $z_i^\dag(\bp')$ are primal-dual solutions of the gsNLP \eqref{eqn:prob} at node $i\in\cV$ for data points $\bp$ and $\bp'$, respectively. Our main result (Theorem \ref{thm:main}) shows that if: (i) the strong second order sufficiency condition (SSOSC) and the linear independence constraint qualification (LICQ) hold at the reference data, and (ii) $\bp$ and $\bp'$ are sufficiently close to {\blue the reference data}, then \eqref{eqn:sens} holds with $C_{ij}= \Upsilon \rho^{d_\cG(i,j)}$. Here, $\Upsilon>0$ and $\rho\in(0,1)$ are constants, $d_\cG(i,j)$ is the graph distance between nodes $i$ and $j$ on $\cG$. In other words, {\em solution sensitivity decays exponentially} with respect to the distance to the perturbation point. We call this property {\it exponential decay of sensitivity} (EDS). This result is a specialization of classical sensitivity results \cite{robinson1980strongly,dontchev2009implicit,robinson1974perturbed,fiacco1976sensitivity,bonnans2013perturbation}.

The constants $(\Upsilon,\rho)$, the magnitude and decay rate of sensitivity, depend on the singular values of the Hessian of the Lagrangian. As such, we establish conditions under which the singular values remain uniformly bounded. This uniformity property will be particularly relevant when analyzing the sensitivity of NLPs that are defined over subgraphs of an infinite graph (a graph with an arbitrarily large domain). Such graphs can be used to analyze the limiting behavior of certain problem classes such as dynamic optimization problems over infinite horizons or of PDE optimization problems over unbounded domains. {We show that $(\Upsilon,\rho)$ remain uniform in the graph size under what we call {\it uniform regularity} conditions: {\it uniformly bounded} Lagrangian Hessian, {\it uniform} SSOSC, and {\it uniform} LICQ}. Unfortunately, these conditions are difficult to verify in practice because the problem becomes arbitrarily large. Accordingly, we establish {\it composable} sufficient conditions for uniform regularity that can be verified in practice; in particular, we show that {(i) uniformly bounded graph degree and second-order derivative functions imply uniformly bounded Lagrangian Hessian and (ii) {\it block}  SSOSC and LICQ conditions (assuming uniform SSOSC and LICQ hold over individual {\it blocks}) guarantee uniform SSOSC and LICQ.}

Question \eqref{question} has been recently addressed in specific settings such as linear and nonlinear dynamic optimization \cite{shin2019parallel,xu2018exponentially,Na2019Exponential,na2020overlapping} and graph-structured quadratic programs \cite{shin2020decentralized,shin2020overlapping}; our work generalizes such results. Addressing \eqref{question} is crucial for understanding solution stability {(in the sense of \cite{kojima1980strongly})}, for designing approximation schemes (often cast as parametric perturbations)  \cite{diehl2002real,biegler2013survey,xu2018exponentially,shin2020diffusing,na2020superconvergence}, and for designing solution algorithms \cite{shin2019parallel,shin2020overlapping,na2020overlapping}. For instance, it has been recently shown that EDS plays a central role in assessing the impact of coarsening schemes \cite{grune2020exponential,shin2020diffusing} for dynamic optimization and in establishing convergence of overlapping Schwarz algorithms for graph-structured problems \cite{shin2019parallel,shin2020overlapping,na2020overlapping}. From an application standpoint, our results seek to provide new insights on how perturbations propagate through graphs and on how the problem formulation influences such propagation. Specifically, we provide empirical evidence that positive objective curvature and constraint flexibility tend to dampen propagation (promote sensitivity decay). Such insights can be used, for instance, to design systems that dampen perturbations or to identify system elements that are sensitive to perturbations. 

The paper is organized as follows: In Section \ref{sec:settings} we introduce the notation and settings. In Section \ref{sec:slinear} we present basic results for {graph-structured matrix} properties. These results generalize the results in \cite{shin2020decentralized,demko1984decay} and will serve as an analytical tool for sensitivity analysis. In Section \ref{sec:sens} we present the main sensitivity results; specifically, we apply {graph-structured matrix} properties to classical NLP sensitivity theory to establish bounds on the nodal sensitivity coefficients in \eqref{eqn:sens}. {In Section \ref{sec:uniform}, we study uniform regularity conditions, which guarantee the uniform boundedness of the sensitivity coefficients}. Numerical results are provided in Section \ref{sec:num}, followed by conclusions in Section \ref{sec:conc}.

\section{Notation and Settings}\label{sec:settings}
{\it Notation:} The set of real numbers and the set of integers are denoted by $\mathbb{R}$ and $\mathbb{I}$, respectively. We define $\mathbb{I}_{A}:=\mathbb{I}\cap A$, where $A$ is a set; $\mathbb{I}_{>0}:=\mathbb{I}\cap(0,\infty)$; $\mathbb{I}_{\geq 0}:=\mathbb{I}\cap[0,\infty)$; $\mathbb{R}_{>0}:=(0,\infty)$; and $\mathbb{R}_{\geq 0}:=[0,\infty)$. For $A\subseteq X$, $f:X\rightarrow Y$, and $x\in X$, where $X$ and $Y$ are linear spaces, $A+x := \{x'+x:x'\in A\}$ and $f(X):=\{f(x)\in Y:x\in X\}$. Vectors are treated as column vectors. We use the syntax: $[M_1;\cdots;M_n]:=[M_1^\top\,\cdots\,M_n^\top]^\top$; $\{M_i\}_{i\in U}:=[M_{i_1}; \cdots; M_{i_{m}}]$; $\{M_{i,j}\}_{i\in U,j\in V}:=\{\{M_{i,j}^\top\}_{j\in V}^\top\}_{i\in U}$, where, $U=\{i_1<\cdots<i_{m}\}$ and $V=\{j_1<\cdots<j_{n}\}$ are strictly ordered sets. Furthermore, $v[i]$ is the $i$-th component of vector $v$; $M[i,j]$ is the $(i,j)$-th component of matrix $M$; $v[I]:=\{v[i]\}_{i\in I}$; $M[I,J]:=\{M[i,j]\}_{i\in I,j\in J}$. For a function $\phi:\mathbb{R}^{n}\rightarrow \mathbb{R}$ and variable vectors $y\in\mathbb{R}^{p}$, $z\in\mathbb{R}^{q}$, $\nabla^2_{yz}\phi(x):=\{\frac{\partial^2}{\partial y[i]\partial z[j]}\phi(x)\}_{i\in\mathbb{I}_{[1,p]},j\in\mathbb{I}_{[1,q]}}$. For a vector function $\varphi:\mathbb{R}^{n}\rightarrow \mathbb{R}^m$ and a variable vector $w\in\mathbb{R}^{s}$, $\nabla_{w}\varphi(x):=\{\frac{\partial}{\partial w[j]}\varphi(x)[i]\}_{i\in\mathbb{I}_{[1,m]},j\in\mathbb{I}_{[1,s]}}$. We use the shorthand notation $\nabla^2\phi(x):=\nabla^2_{xx}\phi(x)$, $\nabla\varphi(x):=\nabla_x\varphi(x)$. Vector 2-norms and induced 2-norms of matrices are denoted by $\Vert\cdot\Vert$. For matrices $A$ and $B$, $A\succ(\succeq) B$ indicates that $A-B$ is positive (semi)-definite while $A>(\geq) B$ denotes a componentwise inequality. The identity matrix is denoted as $\bI$ and the zero matrix or vector are denoted as $\bzero$. Specific notations will be introduced at the first appearance. {\blue We denote the largest and smallest eigenvalues of a symmetric matrix and non-trivial singular values of general matrices by $\olambda(\cdot)$, $\ulambda(\cdot)$, $\osigma(\cdot)$, and $\usigma(\cdot)$}.

{\it Settings:} We study gsNLPs of the following form: 
\begin{subequations}\label{eqn:prob}
  \begin{align}
    \min_{\{x_i\in\mathbb{R}^{r_i}\}_{i\in \cV}}\;& \sum_{i\in \cV}f_i(\{x_j\}_{j\in \cN_\cG[i]};\{p_j\}_{j\in \cN_\cG[i]})\label{eqn:prob-obj}\\
    \st\;& c^E_i(\{x_j\}_{j\in \cN_\cG[i]};\{p_j\}_{j\in \cN_\cG[i]}) = 0,\;i\in \cV,\;(y^E_i)\label{eqn:prob-eq}\\
                                                  &c^I_i(\{x_j\}_{j\in \cN_\cG[i]};\{p_j\}_{j\in \cN_\cG[i]}) \geq 0,\;i\in \cV,\;(y^I_i)\label{eqn:prob-ineq},
  \end{align}
\end{subequations}
where $\cG=(\cV,\cE)$ is an undirected graph, $\cV$ is a strictly ordered node set, $\cE\subseteq \{\{i,j\}\subseteq \cV: i\neq j\}$ is the edge set, and $\cN_\cG[i]:=\{j\in \cV:\{i,j\}\in \cE\}\cup\{i\}$ denotes the closed neighborhood\footnote{Note that closed neighborhood is different from the standard notion of neighborhood (also called open neighborhood), defined as $\cN_\cG(i):=\{j\in \cV:\{i,j\}\in \cE\}$. The closed neighborhood of a node contains itself whereas the open neighborhood does not.} of $i\in \cV$. For each node $i\in \cV$, $x_i\in\mathbb{R}^{r_i}$ is the primal variable; $p_i\in\mathbb{R}^{l_i}$ is the data; $f_i:\prod_{j\in \cN_\cG[i]}\mathbb{R}^{r_j}\times\prod_{j\in \cN_\cG[i]}\mathbb{R}^{l_j}\rightarrow \mathbb{R}$ is the objective function; $c^E_i:\prod_{j\in \cN_\cG[i]}\mathbb{R}^{r_j}\times \prod_{j\in \cN_\cG[i]}\mathbb{R}^{l_j}\rightarrow \mathbb{R}^{m^E_i}$ is the equality constraint function; $c^I_i:\prod_{j\in \cN_\cG[i]}\mathbb{R}^{r_j}\times \prod_{j\in \cN_\cG[i]}\mathbb{R}^{l_j}\rightarrow \mathbb{R}^{m^I_i}$ is the inequality constraint function; and $y^E_i\in\mathbb{R}^{m^E_i}$ and $y^I_i\in\mathbb{R}^{m^I_i}$ are the dual variables associated with \eqref{eqn:prob-eq} and \eqref{eqn:prob-ineq}. The objective and constraint functions of node $i$ depend on the variables $\{x_j\}_{j\in \cN_\cG[i]}$ and data $\{p_j\}_{j\in \cN_\cG[i]}$ of the neighboring nodes. That is, node $i$ is coupled algebraically only to its neighbors, and the topology of such connectivity is dictated by $\cG$.

{\blue
  We can express \eqref{eqn:prob} as a standard NLP of the following form:
 \begin{subequations}
   \begin{align}
     P(\bp):\;\min_{\bx}\;&\bff(\bx;\bp)\\
     \st\;&\bc^E(\bx;\bp)=\bzero,\; (\by^E)\\
                          &\bc^I(\bx;\bp)\geq \bzero,\;\;(\by^I).
   \end{align}
 \end{subequations}
 Here, we compact the notation as: $\bx:=\{x_i\}_{i\in \cV}$; $\by^E:=\{y^E_i\}_{i\in \cV}$; $\by^I:=\{y^I_i\}_{i\in \cV}$; $\bp:=\{p_i\}_{i\in \cV}$; $\br=\sum_{i\in \cV} r_i$; $\bm^E=\sum_{i\in \cV} m^E_i$; $\bm^I=\sum_{i\in \cV} m^I_i$; $\bl = \sum_{i\in \cV}l_i$; $\bff(\bx;\bp):=\sum_{i\in \cV}f_i(\{x_j\}_{j\in \cN_\cG[i]};\{p_j\}_{j\in \cN_\cG[i]})$; $\bc^E(\bx;\bp):=\{c^E_i(\{x_j\}_{j\in \cN_\cG[i]};\{p_j\}_{j\in \cN_\cG[i]})\}_{i\in \cV}$; and $\bc^I(\bx;\bp):=\{c^I_i(\{x_j\}_{j\in \cN_\cG[i]};\{p_j\}_{j\in \cN_\cG[i]})\}_{i\in \cV}$. We also define $c_i(\cdot):=[c^E_i(\cdot);c^I_i(\cdot)]$; $y_i:=[y^E_i;y^I_i]$; $z_i:=[x_i;y_i]$; $m_i=m^E_i+m^I_i$; $n_i=r_i+m_i$; $\by:=\{y_i\}_{i\in \cV}$; $\bz:=\{z_i\}_{i\in \cV}$;  $\bm=\sum_{i\in \cV} m_i$; $\bn=\sum_{i\in \cV} n_i$; and $\bc(\bx;\bp):=\{c_i(\{x_j\}_{j\in \cN_\cG[i]};\{p_j\}_{j\in \cN_\cG[i]})\}_{i\in \cV}$.
}

\begin{remark}\label{rem:gen-1}
  In some applications (e.g., energy networks), we encounter variables, data, objectives, and constraints associated with edges. Those are not explicitly expressed in  \eqref{eqn:prob}, but \eqref{eqn:prob} still can serve as an abstraction for such problems. In particular, such edge components can be captured within {\it super-nodes} that encapsulate edges; this is possible because the formulation allows for nodes with different numbers of variables and constraints. Alternatively, one may treat edges in the graph as nodes and rewrite the problem with a newly defined lifted graph $\widetilde{\cG}=(\widetilde{\cV},\widetilde{\cE})$, where $\widetilde{\cV}:=\cV\cup \cE$ and $\widetilde{\cE}:=\{\{i,e\}:i\in \mathcal{\cV},\;i\in e,\;e\in \mathcal{\cE}\}$ (an order needs to be assigned for $\widetilde{\cV}$).
\end{remark}

{
  \begin{remark}\label{rem:blackbox}
    Different graph structures can be imposed on the same NLP. For example, consider the following problem:
    \begin{subequations}\label{eqn:eg}
      \begin{align}
        \min_{x_1,x_2,x_3,x_4}\;&x_1^2+x_2^2+x_3^2+x_4^2\\
        \st\;& x_1+x_2 = 0,\quad(y_1)\label{eqn:eg-con-1}\\
        \;& x_2+x_3 = 0,\quad(y_2)\label{eqn:eg-con-2}\\
        \;& x_3+x_4 = 0,\quad(y_3).\label{eqn:eg-con-3}
      \end{align}
    \end{subequations}
    One can model the problem as a gsNLP with the graph with node set $\cV=\{1,2,3,4\}$ and edge set $\cE=\{\{1,2\},\{2,3\},\{3,4\}\}$; here, we assign variable $x_1,x_2,x_3,x_4$ into node $1,2,3,4$ and constraints \eqref{eqn:eg-con-1}, \eqref{eqn:eg-con-2}, \eqref{eqn:eg-con-3} into node $1,2,3$, respectively. Alternatively, one can use the graph with node set $\cV=\{1,2\}$ and edge set $\cE=\{\{1,2\}\}$, and assign variable $x_1,x_2$ into node $1$, variable $x_3,x_4$ into $2$, constraint \eqref{eqn:eg-con-1}, \eqref{eqn:eg-con-2} into node $1$, and  constraint \eqref{eqn:eg-con-1} into node $2$. Thus, there could be different gsNLP abstractions for the same NLP. Oftentimes, the graph comes from the modeler's insight into the problem. Obtaining modeler's insight can be straightforward in some applications (e.g., using time horizon as a graph for dynamic optimization problem), but it can be non-trivial when the problem is complex (e.g., dynamic/stochastic optimization over a network which also embeds PDEs). When a modeler's insight is not available, one can resort to formulate a gsNLP based on what we call a {\it primal-dual connectivity graph}. This graph can be obtained by viewing each primal-dual variable as a node. Then, we construct an edge set by using the connectivity pattern given by the problem. For example, the primal-dual connectivity graph for \eqref{eqn:eg} is $\cG=(\{1,2,\cdots,7\},\{\{1,5\},\{2,5\},\{2,6\},\{3,6\},\{3,7\},\{4,7\}\})$, where $x_1,\cdots x_4$ corresponds to node $1,\cdots,4$ and the dual $y_1,\cdots,y_3$ correspond to node $5,\cdots,7$, respectively. In this way, we can impose a graph structure on any NLP without a modeler's insight. However, in terms of interpretability, using modeler's insight may be beneficial. Moreover, the choice of the graph may have a strong impact to the usefulness of the subsequent analysis. For instance, as our results explore exponential decay with graph distance, choosing a graph with a sufficiently large diameter (the largest graph distance between a pair of nodes) facilitates the applicability of our results. 
  \end{remark}
}

\section{Graph-Structured Matrix Properties}\label{sec:slinear}
{This section derives basic properties of graph-structured matrices. A {\it graph-structured matrix} is a matrix that has a graph that imposes a structure to the matrix; in other words, the matrix does not arise from a graph, but the graph imposes a structure to the matrix. The results in this section will be crucial in establishing our main sensitivity results. Properties of graph-structured {\it positive definite} matrices are reported in \cite{shin2020decentralized,demko1984decay}; here, we establish properties for general (non-symmetric and indefinite) matrices.} We begin by introducing the notion of distance on graphs and establishing its basic properties.
\begin{definition}[Graph Distance and Diameter]\label{def:distance}
  The distance $d_\cG(i,j)$ between nodes $i,j\in \cV$ on graph $\cG=(\cV,\cE)$ is the number of edges in the shortest path connecting them; {if a shortest path does not exist, the distance is $+\infty$}. Furthermore, the diameter $D_\cG$ of $\cG$ is the longest distance between any pair of nodes in $\cV$.
\end{definition}
\begin{proposition}\label{prop:metric}
  The distance $d_\cG:\cV\times \cV \rightarrow \mathbb{I}_{\geq 0}{\cup\{\infty\}}$ is a metric on $\cV$; that is, (a) $d_\cG(i,j)\geq 0$ for any $i,j\in \cV$; (b) $i=j$ if and only if $d_\cG(i,j)=0$; (c) $d_\cG(i,j)=d_\cG(j,i)$ for any $i,j\in \cV$; (d) $d_\cG(i,j)\leq d_\cG(i,k)+d_\cG(k,j)$ for any $i,j,k\in \cV$.
\end{proposition}
The proof of this result is straightforward and thus omitted. {Note that the shortest path does not exist when the graph is not connected. } 

{A graph-structured matrix $X\in\mathbb{R}^{m\times n}$ is indexed by a graph $\cG=(\cV,\cE)$ and a pair of index set families $\cI:=\{I_{i}\}_{i\in \cV}$ and $\cJ:=\{J_{i}\}_{i\in \cV}$ that partition\footnote{In this paper, we call a family $\{X_1,\cdots,X_k\}$ of subsets of $X$ to be a partition if $\bigcup_{k=1}^K X_k=X$ and $X_1,\cdots,X_k$ are disjoint; here, {\it we allow $X_k$ to be empty sets}. Note that this differs from the standard definition of a partition, where the nonemptiness of $X_k$ is enforced. This modification allows us to handle nodes with empty variables, constraints, or data, which is useful in many settings.} $\mathbb{I}_{[1,m]}$ and $\mathbb{I}_{[1,n]}$, respectively. We refer to $X_{[i][j]}$ as the $[i][j]$-block of matrix $X$, which can later be associated with the node pair $(i,j)$. It should be noted that the short-hand notation $X_{[i][j]}:=X[I_i,J_j]$ assumes that the proper graph and index set families $(\cG,\cI,\cJ)$ are introduced or can be inferred from the context. We are now ready to introduce the concept of graph-induced {\it bandwidth}.} 

\begin{definition}[Graph-Induced Bandwidth]\label{def:bandwidth}
  Consider a matrix $X\in\mathbb{R}^{m\times n}$, a graph $\cG=(\cV,\cE)$, and index sets $\cI=\{I_i\}_{i\in \cV}$, $\cJ=\{J_i\}_{i\in \cV}$ that partition $\mathbb{I}_{[1,m]}$ and $\mathbb{I}_{[1,n]}$, respectively. Matrix $X$ is said to have bandwidth $B$ induced by $(\cG,\cI,\cJ)$, if $B$ is the smallest nonnegative integer such that $X_{[i][j]}=\bzero$ for any $i,j\in \cV$ with $d_\cG(i,j)> B$, where $X_{[i][j]}:=X[I_i,J_j]$.
\end{definition}
The graph-induced bandwidth defined above is a generalization of the standard notion of matrix bandwidth \cite[Section 1.2.1]{golub2012matrix}. If the matrix $X\in\mathbb{R}^{n\times n}$ is square, $\cV=\mathbb{I}_{[1,n]}$, $\cE=\{\{i,i+1\}\}_{i=1}^{n-1}$, and $\cI=\cJ=\{\{i\}\}_{i=1}^n$, then the bandwidth induced by $(\cG,\cI,\cJ)$ reduces to the standard definition of matrix bandwidth. 

{
  \begin{example}\label{eg:graph}
    Consider the following sparse matrix:
    \begin{align*}
      M=
      \begin{bmatrix}
        1&0&0&0&0&0&0&0\\
        -1&-1&1&0&0&0&0&0\\
        0&0&-1&-1&1&0&0&0\\
        0&0&0&0&1&-1&1&0
      \end{bmatrix}.
    \end{align*}
    The triple $(\cG,\cI,\cJ)$ is: $\cG=(\cV,\cE)$, $\cV:=\{1,2,3,4\}$, $\cE:=\{\{1,2\},\{2,3\},\{3,4\}\}$, $\cI:=\{\{1\},\{2\},\{3\},\{4\}\}$, and $\cJ:=\{\{1,2\},\{3,4\},\{5,6\},\{7,8\}\}$. The matrix $M$ can be written in the following block matrix form:
    \begin{align*}
      M=
      \begin{bmatrix}
        M_{[1][1]}&M_{[1][2]}&M_{[1][3]}&M_{[1][4]}\\
        M_{[2][1]}&M_{[2][2]}&M_{[2][3]}&M_{[2][4]}\\
        M_{[3][1]}&M_{[3][2]}&M_{[3][3]}&M_{[3][4]}\\
        M_{[4][1]}&M_{[4][2]}&M_{[4][3]}&M_{[4][4]}
      \end{bmatrix},
    \end{align*}
    where $M_{[i][j]}\in\mathbb{R}^{1\times 2}$ for each $i,j\in \cV$. Also, we observe that $M_{[1][3]}$, $M_{[1][4]}$, $M_{[2][4]}$, $M_{[3][1]}$, $M_{[4][1]}$, $M_{[4][2]}=0$; thus, the matrix bandwidth is $1$.
  \end{example}
}

Definition \ref{def:bandwidth} enables defining {\it {graph-induced banded} matrix}, a graph-structured matrix whose bandwidth is much smaller than the diameter of the structure. This corresponds to the standard notion of banded matrix, {which is defined as a matrix with a bandwidth much smaller than the size of the matrix}.

We now state the basic properties of the matrix bandwidth.

\begin{lemma}\label{lem:bandwidth}
  Consider $X\in\mathbb{R}^{m\times n}$ with bandwidth not greater than $B_X$ induced by $(\cG,\cI,\cJ)$; we have that:
  (a) $X^\top$ has bandwidth not greater than $B_X$ induced by $(\cG,\cJ,\cI)$;
  (b) if $Y\in\mathbb{R}^{m\times n}$ has bandwidth not greater than $B_Y$ induced by $(\cG,\cI,\cJ)$, then $X+Y$ has bandwidth not greater than $\max(B_X,B_Y)$ induced by $(\cG,\cI,\cJ)$;
  (c) if $W\in\mathbb{R}^{n\times k}$ has bandwidth not greater than $B_W$ induced by $(\cG,\cJ,\cK)$, then $XW$ has bandwidth not greater than $B_X+B_W$ induced by $(\cG,\cI,\cK)$.
\end{lemma}
\begin{proof}[Proof of (a)] We have that $(X^\top)_{[i][j]}=(X^\top){[J_i,I_j]}=(X[I_j,J_i])^\top=(X_{[j][i]})^\top$. From the assumption that $X$ has bandwidth not greater than $B_X$ and Proposition \ref{prop:metric}(c), $(X^\top)_{[i][j]}=\bzero$ if $d_\cG(i,j)>B_X$; therefore, $X$ has bandwidth not greater than $B_X$, and induced by $(\cG,\cJ,\cI)$.
\end{proof}
\begin{proof}[Proof of (b)] We have that $X_{[i][j]}=\bzero$ and $Y_{[i][j]}=\bzero$ if $d_\cG(i,j)>\max(B_X,B_Y)$, which yields $(X+Y)_{[i][j]}=\bzero$ if $d_\cG(i,j)>\max(B_X,B_Y)$. Thus, $X+Y$ has bandwidth not greater than $\max(B_X,B_Y)$, and induced by $(\cG,\cI,\cJ)$.
\end{proof}
\begin{proof}[Proof of (c)] If $d_\cG(i,j)>B_X+B_W$, from Proposition \ref{prop:metric}(d) we have that for any $k\in \cV$, $d_\cG(i,k)>B_X$ or $d_\cG(j,k)>B_W$. Thus, if $d_\cG(i,j)>B_X+B_W$,  we have $(XW)_{[i][j]}=\sum_{k\in \cV} X_{[i][k]}W_{[k][j]}= \bzero$ (where the first equality comes from the block matrix multiplication law and the second equality comes from observing that $d_\cG(i,k)>B_X$ or $d_\cG(j,k)>B_W$). Therefore, $XW$ has bandwidth not greater than $B_X+B_W$, and induced by $(\cG,\cI,\cK)$.  
\end{proof}

Lemma \ref{lem:bandwidth} implies that {graph-induced bandwidth} is preserved under transposition, addition, and multiplication, as long as the associated index sets are compatible. Using Lemma \ref{lem:bandwidth}, we can establish the main result of this section, the properties of the inverse of a {graph-induced banded} matrix. 

\begin{theorem}\label{thm:inv}
  Consider a nonsingular matrix $X\in\mathbb{R}^{n\times n}$ with bandwidth not greater than $B_X\geq 1$ induced by $(\cG,\cK,\cP)$, $Y\in\mathbb{R}^{n\times m}$ with bandwidth not greater than $B_Y$ induced by $(\cG, \cK,\cJ)$, and $W\in\mathbb{R}^{\ell\times n}$ with bandwidth not greater than $B_W$ induced by $(\cG, \cI,\cP)$; for constants $\osigma_X\geq \osigma(X)$, $\osigma_Y\geq \osigma(Y)$, $\osigma_W\geq \osigma(W)$, and $0<\usigma_X\leq \usigma(X)$, the following holds:
  \begin{align}\label{eqn:inv}
    \Vert (WX^{-1}Y)_{[i][j]}\Vert \leq
    \frac{\osigma_X\osigma_Y\osigma_W}{\usigma_X^2}\left(\frac{\osigma_X^2-\usigma_X^2}{\osigma_X^2+\usigma_X^2}\right)^{\left\lceil \frac{d_\cG(i,j)-B_X-B_Y-B_W}{2B_X}\right\rceil_+} ,\quad i,j\in \cV,
  \end{align}
  where $(WX^{-1}Y)_{[i][j]}:=(WX^{-1}Y)[I_i,J_j]$ and $\lceil\cdot\rceil_+$ is the smallest non-negative integer that is greater than or equal to the argument.
\end{theorem}
\begin{proof}
  By definition of singular values and the assumptions on $\osigma_X$, $\osigma_Y$ , $\osigma_W$, and $\usigma_X$, we have $\usigma_X^2 \bI \preceq\usigma(X)^2 \bI \preceq X X^\top\preceq  \osigma(X)^2\bI\preceq\osigma_X^2 \bI $. From this, one can obtain:
 \begin{align}\label{eqn:factor}
   \frac{\usigma_X^2-\osigma_X^2}{\usigma_X^2+\osigma_X^2} \bI\preceq \bI-\frac{2}{\usigma_X^2+\osigma_X^2}XX^\top
   \preceq \frac{-\usigma_X^2+\osigma_X^2}{\usigma_X^2+\osigma_X^2}\bI.
 \end{align}
 By Lemma \ref{lem:bandwidth}(c), $XX^\top$ has bandwidth not greater than $2B_X$ induced by $(\cG,\cK,\cK)$. From the nonsingularity of $XX^\top$, we obtain the following:
 \begin{subequations}\label{eqn:key-1}
   \begin{align}
     WX^{-1}Y &= \frac{2}{\usigma_X^2+\osigma_X^2} WX^\top \left(\frac{2}{\usigma_X^2+\osigma_X^2}XX^\top \right)^{-1}Y\\
     &= \frac{2}{\usigma_X^2+\osigma_X^2}WX^\top \left(\sum_{q=0}^\infty \left(\bI-\frac{2}{\usigma_X^2+\osigma_X^2}XX^\top \right)^q\right)Y,\label{eqn:key-1-c}\\
     &=\frac{2}{\usigma_X^2+\osigma_X^2} \sum_{q=0}^\infty WX^\top\left(\bI-\frac{2}{\usigma_X^2+\osigma_X^2}XX^\top \right)^q Y.
 \end{align}
 \end{subequations}
 Here the second equality follows from \cite[Corollary 5.6.16]{horn2012matrix} and \eqref{eqn:factor}, and the last equality follows from the fact that the series in \eqref{eqn:key-1-c} converges (due to \eqref{eqn:factor}). Furthermore, from Lemma \ref{lem:bandwidth}, we see that $WX^\top\left(\bI-\frac{2}{\usigma_X^2+\osigma_X^2}XX^\top \right)^q Y$ has bandwidth not greater than $(2q+1)B_X +B_Y+B_W$, induced by $(\cG,\cI,\cJ)$. By extracting submatrices $(\cdot)[I_i,J_j]$ from \eqref{eqn:key-1}, we obtain:
  \begin{align*}
    (WX^{-1}Y)_{[i][j]}&=\frac{2}{\usigma_X^2+\osigma_X^2} \sum_{q=q_0(i,j)}^\infty \left(WX^\top\left(\bI-\frac{2XX^\top}{\usigma_X^2+\osigma_X^2} \right)^q Y\right)_{[i][j]}
  \end{align*}
  where $q_0(i,j):=\left\lceil \frac{d_\cG(i,j)-B_X-B_Y-B_W}{2B_X}\right\rceil_+$; the summation over $q=0,\cdots, q_0(i,j)-1$ is zero because such $q$ satisfy $(2q+1)B_X+B_Y+B_W< d_\cG(i,j)$. Using the triangle inequality and the fact that the matrix norm of a submatrix is smaller than that of the original matrix,
  \begin{align}\label{eqn:Hbound}
    \Vert (WX^{-1}Y)_{[i][j]} \Vert &\leq \frac{2}{\usigma_X^2+\osigma_X^2} \sum_{q=q_0(i,j)}^\infty \left\Vert WX^\top\left(\bI-\frac{2XX^\top}{\usigma_X^2+\osigma_X^2}\right)^q Y\right\Vert\\
    &\leq \frac{2}{\usigma_X^2+\osigma_X^2} \sum_{q=q_0(i,j)}^\infty \osigma_W\osigma_X\left(\frac{\osigma_X^2-\usigma_X^2}{\osigma_X^2+\usigma_X^2}\right)^q \osigma_Y\nonumber\\
    &\leq \frac{\osigma_X\osigma_Y\osigma_W}{\usigma_X^2}\left(\frac{\osigma_X^2-\usigma_X^2}{\osigma_X^2+\usigma_X^2}\right)^{\left\lceil \frac{d_\cG(i,j)-B_X-B_Y-B_W}{2B_X}\right\rceil_+} .\nonumber
  \end{align}
  The second inequality follows from the submultiplicativity of the matrix norm and \eqref{eqn:factor}; and the last inequality follows from the summation of geometric series. 
\end{proof}

The result indicates that the norm of the $[i][j]$-block of $WX^{-1}Y$ {\it decays exponentially} with respect to the distance between nodes $i$ and $j$ on $\cG$; the decay rate becomes faster (smaller) as the condition number $\osigma(X)/\usigma(X)$ of $X$ decreases; and the decay rate becomes faster as the bandwidths $B_X$, $B_Y$, and $B_W$ decrease. This property will be key in establishing EDS for the gsNLPs \eqref{eqn:prob} and hints at the fact that EDS arises from the connectivity of the graph (at the linear algebra level).

Theorem \ref{thm:inv} is a generalization of \cite[Theorem 1]{shin2020decentralized}, which assumes positive definiteness of $X$ and $Y=W=\bI$. We also note that \cite{demko1984decay} has studied exponential decay of the components of the inverse of graph-induced banded matrices. Specifically, in \cite[Theorem 2.4]{demko1984decay}, exponential decay for indefinite graph-induced banded matrices (with the standard definition of bandwidth) is established. Furthermore, \cite[Proposition 5.1]{demko1984decay} presents a less general form of Theorem \ref{thm:inv}, which lacks the notion of {graph-structured} matrix and assumes positive definiteness. Theorem \ref{thm:inv} generalizes these results by introducing the notion of {graph-structured} matrices and by considering non-symmetric, indefinite matrices.

\section{Exponential Decay of Sensitivity}\label{sec:sens}
This section aims to provide an answer to question \eqref{question}. {To be specific, we show that the nodal sensitivity decays exponentially with respect to the distance to the perturbation.} The sketch of our analysis is as follows: We first invoke classical results of NLP sensitivity theory \cite{robinson1980strongly,dontchev2009implicit} to obtain an expression for the one-sided directional derivative of the primal-dual solution mapping with respect to the data; this representation involves the inverse of a {graph-induced banded} matrix. {Theorem \ref{thm:inv} is then applied to establish bounds on the nodal one-sided directional derivatives. Using the continuity of singular values, one can extend this upper bound result to a neighborhood of the given base solution and data.} Finally, the one-sided directional derivative is integrated over the line segment between a pair of data points that are within the neighborhood of the base data to obtain the result in the form of \eqref{eqn:sens}. This yields explicit expressions for $(\Upsilon,\rho)$. {To facilitate the application of the result to a wide range of problem classes, it is necessary to express the constants $(\Upsilon,\rho)$ with problem-independent parameters. Further characterization of $(\Upsilon,\rho)$ is separately presented in the subsequent section.}
 
\subsection{Preliminaries}\label{sec:prelim}
We use $\bz^\star \in\mathbb{R}^{\bn}$ to denote the primal-dual solution of $P(\bp^\star)$. We refer to $\bp^\star$ as the base data and $\bz^\star$ as the base solution. We denote the primal and dual components of $\bz^\star=\{z^\star_i\}_{i\in\cV}=\{[x^\star_i;y^\star_i]\}_{i\in\cV}$ as $\bx^\star:=\{x^\star_i\}_{i\in\cV}$ and $\by^\star=\{y^\star_i\}_{i\in\cV}$, respectively. We say that $\bz^\star$ is a primal-dual solution if $\bx^\star$ satisfies the first-order optimality conditions with Lagrange multiplier $\by^\star$ (see \cite{nocedal2006numerical}). We now make key assumptions that are necessary to establish our main results. 

\begin{assumption}[Twice Continuous Differentiability of Functions]\label{ass:c2}
  The functions $\bff:\mathbb{R}^{\br}\times \mathbb{R}^{\bl}\rightarrow \mathbb{R}$ and $\bc:\mathbb{R}^{\br}\times \mathbb{R}^{\bl}\rightarrow \mathbb{R}^{\bm}$ are twice continuously differentiable in the neighborhood of $[\bx^\star;\bp^\star]$.
\end{assumption}

\begin{assumption}[Regularity of Solution]\label{ass:regularity}
  The base solution $\bz^\star$ satisfies SSOSC and LICQ for $P(\bp^\star)$ with base data $\bp^\star$.
\end{assumption}
We recall that SSOSC requires positive definiteness of the reduced Hessian of the Lagrangian at $\bz^\star$. The reduced Hessian is the Hessian projected on the null space defined by equality constraints and active inequality constraints with nonzero duals. LICQ requires that the constraint Jacobian defined by equality and active inequality constraints are linearly independent at $\bz^\star$. These requirements are stated as: 
\begin{align*}
  ReH(\nabla^2_{\bx\bx}\cL(\bz^\star;\bp^\star),\nabla_{\bx}\bc(\bx^\star;\bp^\star)[\cA^0(\bp^\star),:]) &\succ \bzero\tag{SSOSC}\\
  \usigma(\nabla_{\bx}\bc(\bx^\star;\bp^\star)[\cA^1(\bp^\star),:]) &> 0 \tag{LICQ}.
\end{align*}
Here, $\cL(\bz;\bp) := \bff(\bx;\bp)-\by^\top \bc(\bx;\bp)$ is the Lagrange function, $ReH(H,A):=Z^\top HZ$ is the reduced Hessian, where $Z$ is a null-space matrix of $A$, and
\begin{align}\label{eqn:active-set-defs}
  \cA^0(\bp^\star)&:=\cA^E\cup\{i\in\cA^I:\bc(\bx^\star)[i]=0,\;\by^\star[i]\neq 0\}\\
  \cA^1(\bp^\star)&:=\cA^E\cup\{i\in\cA^I:\bc(\bx^\star)[i]=0\},
\end{align}
where $\cA^E$ and $\cA^I$ are the set of equality and inequality constraint indices. Hereafter, we call a solution of gsNLP regular if the solution satisfies SSOSC and LICQ. Note that SSOSC and LICQ are standard assumptions used in NLP sensitivity theory \cite{robinson1980strongly}. We now state the classical sensitivity result.

\begin{lemma}\label{lem:nlp-sensitivity}
  Under Assumptions \ref{ass:c2}, \ref{ass:regularity}, there exist neighborhoods $\mathbb{P}\subseteq\mathbb{R}^{\bl}$ of $\bp^\star$ and $\bbZ\subseteq\mathbb{R}^{\bn}$ of $\bz^\star$ and a continuous function $\bz^\dag:\mathbb{P}\rightarrow \bbZ$ such that $\bz^\dag(\bp)$ is a primal-dual solution of $P(\bp)$ that satisfies SSOSC and LICQ. Furthermore, for any $\bq:=\{q_i\}_{i\in \cV}\in\mathbb{R}^{\bl}$, the one-sided directional derivative of $\bz^\dag(\cdot)$ is given by:
  \begin{align*}
    D_{\bq} \bz^\dag(\bp)&:=\lim_{h\searrow 0}\frac{\bz^\dag(\bp+\bq h)-\bz^\dag(\bp)}{h}; \label{eqn:one-sided-DD}
  \end{align*}
  with $D_{\bq}\bx^\dag(\bp) = \bxi^\dag(\bp,\bq)$, $D_{\bq}\by^\dag(\bp,\bq) = \bbeta^\dag(\bp,\bq)$, where $(\bxi^\dag(\bp,\bq),\bbeta^\dag(\bp,\bq)[\cA^1(\bp)])$ is a primal-dual solution of the quadratic program:
  \begin{subequations}\label{eqn:nlp-sensitivity-0}
    \begin{align}
      QP(\bp,\bq):\quad\min_{\bxi}\;&\frac{1}{2}\bxi^\top\bQ(\bp)\bxi + \bxi^\top\bS(\bp)\bq,\label{eqn:nlp-sensitivity-0-obj}\\
      \st\;&\bA(\bp)[i,:]^\top\bxi + \bB(\bp)[i,:]\bq = 0,\;i\in \cA^0(\bp),\; (\bbeta[i])\\
                   &\bA(\bp)[i,:]^\top\bxi + \bB(\bp)[i,:]\bq \geq 0,\;i\in \cA^1(\bp)\setminus \cA^0(\bp),\; (\bbeta[i])
    \end{align}      
  \end{subequations}
  and $\bbeta^\dag(\bp,\bq)[\mathbb{I}_{[1,\bm]}\setminus\cA^1(\bp)]=\bzero$ (i.e., the free dual variables are fixed to zero); here,
  \begin{subequations}\label{eqn:AAQA}
    \begin{align}
      \cA^0(\bp)&:=\cA^E\cup\{i\in\cA^I:\bc(\bx^\dag(\bp))[i]=0,\;\by^\dag(\bp)[i]\neq 0\}\\
      \cA^1(\bp)&:=\cA^E\cup\{i\in\cA^I:\bc(\bx^\dag(\bp))[i]=0\}\\
      \bQ(\bp)&:=\nabla_{\bx\bx}\cL(\bz^\dag(\bp);\bp);\quad \bS(\bp):=\nabla_{\bx\bp}\cL(\bz^\dag(\bp);\bp)\\
      \bA(\bp)&:=\nabla_{\bx}\bc(\bx^\dag(\bp);\bp);\quad \bB(\bp):=\nabla_{\bp}\bc(\bx^\dag(\bp);\bp).
    \end{align}
  \end{subequations}
  Moreover, a unique primal-dual solution of $QP(\bp,\bq)$ exists for any $\bp\in\mathbb{P}$ and $\bq\in\mathbb{R}^{\bl}$; thus, $\bxi^\dag(\cdot,\cdot)$ and $\bbeta^\dag(\cdot,\cdot)$ are well-defined.
\end{lemma}
{Note that Lemma \ref{lem:nlp-sensitivity} is a well-known result in the NLP sensitivity literature. We provide a short proof below based on the results reported in \cite{dontchev2009implicit}.}
\begin{proof}
  The results in \cite[Theorem 2G.8]{dontchev2009implicit} ensure semidifferentiability (which guarantees continuity) of the solution of the first-order optimality conditions for $P(\bp)$ over a certain neighborhood of $\bp^\star$. Furthermore, \cite[Theorem 2G.9]{dontchev2009implicit} establishes that over a certain neighborhood $\mathbb{P}$, the solution mapping of the generalized equation (GE) satisfies SSOSC and LICQ; that is, within $\mathbb{P}$, the solution mapping for the GE is the solution mapping for $P(\bp)$ at which SSOSC and LICQ are satisfied. Moreover, by \cite[Theorem 2G.8]{dontchev2009implicit}, the one-sided directional derivative of the solution mapping for the GE (which exists for any direction $\bq\in\mathbb{R}^{\bl}$ by semidifferentiability \cite[Theorem 2D.1]{dontchev2009implicit}) can be evaluated by using the linearized GE. The linear GE corresponds to the first-order optimality conditions for $QP(\bp,\bq)$ (see \cite[Equation (35)]{dontchev2009implicit}); here, the first-order conditions are necessary and sufficient conditions for the optimality due to the convexity $QP(\bp,\bq)$ (guaranteed by SSOSC and LICQ of $\bz^\star$ for the original problem). As such, the one-sided directional derivative of the solution mapping for $P(\bp)$ can be evaluated by solving $QP(\bp,\bq)$. Finally, the strong regularity of the GE at $\bz^\star$ guarantees that there exists a unique solution of the linearized GE, which in turn guarantees the existence of a unique solution of $QP(\bp,\bq)$.  
\end{proof}

Under Lemma \ref{lem:nlp-sensitivity}, the rate of change $D_{\bq}\bz^\dag(\bp)$ of the primal-dual solution of $P(\bp)$ (for a given direction $\bq$) can be quantified by using the solution of $QP(\bp,\bq)$. For given $\bp$ and $\bq$, the parameters in $QP(\bp,\bq)$ can be evaluated explicitly and thus its solution can be calculated. As such, Lemma \ref{lem:nlp-sensitivity} provides a computational procedure to evaluate primal-dual solution sensitivity.

We now further analyze $QP(\bp,\bq)$; its first-order conditions are as follows:
\begin{subequations}\label{eqn:kkt}
  \begin{align}
    \bQ(\bp)\bxi + \bS(\bp)\bq + \bA(\bp)[\cA^{1}(\bp),:]^\top \bbeta &= \bzero\label{eqn:kkt-du}\\
    \bA(\bp)[i,:]\bxi + \bB(\bp)[i,:]\bq &= 0,\; i\in\cA^0(\bp)\label{eqn:kkt-pr-1}\\
    \bA(\bp)[i,:]\bxi + \bB(\bp)[i,:]\bq &\geq 0,\;i\in\cA^1(\bp)\setminus\cA^0(\bp)\label{eqn:kkt-pr-2}\\
    \bbeta[i] &\geq 0,\;i\in\cA^1(\bp)\setminus\cA^0(\bp)\\
    \left(\bA(\bp)[i,:]\bxi + \bB(\bp)[i,:]\bq \right)\bbeta[i]&=0,\; i\in\cA^1(\bp)\setminus\cA^0(\bp).\label{eqn:kkt-compl}
  \end{align}
\end{subequations}
Under SSOSC and LICQ for $P(\bp^\star)$ at $\bz^\star$, these conditions are necessary and sufficient for optimality. From the complementarity condition \eqref{eqn:kkt-compl}, one can observe that there exists $\cA^0(\bp)\subseteq\cA'(\bp,\bq)\subseteq \cA^1(\bp)$ such that:
\begin{subequations}\label{eqn:compl-explicit}
\begin{align}
  \bA(\bp)[i,:]\bxi + \bB(\bp)[i,:]\bq &= 0,\; i\in \cA'(\bp,\bq)\\
  \bbeta[i] &= 0,\; i\in \cA^1(\bp)\setminus \cA'(\bp,\bq).
\end{align}
\end{subequations}
are satisfied at $\bxi^\dag(\bp,\bq)$, $\bbeta^\dag(\bp,\bq)$. Thus, from \eqref{eqn:kkt}, \eqref{eqn:compl-explicit}, and $\bbeta[i]=0$ for $i\in\mathbb{I}_{[1,\bm]}\setminus \cA^1(\bp)$ (from Lemma \ref{lem:nlp-sensitivity}), we have:
\begin{align}\label{eqn:linear-form-0}
  \begin{aligned}
  \begin{bmatrix}
    \bQ(\bp)& \bA(\bp)[\cA'(\bp,\bq),:]^\top\\
    \bA(\bp)[\cA'(\bp,\bq),:]
  \end{bmatrix}
  \begin{bmatrix}
    \bxi\\
    \bbeta[\cA'(\bp,\bq)]
  \end{bmatrix} &= -
  \begin{bmatrix}
    \bS(\bp)\\  
    \bB(\bp)[\cA'(\bp,\bq),:]
  \end{bmatrix}\bq\\
  \bbeta [\mathbb{I}_{[1,\bm]}\setminus \cA'(\bp,\bq)] &= \bzero.
\end{aligned}
\end{align}
By rearranging $[\bxi^\dag(\bp,\bq);\bbeta^\dag(\bp,\bq)]$ one can obtain $\bzeta^\dag(\bp)=\{[\xi_i^\dag(\bp,\bq);\eta_i^\dag(\bp,\bq)]\}_{i\in\cV}$, where $\bxi^\dag(\bp,\bq)=\{\xi_i^\dag(\bp,\bq)\}_{i\in\cV}$, $\bbeta^\dag(\bp,\bq)=\{\eta_i^\dag(\bp,\bq)\}_{i\in\cV}$. To perform such rearrangement, we consider a permutation $\phi:\mathbb{I}_{[1,\bn]}\rightarrow\mathbb{I}_{[1,\bn]}$ that achieves $\bzeta[\phi(i)] = [\bxi;\bbeta][i]$. This permutation enables the following definition:
\begin{subequations}\label{eqn:Bs}
\begin{align}
  \cB^0(\bp)&:=\phi(\mathbb{I}_{[1,\br]}\cup(\cA^0(\bp)+\br)),\quad
                      \cB^1(\bp):=\phi(\mathbb{I}_{[1,\br]}\cup(\cA^1(\bp)+\br))\\
  \cB'(\bp,\bq)&:=\phi(\mathbb{I}_{[1,\br]}\cup(\cA'(\bp,\bq)+\br)).
\end{align}
\end{subequations}
Finally, using the notation in \eqref{eqn:Bs}, we obtain the following from \eqref{eqn:linear-form-0}:
\begin{subequations}\label{eqn:linear-form}
\begin{align}
  D_{\bq}\bz^\dag(\bp)[\cB'(\bp,\bq)] &= -\left(\bH(\bp)[\cB'(\bp,\bq),\cB'(\bp,\bq)]\right)^{-1}\bR(\bp)[\cB'(\bp,\bq),:] \bq,\\
  D_{\bq}\bz^\dag(\bp)[\mathbb{I}_{[1,\bn]}\setminus\cB'(\bp,\bq)] &= \bzero,\label{eqn:linear-form-b}
\end{align}
\end{subequations}
where:
\begin{align}\label{eqn:HR}
  \bH(\bp) &:= \nabla_{\bz\bz}\cL(\bz^\dag(\bp);\bp);\quad \bR(\bp) := \nabla_{\bz\bp}\cL(\bz^\dag(\bp);\bp).
\end{align}
Here, the nonsingularity of $\bH(\bp)[\cB'(\bp,\bq),\cB'(\bp,\bq)]$
follows from LICQ, SSOSC, and \cite[Lemma 16.1]{nocedal2006numerical}.

\subsection{Nodal Sensitivity Result}\label{sec:nodal}
{We now establish the main nodal sensitivity result. First, we consider $\bH(\bp)$ and $\bR(\bp)$ as graph-structured matrices.} Specifically, we introduce the index set families $\cI:=\{I_i\}_{i\in\cV}$ and $\cK:=\{K_i\}_{i\in\cV}$with:
\begin{align}\label{eqn:ind-1}
  I_i &:= \mathbb{I}_{\sum_{j\in \cV,j<i}n_j+[1,n_i]},\quad
  K_i := \mathbb{I}_{\sum_{j\in \cV,j<i}l_j+[1,l_i]},
\end{align}
and impose graph structures to $\bH(\bp)$ and $\bR(\bp)$ with $(\cG,\cI,\cI)$ and $(\cG,\cI,\cK)$, respectively. Note that $\cI$ and $\cK$ partition $\mathbb{I}_{[1,\bn]}$ and $\mathbb{I}_{[1,\bl]}$, respectively. For $i,j\in\cV$,
\begin{subequations}\label{eqn:HRij}
\begin{align}
  H_{ij}(\bp)&:=\nabla^2_{z_iz_j}\cL(\bz^\dag(\bp);\bp)=(\bH(\bp))_{[i][j]}\\
  R_{ij}(\bp)&:=\nabla^2_{z_ip_j}\cL(\bz^\dag(\bp);\bp)=(\bR(\bp))_{[i][j]}.
\end{align}
\end{subequations}
{We now can observe that $\bH(\bp)$ and $\bR(\bp)$ have bandwidth not greater than two. In other words, if $d_\mathcal{G}(i,j)> 2$, $\nabla_{z_i z_j}\mathcal{L}(\bz;\bp)=0$ and $\nabla_{z_i p_j}\mathcal{L}(\bz;\bp)=0$. This is because in the graph-structured nonlinear programming setting in (1.1), we have assumed that the nodal objective and constraint functions are only dependent on the variables that are within the closed neighborhood of the associated node.}

The submatrices of $\bH(\bp)$ and $\bR(\bp)$ are also graph-structured. In particular, $\bH(\bp)[\cB,\cB]$ and $\bR(\bp)[\cB,:]$ with $\cB\subseteq \mathbb{I}_{[1,\bn]}$ have bandwidth not greater than two induced by $(\cG,\cI^\cB,\cI^\cB)$ and $(\cG,\cI^\cB,\cK)$, where $\cI^\cB:=\{I_i\cap\cB\}_{i\in\cV}$. We thus see that \eqref{eqn:linear-form} involves the inverse of a {graph-induced banded} matrix. By combining Theorem \ref{thm:inv} and Lemma \ref{lem:nlp-sensitivity}, we obtain the following result.

\begin{lemma}\label{lem:dder}
  Suppose Assumptions \ref{ass:c2}, \ref{ass:regularity} hold, and suppose that, for given $\bp\in\mathbb{P}$ (defined in Lemma \ref{lem:nlp-sensitivity}) and $\bq:=\{q_i\}_{i\in \cV}\in\mathbb{R}^{\bl}$, we have $\osigma_{\bH}(\bp,\bq)$, $\osigma_{\bR}(\bp,\bq)$, and $\usigma_{\bH}(\bp,\bq)$ such that:
  \begin{subequations}\label{eqn:dder-0}
  \begin{align}
    \osigma_{\bH}(\bp,\bq)&\geq \osigma(\bH(\bp)[\cB'(\bp,\bq),\cB'(\bp,\bq)])\\
    \osigma_{\bR}(\bp,\bq)&\geq \osigma(\bR(\bp)[\cB'(\bp,\bq),:])\\
    0<\usigma_{\bH}(\bp,\bq)&\leq \usigma(\bH(\bp)[\cB'(\bp,\bq),\cB'(\bp,\bq)]);\label{eqn:dder-0-c}
  \end{align}
\end{subequations}
then the following holds for any $i\in \cV$:
  \begin{align*}
    \Vert D_{\bq} z^\dag_i(\bp,\bq)\Vert\leq \sum_{j\in \cV}\frac{\osigma_{\bH}(\bp,\bq)\osigma_{\bR}(\bp,\bq)}{\usigma_{\bH}(\bp,\bq)^2}\left(\frac{\osigma_{\bH}(\bp,\bq)^2-\usigma_{\bH}(\bp,\bq)^2}{\osigma_{\bH}(\bp,\bq)^2+\usigma_{\bH}(\bp,\bq)^2}\right)^{\left\lceil \frac{d_\cG(i,j)}{4}-1\right\rceil_+}\Vert q_j\Vert.
  \end{align*}
\end{lemma}
\begin{proof}
  {\blue For simplicity, we denote $\cB'(\bp,\bq)$, $\osigma_{\bH}(\bp,\bq)$, $\osigma_{\bR}(\bp,\bq)$, $\usigma_{\bH}(\bp,\bq)$ by $\cB'$, $\osigma_{\bH}$, $\osigma_{\bR}$, $\usigma_{\bH}$}. From that $\bH(\bp)[\cB',\cB']$ is always nonsingular, we have that $\usigma_{\bH}(\bp,\bq)$ satisfying \eqref{eqn:dder-0-c} always exists. By inspecting the block structure of \eqref{eqn:linear-form} we obtain:
  \begin{align}\label{eqn:nlp-sensitivity-1}
    D_{\bq}\bz^\dag(\bp)[I_i\cap\cB'] &= \sum_{j\in \cV} -((\bH(\bp)[\cB',\cB'])^{-1}\bR(\bp)[\cB',:])_{[i][j]} q_j,
  \end{align}
  where $\cI^{\cB'}:=\{I_i\cap{\cB'}\}_{i\in \cV}$ and $\cK:=\{K_i\}_{i\in \cV}$ (defined in \eqref{eqn:ind-1}) are used for index sets. We have already established that $\bH(\bp)[\cB',\cB']$ has bandwidth not greater than two, induced by $(\cG,\cI^{\cB'},\cI^{\cB'})$ and that $\bR(\bp)[\cB',:]$ has bandwidth not greater than two, induced by $(\cG,\cI^{\cB'},\cK)$. By applying Theorem \ref{thm:inv}, we obtain:
  \begin{align}\label{eqn:nlp-sensitivity-2}\blue
    \Vert((\bH(\bp)[\cB',\cB'])^{-1}\bR(\bp)[\cB',:])_{[i][j]}\Vert\leq \frac{\osigma_{\bH}\osigma_{\bR}}{\usigma_{\bH}^2}\left(\frac{\osigma_{\bH}^2-\usigma_{\bH}^2}{\osigma_{\bH}^2+\usigma_{\bH}^2}\right)^{\left\lceil \frac{d_\cG(i,j)}{4}-1\right\rceil_+} .
  \end{align}
  Now note that
  $\Vert D_{\bq}z^\dag_i(\bp)\Vert \leq \Vert D_{\bq}\bz^\dag(\bp)[I_i\cap\cB'] \Vert  + \Vert D_{\bq}\bz^\dag(\bp)[I_i\setminus\cB']\Vert$,
  and recall from \eqref{eqn:linear-form-b} that $D_{\bq}\bz^\dag(\bp)[I_i\setminus\cB'] =\bzero$. Hence, by applying \eqref{eqn:nlp-sensitivity-2} to \eqref{eqn:nlp-sensitivity-1}, and applying triangle inequality, we obtain the desired result.
\end{proof}

Lemma \ref{lem:dder} establishes that the dependence of $\Vert D_{\bq}z_i^\dag(\bp)\Vert$ on  perturbation $p_j$ decays with $d_\cG(i,j)$. However, the right-hand side of \eqref{eqn:nlp-sensitivity-2} still depends on $\bp,\bq$. To express this result as in \eqref{eqn:sens}, wherein the constants on the right-hand side are independent of $\bp,\bq$, we exploit the continuity of singular values \cite[Corollary 8.6.2]{golub2012matrix}. This gives us the main result of the paper. 

\begin{theorem}[Exponential Decay of Sensitivity (EDS)]\label{thm:main}
  Under Assumptions \ref{ass:c2}, \ref{ass:regularity}, and for given $\epsilon>0$, $\osigma_{\bH}\geq\oosigma_{\bH}(\bp^\star)$, $\osigma_{\bR}\geq\oosigma_{\bR}(\bp^\star)$, and $0<\usigma_{\bH}\leq\uusigma_{\bH}(\bp^\star)$, where
  \begin{subequations}\label{eqn:sigmas}
    \begin{align}
      \oosigma_{\bH}(\bp)&:=\max\{\osigma(\bH(\bp)[\cB,\cB]):\cB^0(\bp^\star)\subseteq\cB\subseteq\cB^1(\bp^\star)\}\\
      \oosigma_{\bR}(\bp)&:=\max\{\osigma(\bR(\bp)[\cB,:]):\cB^0(\bp^\star)\subseteq\cB\subseteq\cB^1(\bp^\star)\}\\
      \uusigma_{\bH}(\bp)&:=\min\{\usigma(\bH(\bp)[\cB,\cB]):\cB^0(\bp^\star)\subseteq\cB\subseteq\cB^1(\bp^\star)\},
    \end{align}
  \end{subequations}
  there exists a neighborhood $\mathbb{P}_{\epsilon}$ of $\bp^\star$ such that the following holds for any $\bp,\bp'\in\mathbb{P}_{\epsilon}$:
  \begin{align}\label{eqn:main}
      &\Vert z^\dag_i(\bp)-z^\dag_i(\bp')\Vert 
      \leq \sum_{j\in \cV} \Upsilon\rho^{\left\lceil \frac{d_\cG(i,j)}{4}-1\right\rceil_+} \Vert p_j-p'_j\Vert,\quad i\in \cV,
  \end{align}
with $\Upsilon :=\dfrac{\osigma_{\bH}\osigma_{\bR}}{{\usigma}_{\bH}^2}+\epsilon$ and $\rho:=\dfrac{\osigma_{\bH}^2-{\usigma}_{\bH}^2}{\osigma_{\bH}^2+{\usigma}_{\bH}^2}+\epsilon$.
\end{theorem}
\begin{proof}
  From the continuity of $\bz^\dag(\cdot)$ in the neighborhood of $\bp^\star$ and the continuity of $\bc(\cdot,\cdot)$ in the neighborhood of $[\bz^\star;\bp^\star]$, there exists a neighborhood $\widetilde{\mathbb{P}}\subseteq\mathbb{P}$ of $\bp^\star$ such that, for $\bp\in\widetilde{\mathbb{P}}$ and $i\in\mathbb{I}_{[1,\bm]}$, $\bc(\bx^\star)[i] > 0\Rightarrow\bc(\bx^\dag(\bp))[i] > 0$ and $\by^\star[i] \neq 0\Rightarrow\by^\dag(\bp)[i] \neq 0$. These conditions and complementarity slackness imply that for $\bp\in\widetilde{\mathbb{P}}$, we have $\cA^0(\bp^\star)\subseteq\cA^0(\bp)$ and $\cA^1(\bp)\subseteq\cA^1(\bp^\star)$; that is, $\cB^0(\bp^\star)\subseteq\cB^0(\bp)$ and $\cB^1(\bp)\subseteq\cB^1(\bp^\star)$. This and $\cB^0(\bp)\subseteq\cB'(\bp,\bq)\subseteq\cB^1(\bp)$ yields:
  \begin{subequations}\label{eqn:main-0}
  \begin{align}
    \osigma({\bH}(\bp)[\cB'(\bp,\bq),\cB'(\bp,\bq)])&\leq \oosigma_{\bH}(\bp)\\
    \osigma({\bR}(\bp)[\cB'(\bp,\bq),:])&\leq \oosigma_{\bR}(\bp)\\
    \usigma({\bH}(\bp)[\cB'(\bp,\bq),\cB'(\bp,\bq)])&\geq \uusigma_{\bH}(\bp).
  \end{align}
\end{subequations}
  By the twice-continuous differentiability of $\cL(\cdot,\cdot)$ and the continuity of $\bz^\dag(\cdot)$, we have that $\bH(\cdot)$ is continuous. The same holds true for its submatrices: $\bH(\cdot)[\cB,\cB]$ with $\cB^0(\bp^\star)\subseteq \cB\subseteq \cB^1(\bp^\star)$. From the continuity of singular values with respect to its entries \cite[Corollary 8.6.2]{golub2012matrix}, we have that $\osigma(\bH(\cdot)[\cB,\cB])$ and $\usigma(\bH(\cdot)[\cB,\cB])$ are continuous for any $\cB^0(\bp^\star)\subseteq \cB\subseteq \cB^1(\bp^\star)$; accordingly, since a maximum and a minimum of a fixed and finite number of continuous functions is continuous, we have that $\oosigma_{\bH}(\bp)$, $\oosigma_{\bR}(\bp)$, $\uusigma_{\bH}(\bp)$ are continuous with respect to $\bp$ in $\widetilde{\mathbb{P}}$. Thus, there exists a convex neighborhood $\mathbb{P}_{\epsilon}\subseteq\widetilde{\mathbb{P}}$ of $\bp^\star$ wherein the following are satisfied:
  \begin{subequations}\label{eqn:main-1}
  \begin{align}
    \frac{\oosigma_{\bH}(\bp)\oosigma_{\bR}(\bp)}{\uusigma_{\bH}(\bp)^2} &\leq
    \frac{\oosigma_{\bH}(\bp^\star)\oosigma_{\bR}(\bp^\star)}{\uusigma_{\bH}(\bp^\star)^2}+\epsilon\\
    \frac{\oosigma_{\bH}(\bp)^2-\uusigma_{\bH}(\bp)^2}{\oosigma_{\bH}(\bp)^2+\uusigma_{\bH}(\bp)^2}&\leq
    \frac{\oosigma_{\bH}(\bp^\star)^2-\uusigma_{\bH}(\bp^\star)^2}{\oosigma_{\bH}(\bp^\star)^2+\uusigma_{\bH}(\bp^\star)^2}+\epsilon.
  \end{align}
\end{subequations}
Here, note that $\uusigma_{\bH}(\bp^\star)>0$ because $\usigma(\bH[\cB,\cB])>0$ holds for any $\cB^0(\bp^\star)\subseteq\cB\subseteq\cB^1(\bp^\star)$. By applying \eqref{eqn:main-0} and \eqref{eqn:main-1} to Lemma \ref{lem:dder}, we obtain:
\begin{align}\label{eqn:main-2}
  \Vert D_{\bq} z^\dag_i(\bp)\Vert&\leq \sum_{j\in \cV}\left(\frac{\oosigma_{\bH}(\bp^\star)\oosigma_{\bR}(\bp^\star)}{\uusigma_{\bH}(\bp^\star)^2}+\epsilon\right)\left(\frac{\oosigma_{\bH}(\bp^\star)^2-\uusigma_{\bH}(\bp^\star)^2}{\oosigma_{\bH}(\bp^\star)^2+\uusigma_{\bH}(\bp^\star)^2}+\epsilon\right)^{\left\lceil \frac{d_\cG(i,j)}{4}-1\right\rceil_+}\Vert q_j\Vert\nonumber\\
  &\leq \sum_{j\in \cV}\left(\frac{\osigma_{\bH}\osigma_{\bR}}{\usigma_{\bH}^2}+\epsilon\right)\left(\frac{\osigma_{\bH}^2-\usigma_{\bH}^2}{\osigma_{\bH}^2+\usigma_{\bH}^2}+\epsilon\right)^{\left\lceil \frac{d_\cG(i,j)}{4}-1\right\rceil_+}\Vert q_j\Vert\nonumber\\
  &\leq \sum_{j\in \cV}\Upsilon\rho^{\left\lceil \frac{d_\cG(i,j)}{4}-1\right\rceil_+}\Vert q_j\Vert,
\end{align}
for any $\bp\in\mathbb{P}_{\epsilon}$ and $\bq\in\mathbb{R}^{\bl}$. Finally, since we have chosen $\mathbb{P}_\epsilon$ to be convex, for any $\bp,\bp'\in\mathbb{P}_{\epsilon}$, the line segment between $\bp,\bp'$ is within $\mathbb{P}_{\epsilon}$. Thus, we have:
\begin{align*}
  \Vert z^\dag_i(\bp)-z^\dag_i(\bp')\Vert
    &\leq \int_{0}^1\Vert D_{\bp'-\bp}z^\dag_i ((1-t)\bp+t\bp')\Vert dt \leq \sum_{j\in \cV}\Upsilon\rho^{\left\lceil \frac{d_\cG(i,j)}{4}-1\right\rceil_+}\Vert p_j-p'_j\Vert,
\end{align*}
where the first inequality is from Newton-Leibniz and triangle inequality for integrals, and the last inequality follows from \eqref{eqn:main-2}.
\end{proof}

One can observe that $\Upsilon>0$ and $\rho\in(0,1)$ hold; consequently, we have that the upper bound of the nodal sensitivity decays exponentially as $d_\cG(i,j)$ increases. We can also see that $(\Upsilon,\rho)$ depend on the singular values of the submatrices of $\bH(\bp^\star)$, $\bR(\bp^\star)$. {Specifically, it is necessary to have a reasonably good conditioning of the $\bH(\bp^\star)[\cB,\cB]$ matrices satisfying $\cB^0(\bp^\star)\subseteq\cB\subseteq\cB^1(\bp^\star)$ in order to guarantee sufficiently fast decay. In the next section, we discuss under which conditions we can guarantee good conditioning in the $\bH(\bp^\star)[\cB,\cB]$ matrices.}
{\blue Finally, after a suitable redefinition ($\Upsilon\leftarrow \Upsilon\rho^{-1}$ and $\rho\leftarrow \rho^{1/4}$), one can obtain the result claimed in Section \ref{sec:intro}; that is, \eqref{eqn:sens} holds with the sensitivity coefficients $\{C_{ij}=\Upsilon\rho^{ d_\cG(i,j)}\}_{i,j\in \cV}$. }

\begin{remark}\label{rem:gen-2}
  One can establish EDS for a more general version of NLP \eqref{eqn:prob}, in which coupling is allowed within the expanded neighborhood $N^B_\cG[i]:=\{j\in \cV: d_\cG(i,j)\leq B\}$ with $B>1$. Such an NLP arises when the algebraic coupling between nodes extends beyond immediate neighbors. For this more general setting, the corresponding results for Theorem \ref{thm:main} can be established; in particular, if the rest of the assumptions in Theorem \ref{thm:main} remain the same, the following holds:
  \begin{align}
  \Vert z^\dag_i(\bp)-z^\dag_i(\bp')\Vert 
    \leq &\sum_{j\in \cV} \Upsilon\rho^{\left\lceil \frac{d_\cG(i,j)}{4B}-1\right\rceil_+} \Vert p_j-p'_j\Vert.
  \end{align}
  We can observe that the exponential decay becomes slower as $B$ increases. 
\end{remark}

\section{Uniform Regularity Conditions}\label{sec:uniform}
An interesting class of graph-structured nonlinear programs is that in which the underlying graph is a subgraph of an infinite-dimensional graph. Examples include time-dependent problems (in which we might want to extend the horizon) and discretized PDE optimization (in which we might want to expand the domain). To analyze this setting, we consider a family of problems $\{P_{(k)}(\cdot)\}_{k\in K}$ with a potentially infinite problem index set $K$. The associated quantities are introduced accordingly; $\{\cG_{(k)}=(\cV_{(k)},\cE_{(k)})\}_{k\in K}$, $\{\bff_{(k)}(\cdot)\}_{k\in K}$, $\{\bc_{(k)}(\cdot)\}_{k\in K}$. Also, a set of data $\{\bp^\star_{(k)}\}_{k\in K}$ and the associated base solutions $\{\bz^\star_{(k)}\}_{k\in K}$ are considered. The submatrices $\{\bH_{(k)}(\bp^\star)\}_{k\in K}$, $\{\bR_{(k)}(\bp^\star)\}_{k\in K}$ of the full Hessian matrix can be defined as in \eqref{eqn:HR} for each $k$. 

This section aims to establish sufficient conditions for:
\begin{align}\label{eqn:uniform-bound}
  \sup_{k\in K} \oosigma_{\bH,(k)}(\bp^\star_{(k)})&<+\infty;\quad \sup_{k\in K}\oosigma_{\bR,(k)}(\bp^\star_{(k)})< +\infty;\quad \inf_{k\in K}\uusigma_{\bH,(k)}(\bp^\star_{(k)})>0,
\end{align}
where $\oosigma_{\bH,(k)}(\bp^\star_{(k)})$, $\oosigma_{\bR,(k)}(\bp^\star_{(k)})$, and $\uusigma_{\bH,(k)}(\bp^\star_{(k)})$ are defined in Theorem \ref{thm:main}, but the problem index $k$ is added. One can observe that, if \eqref{eqn:uniform-bound} is violated, $\Upsilon_{(k)}$ may become indefinitely large and $\rho_{(k)}$ may approach one (thus making the bounds derived in Theorem \ref{thm:main} not particularly useful). Hence, ensuring \eqref{eqn:uniform-bound} is crucial for guaranteeing a moderately bounded sensitivity magnitude $\Upsilon_{(k)}$ and a fast sensitivity decay rate $\rho_{(k)}$. 

We call \eqref{eqn:uniform-bound} {\it uniform} boundedness conditions; furthermore, we call a quantity to be {\it uniform} in $k$ if the quantity is independent of the problem index $k$. Note that \eqref{eqn:uniform-bound} holds trivially if $K$ is finite and Assumptions \ref{ass:c2}, \ref{ass:regularity} hold for each $k\in K$. However, even if $K$ is finite, it is necessary for Theorem \ref{thm:main} to be practically useful that $\inf_{k\in K}\uusigma_{\bH,(k)}(\bp^\star_{(k)})$ is sufficiently bounded away from zero and that $\sup_{k\in K} \oosigma_{\bH,(k)}(\bp^\star_{(k)})$ and $\sup_{k\in K}\oosigma_{\bR,(k)}(\bp^\star_{(k)})$ are bounded above by a moderately large number. As such, the results in this section provide useful information even if $K$ is finite (and even if $K$ is a singleton). Hereafter, to reduce the notational burden, we will drop the notation for the dependency on $k$ as well as $\bp$ (e.g., $\bH\leftarrow\bH_{(k)}(\bp^\star)$). {This is because (i) even if we assume that there might be multiple problems $P_k(\cdot)$, we do not need to specify the problem index $k$ each time because we study the condition that applies uniformly to each problem; and (ii) $\bp$ is fixed to $\bp^\star$ for the rest of the discussion in this section (note that Theorem \ref{thm:main} only requires upper/lower bounds of the singular values evaluated at a particular pair of base solution and data).}

\subsection{Sufficient Conditions for Uniform Boundedness}\label{sec:uniform-sub}
We now state key {\it uniform regularity assumptions} that are necessary for establishing uniform boundedness conditions \eqref{eqn:uniform-bound}. {The uniform regularity conditions consist of uniformly bounded Lagrangian Hessian (uBLH), uniform strong second-order sufficiency condition (uSSOSC), and uniform linear independence constraint qualification (uLICQ)}. These assumptions provide basic uniform parameters from which we can establish explicit bounds for the quantities in \eqref{eqn:uniform-bound}.

\begin{assumption}[$L$-Uniformly Bounded Lagrangian Hessian]\label{ass:ublh}
  There exists $L\geq 0$ such that the following holds: $\|\bH\|,\|\bR\|\leq L$; here, $\bH$ and $\bR$ are defined in \eqref{eqn:HR}.
\end{assumption}
\begin{assumption}[$\gamma$-Uniform SSOSC]\label{ass:ussosc}
  There exists $\gamma>0$ such that:\\ $ReH(\bQ,\bA^0)\succeq \gamma \bI$, where $\bA^0:=\bA[\mathcal{A}^0,:]$; here, $\mathcal{A}^0$, $\bQ$, and $\bA$ are defined in \eqref{eqn:AAQA}.
\end{assumption}
\begin{assumption}[$\beta$-Uniform LICQ]\label{ass:ulicq}
  There exists $\beta>0$ such that: $\bA^1 (\bA^1)^\top \succeq \beta \bI$, where $\bA^1:=\bA[\mathcal{A}^1,:]$; here, $\mathcal{A}^1$ and $\bA$ are defined in \eqref{eqn:AAQA}.
\end{assumption}
Assumption \ref{ass:ublh} is an extension of Assumption \ref{ass:c2}; it assumes that the second order derivative not only exists, but also is uniformly bounded. Assumption \ref{ass:ussosc}, \ref{ass:ulicq} are extensions of Assumption \ref{ass:regularity}; the assumptions are strengthened by introducing additional uniform parameters, $\gamma,\beta> 0$. With parameters $L,\gamma,\beta$, we can establish the uniform bounds in \eqref{eqn:uniform-bound}. First, the upper bounds in \eqref{eqn:uniform-bound} are trivially obtained.

\begin{lemma}\label{lem:trivial}
  Under Assumption \ref{ass:ublh}, $\oosigma_{\bH},\oosigma_{\bR}\leq L$ ($\oosigma_{\bH},\oosigma_{\bR}$ are defined in \eqref{eqn:sigmas}).
\end{lemma}
\begin{proof}
  Follows from $\bH[\cB,\cB]$, $\bR[\cB,:]$ being submatrices of $\bH,\bR$.
\end{proof}

To establish a lower bound for $\uusigma_{\bH}$, we prove the following lemma.
\begin{lemma}\label{lem:uniform-2}
  Under Assumption \ref{ass:ublh}, \ref{ass:ussosc}, \ref{ass:ulicq} and $\overline{\mu}:=(2L^2/\gamma + \gamma+  L)/\beta$,
  \begin{align}\label{eqn:uniform-2}
    \uusigma_{\bH}\geq
    \left(\frac{2}{\gamma} +  \frac{8\overline{\mu} L^2}{\gamma^3\beta} +\frac{4L}{\gamma^2\beta}\right)^{-1}\left(1+\overline{\mu} L\right)^{-1}.
  \end{align}
\end{lemma}
We first prove that $\bQ_{}+\overline{\mu}(\bA^0_{})^\top \bA^0_{}$ is uniformly positive definite (recall that SSOSC does not necessarily guarantee positive definiteness of $\bQ_{}$).

\begin{lemma}\label{lem:convexification}
  Under Assumption \ref{ass:ublh}, \ref{ass:ussosc}, \ref{ass:ulicq}, ${\bQ_{}}+\overline{\mu} (\bA^0_{})^\top {\bA^0_{}}\succeq (\gamma/2) \bI.$
\end{lemma}
\begin{proof}
  From Assumption \ref{ass:ublh}, $\Vert \bH\Vert\leq L$; this implies that its submatrices $\bQ,\bA^0$ satisfy $\osigma(\bQ),\osigma(\bA^0)\leq L$. We can obtain $\ulambda(\bQ+\overline{\mu}(\bA^0)^\top \bA^0)$ from:
  \begin{subequations}\label{eqn:eigprob}
    \begin{align}
      \min_w\;& w^\top ({\bQ}+\overline{\mu} (\bA^0)^\top {\bA^0})w\label{eqn:eigprob-obj}\\
      \st\;& \Vert w \Vert =1.\label{eqn:eigprob-con}
    \end{align}
  \end{subequations}
  By fundamental theorem of linear algebra, any $w\in\mathbb{R}^{\bn_{\bx}}$ can be expressed as $w=Zw_Z+ Yw_Y$, where the columns of $Z$ form an orthonormal basis for the null space of ${\bA^0}$ and the columns of $Y$ form an orthonormal basis for the row space of ${\bA^0}$. We have that $\Vert w_Z\Vert^2 + \Vert w_Y\Vert^2=1$, $\Vert Zw_Z\Vert=\Vert w_Z\Vert$, and $\Vert Yw_Y\Vert=\Vert w_Y\Vert$, which follows from \eqref{eqn:eigprob-con} and orthogonality of $Z$ and $Y$. The objective \eqref{eqn:eigprob-obj} satisfies:
  \begin{align}\label{eqn:convexification-1}
    &w^\top ({\bQ}+\overline{\mu} (\bA^0)^\top {\bA^0}) w\\
    &= w_Z^\top Z^\top {\bQ} Zw_Z + 2  w_Y^\top Y^\top {\bQ}Zw_Z + w_Y^\top Y^\top {\bQ}Yw_Y + \overline{\mu}  w_Y^\top Y^\top (\bA^0)^\top {\bA^0}Yw_Y \nonumber\\
    &\geq \gamma  \Vert w_Z\Vert^2- 2\Vert \bQ\Vert\Vert Z w_Z \Vert \Vert Y w_Y\Vert -\Vert \bQ \Vert  \Vert Y w_Y\Vert^2 + \overline{\mu} \ulambda(Y^\top (\bA^0)^\top \bA^0 Y) \Vert w_Y\Vert^2\nonumber\\
    &\geq \gamma \Vert w_Z\Vert^2 - 2\osigma(\bQ) \Vert w_Z\Vert\Vert w_Y\Vert- \osigma(\bQ) \Vert w_Y\Vert^2 + \overline{\mu} \ulambda(\bA^0 YY^\top (\bA^0)^\top )\Vert w_Y\Vert^2\nonumber\\
    &\geq \gamma (1-\Vert w_Y\Vert^2) - 2L\Vert w_Y\Vert- L \Vert w_Y\Vert^2 + \overline{\mu} \ulambda(\bA^0 (\bA^0)^\top )\Vert w_Y\Vert^2\nonumber\\
    &\geq  \gamma  -2L \Vert w_Y\Vert +(\overline{\mu}\beta - \gamma -L)\Vert w_Y\Vert^2.\nonumber
  \end{align}
  The equality follows from $\bA^0 Z=\bzero$; the first inequality follows from (i) Assumption \ref{ass:ussosc}, (ii) submultiplicativity of matrix norms, and (iii) the fact that $w^\top Mw\geq \ulambda(M)\Vert w\Vert^2$ for positive definite $M$; the second inequality comes from (i) the fact that the induced 2-norm is equal to the largest singular value and (ii) the equality $\ulambda(MM^\top)=\ulambda(M^\top M)$ for square $M$; the third inequality follows from (i) Assumption \ref{ass:ublh} and (ii) $\bA^0 YY^\top=\bA^0$ since $Y$ is an orthogonal matrix whose columns span the row space of $\bA^0$; and the last inequality follows from Assumption \ref{ass:ulicq}. We have $L\neq 0$ from SSOSC and LICQ. This implies that $\overline{\mu}\beta -\gamma -L=2L^2/\gamma>0$. Accordingly, the quadratic expression on the right-hand side of the last inequality of \eqref{eqn:convexification-1} is lower-bounded by: $w^\top ({\bQ}+\overline{\mu} (\bA^0)^\top {\bA^0})w \geq \gamma-\frac{L^2}{\overline{\mu}\beta -\gamma -L}  =\gamma/2$.
\end{proof}
\begin{proof}[Proof of Lemma \ref{lem:uniform-2}]
  It suffices to show that $\usigma(\bH[\cB,\cB])$ for $\cB^0\subseteq \cB\subseteq \cB^1$ is lower bounded by the right-hand-side of \eqref{eqn:uniform-2}. We know that $\bH[\cB,\cB]$ is a permutation of:
  \begin{align}\label{eqn:QAmat}
    \begin{bmatrix}
      \bQ &(\bA')^\top\\
      \bA',
    \end{bmatrix}
  \end{align}
  where $\bA':=\bA[\mathcal{A},:]$, and $\mathcal{A}:=(\phi^{-1}(\cB)\setminus\mathbb{I}_{[1,\bn_{\bx}]})-\bn_{\bx}$; here, $\phi:\mathbb{I}_{[1,\bn]}\rightarrow \mathbb{I}_{[1,\bn]}$ is a permutation that achieves $\bz[\phi(i)] = [\bxi;\bbeta][i]$. It thus suffices to show that the smallest singular value of the matrix in \eqref{eqn:QAmat} with $\mathcal{A}^0\subseteq\mathcal{A}\subseteq\mathcal{A}^1$ is lower bounded by the right-hand side of \eqref{eqn:uniform-2}. We now make the following observation:
  \begin{align}\label{eqn:AAs}
    (\bA')^\top \bA'&\succeq(\bA^0)^\top \bA^0 ;\quad
    \ulambda(\bA' (\bA')^\top )\geq\ulambda(\bA^1 (\bA^1)^\top );
  \end{align}
here, the first inequality results from $(\bA')^\top \bA'-(\bA^0)^\top \bA^0 = \bA[\mathcal{A}\setminus \mathcal{A}^0,:]^\top\bA[\mathcal{A}\setminus \mathcal{A}^0,:]\succeq \bzero$. To establish the second inequality, we consider a unit vector $w\in\mathbb{R}^{\bm}$ such that $w[\mathcal{A}]$ is the eigenvector of $\bA'(\bA')^\top$ associated with the smallest eigenvalue and $w[\mathbb{I}_{[1,\bm]}\setminus \mathcal{A}]=\bzero$. We can see that:
  \begin{align}
    \ulambda(\bA^1(\bA^1)^\top)\leq w[\mathcal{A}^1]^\top \bA^1(\bA^1)^\top w[\mathcal{A}^1] = \ulambda(\bA'(\bA')^\top);
  \end{align}
  here, the first inequality follows from the fact that $\lambda(\bA^1(\bA^1)^\top)$ is the smallest  eigenvalue, and the equality follows from the fact that $w[\mathcal{A}^1\setminus\mathcal{A}]=\bzero$. This establishes the second inequality in \eqref{eqn:AAs}.

  We now study the inverse of the matrix in \eqref{eqn:QAmat}; note that $\usigma(\bH)=\Vert \bH^{-1}\Vert$ and
  \begin{subequations}\label{eqn:uniform-2-1}
    \begin{align}
      \begin{bmatrix}
        \bQ &(\bA')^\top\\
        \bA'
      \end{bmatrix}^{-1}
            &=
              \begin{bmatrix}
                \bQ+\overline{\mu}(\bA')^\top \bA' & (\bA')^\top\\
                \bA'
              \end{bmatrix}^{-1}
      \begin{bmatrix}
        \bI & \overline{\mu} (\bA')^\top\\
        &\bI
      \end{bmatrix}
      \\
            &=
              \begin{bmatrix}
                T + \overline{\mu} T (\bA')^\top S\bA' T
                &
                T (\bA')^\top S\\
                S\bA' T
              \end{bmatrix}
      \begin{bmatrix}
        \bI & \overline{\mu}(\bA')^\top\\
        &\bI
      \end{bmatrix},
    \end{align}
  \end{subequations}
  where $T:= (\bQ+\overline{\mu}(\bA')^\top\bA')^{-1}$, and $S:=(\bA' (\bQ+\overline{\mu}(\bA')^\top\bA') (\bA')^\top)^{-1}$;
  here, the first equality can be easily verified; and the second equality follows from \cite[Proposition 2.8.7]{bernstein2009matrix}. Now observe that:
  \begin{align}\label{eqn:'0}
    \ulambda(\bQ+\overline{\mu}(\bA')^\top\bA')&\geq \ulambda(\bQ+\overline{\mu}(\bA^0)^\top\bA^0)\geq\gamma/2;
  \end{align}
here, the first inequality follows from \eqref{eqn:AAs} and the second inequality follows from Lemma \ref{lem:convexification}. Furthermore,
  \begin{align*}
    \ulambda(\bA' (\bQ+\overline{\mu}(\bA')^\top\bA') (\bA')^\top) &\geq \ulambda(\bQ+\overline{\mu}(\bA')^\top\bA')\ulambda(\bA'(\bA')^\top)\geq \gamma\beta/2;
  \end{align*}
  here, the first inequality follows from
  \begin{align*}
    \min_{\Vert w\Vert\leq 1}w^\top(\bA' (\bQ+\overline{\mu}(\bA')^\top\bA') (\bA')^\top)w&\geq \min_{\Vert w\Vert\leq 1} \ulambda(\bQ+\overline{\mu}(\bA')^\top\bA') \Vert(\bA')^\top w\Vert^2 \\
                                                                                  &\geq \ulambda(\bQ+\overline{\mu}(\bA')^\top\bA')\ulambda(\bA'(\bA')^\top),
  \end{align*}
  and the second inequality follows from \eqref{eqn:'0} and \eqref{eqn:AAs}. Thus, $\Vert T\Vert \leq 2/\gamma$ and $\Vert S\Vert \leq 2/\gamma\beta$. By using Lemma \ref{lem:golub} (will be introduced later), the triangle inequality, the submultiplicativity of matrix norms, and the fact that $\bQ$, $\bA'$ are submatrices of $\bH$, we have:
  \begin{subequations}\label{eqn:uniform-2-2}
    \begin{align}
      \left\Vert
      \begin{bmatrix}
        T + \overline{\mu} T (\bA')^\top S\bA' T
        &
        T (\bA')^\top S\\
        S\bA' T
      \end{bmatrix}\right\Vert
        &\leq
          \frac{2}{\gamma} +  \frac{8\overline{\mu} L^2}{\gamma^3\beta} +\frac{4L}{\gamma^2\beta}
      \\
      \left\Vert
      \begin{bmatrix}
        \bI& \overline{\mu}(\bA')^\top\\
        & \bI
      \end{bmatrix}\right\Vert
      &\leq
        1+\overline{\mu} L.
    \end{align}
  \end{subequations}
  Therefore, from \eqref{eqn:uniform-2-1} and \eqref{eqn:uniform-2-2}, we obtain:
  \begin{align}\label{eqn:uniform-done}
    \left\|\begin{bmatrix}\bQ &(\bA')^\top\\ \bA' \end{bmatrix}^{-1} \right\| \leq
    \left(\frac{2}{\gamma} +  \frac{8\overline{\mu} L^2}{\gamma^3\beta} +\frac{4L}{\gamma^2\beta}\right)\left(1+\overline{\mu} L\right).
  \end{align}
  Because \eqref{eqn:uniform-done} holds for any $\mathcal{A}^0\subset\mathcal{A}\subset\mathcal{A}^1$, the desired condition is obtained.
\end{proof}

We have established in Lemma \ref{lem:trivial}, \ref{lem:uniform-2} that Assumption \ref{ass:ublh}, \ref{ass:ussosc}, \ref{ass:ulicq} guarantee the desired uniform boundedness \eqref{eqn:uniform-bound}. The result is summarized as follows.
\begin{theorem}\label{thm:summary-1}
  Under Assumption \ref{ass:ublh}, \ref{ass:ussosc},  \ref{ass:ulicq}, we have that:
  \begin{align*}
    \oosigma_{\bH}\leq L,\quad
    \oosigma_{\bR}\leq L,\quad
    \uusigma_{\bH}\geq \left(\frac{2}{\gamma} +  \frac{8\overline{\mu} L^2}{\gamma^3\beta} +\frac{4L}{\gamma^2\beta}\right)^{-1}\left(1+\overline{\mu} L\right)^{-1},
  \end{align*}
  where $\overline{\mu}$ is defined in Lemma \ref{lem:uniform-2} and $\oosigma_{\bH},\oosigma_{\bR},\uusigma_{\bH}$ are defined in Theorem \ref{thm:main}.
\end{theorem}
\begin{proof}
  The result follows directly from Lemma \ref{lem:trivial}, \ref{lem:uniform-2}.
\end{proof}

Now that we have the uniform upper and lower bounds of $\oosigma_{\bH},\oosigma_{\bR},\uusigma_{\bH}$, all the quantities in Theorem \ref{thm:main} can be expressed in terms of $L$, $\gamma$, and $\beta$. Accordingly, we can uniformly bound the parameters $(\Upsilon,\rho)$ in Theorem \ref{thm:main} using Theorem \ref{thm:summary-1}.

\subsection{Sufficient Conditions for Uniformly Bounded Lagrangian Hessian}\label{sec:block-0}
We now show that uBLH (Assumption \ref{ass:ublh}) can be established from {\it composable} sufficient conditions. In particular, we assume the boundedness conditions for each node and the boundedness in the graph degree to establish the uniform boundedness condition for the full problem. The key assumptions are stated as follows.

\begin{assumption}[Uniformly Bounded Degree of Graphs]\label{ass:uniform-degree}
  There exists $D\in\mathbb{I}_{>0}$ such that $|\cN_{\cG}[i]|\leq D$ for any $ i\in \cV$. 
\end{assumption}
\begin{assumption}[Uniformly Bounded Second Derivatives]\label{ass:uniform-2nd}
  There exists $C\geq 0$ such that $ \Vert H_{ij} \Vert, \Vert R_{ij}\Vert \leq C$ for any $i,j\in \cV$, where $H_{ij}$ and $R_{ij}$ are defined in \eqref{eqn:HRij}.
\end{assumption}

The following lemma establishes that Assumptions \ref{ass:uniform-degree}, \ref{ass:uniform-2nd} imply uBLH.
\begin{lemma}\label{lem:uniform-1}
  Under Assumption \ref{ass:uniform-degree}, \ref{ass:uniform-2nd}, we have that $\|\bH\|,\|\bR\|\leq CD^2.$
\end{lemma}
In order to prove this lemma, we first need to establish a general inequality for matrix norms. The following lemma is a generalization of {the} inequality $\Vert M\Vert \leq \left(\Vert M \Vert_1 \Vert M\Vert_\infty\right)^{1/2}$ \cite[Corollary 2.3.2]{golub2012matrix}.
\begin{lemma}\label{lem:golub}
  Consider $M\in\mathbb{R}^{m\times n}$ with index set families $\cI:=\{I_i\}_{i\in\cV}$ and $\cJ:=\{J_i\}_{i\in\cV}$ partitioning $\mathbb{I}_{[1,m]}$ and $\mathbb{I}_{[1,n]}$. The following holds, where $M_{[i][j]}:=M[I_i,J_j]$.
  \begin{align}\label{eqn:golub}
    \osigma(M)\leq \Big(\Big(\max_{i\in \cV}\sum_{j\in \cV}\Vert M_{[i][j]}\Vert\Big)\Big(\max_{j\in \cV}\sum_{i\in \cV}\Vert M_{[i][j]}\Vert\Big)\Big)^{1/2}
  \end{align}
\end{lemma}
\begin{proof}
  The inequality holds trivially if $M=\bzero$; we thus assume $M\neq \bzero$. Consider the left singular vector $v\in \mathbb{R}^n$ of $M$ with singular value $\osigma(M)$. We have that $\osigma(M)^2 v =  M M^\top v$. We let $u=M^\top v$, which yields $\osigma(M)^2 v =  M u$; accordingly,
  \begin{align*}
    \osigma(M)^2 \sum_{i\in \cV}\Vert v_{[i]}\Vert\leq  \sum_{j\in \cV}(\sum_{i\in \cV}\Vert M_{[i][j]}\Vert) \Vert u_{[j]} \Vert\leq \Big(\max_{j\in \cV}\sum_{i\in \cV} \Vert M_{[i][j]}\Vert \Big)\Big(\sum_{j\in \cV}\Vert u_{[j]}\Vert\Big),
  \end{align*}
  where the first inequality is obtained by applying the triangle inequality and the submultiplicativity of the matrix norm, and by switching the order of summation; the second inequality is obtained from $\sum_{i\in \cV}\Vert M_{[i][j]}\Vert \leq \max_{j\in \cV}\sum_{i\in \cV}\Vert M_{[i][j]}\Vert.$ Using the same logic, we obtain: $\sum_{j\in \cV}\Vert u_{[j]}\Vert \leq\left(\max_{i\in \cV}\sum_{j\in\cV}\Vert M_{[i][j]} \Vert\right)\left(\sum_{i\in \cV}\Vert v_{[i]}\Vert\right).$ From these results and the fact that $v\neq \bzero$ (by $M\neq\bzero$), we obtain \eqref{eqn:golub}.
\end{proof}

\begin{proof}[Proof of Lemma \ref{lem:uniform-1}]
  It suffices to show that $\osigma(\bH),\osigma(\bR) \leq CD^2$. As observed in Section \ref{sec:nodal}, $\bH$ and $\bR$ have bandwidth not greater than two since $H_{ij}$ and $R_{ij}$ equal zero if $d_\cG(i,j)>2$. Hence, the number of nonzero blocks on one-block rows or on one-block columns of $\bH$ and $\bR$ is at most $ D^2$, since $$|N^2_\cG[i]|\leq \underbrace{1}_{\text{$i$ itself}}+\underbrace{(D-1)}_{\text{nodes with distance $1$}}+\underbrace{(D-1)^2}_{\text{nodes with distance $2$}}\leq D^2$$ for any $i\in \cV$ (i.e., for any node, there exist at most $D^2$ nodes within distance two); here, the second inequality follows from $D\geq 1$ (the graph is nonempty). As such, we have: $\max_{i\in \cV}\sum_{j\in \cV}\Vert H_{ij}\Vert\vee \max_{j\in \cV}\sum_{i\in \cV}\Vert H_{ij}\Vert\leq CD^2$ and $\max_{i\in \cV}\sum_{j\in \cV}\Vert R_{ij}\Vert\vee\max_{j\in \cV}\sum_{i\in \cV}\Vert R_{ij}\Vert\leq CD^2$. By Lemma \ref{lem:golub}, $\osigma(\bH)\leq CD^2$ and $\osigma(\bR)\leq CD^2$.
\end{proof}

\subsection{Sufficient Conditions for Uniform SSOSC and LICQ}\label{sec:block}
We now provide {\it composable sufficient conditions} for uSSOSC and uLICQ. In particular, we establish sufficient conditions for uSSOSC and uLICQ that do not require checking singular values of indefinitely large matrices. A key characteristic of such problems is that there exists a recurrent structure (as depicted in Figure \ref{fig:graphs}). As such, we can construct sufficient conditions based on uSSOSC and uLICQ over {\em blocks} of the Hessian and Jacobian matrices, induced a partition $\cJ_{}:=\{J_{(q)}\}_{q\in\mathbb{I}_{[1,\overline{q}]}}$ of the primal variable index set $\mathbb{I}_{[1,\bn_{\bx}]}$. To state these assumptions, we define the following submatrices of $\bQ_{}$ and $\bA_{}$ for each $q\in \mathbb{I}_{[1,\overline{q}]}$: 
\begin{align*}\small
  \bQ_{(q)}:= \bQ_{}[J_{(q)},J_{(q)}],\;
  \bQ_{(-q)}:= \bQ_{}[J_{(q)},J_{(-q)}],\;
  \bA^{-}_{(q)}:= \bA_{}[\mathcal{A}^-_{(q)},J_{(q)}],\;
  \bA^{+}_{(q)} := \bA_{}[\mathcal{A}^+_{(q)},J_{(q)}], 
\end{align*}
where: $J_{(-q)}:=\mathbb{I}_{[1,\bn_{\bx}]}\setminus J_{(q)}$, $\mathcal{A}^-_{(q)}:=\{i\in\mathcal{A}^0_{}: \bA_{}[i,\mathbb{I}_{[1,\bn_{\bx}]}\setminus J_{(q)}]= \bzero\}$, and $\mathcal{A}^+_{(q)}:=\{i\in\mathcal{A}^1_{}:\bA_{}[i,J_{(q)}]\neq\bzero\}$. In words, $\mathcal{A}^-_{(q)}$ denotes the set of constraint indices that are exclusively coupled with the variables in $J_{(q)}$, and $\mathcal{A}^+_{(q)}$ denotes the set of constraint indices that have nonempty coupling with the variables in $J_{(q)}$. Based on these, we now present the key assumptions.

\begin{assumption}[Block Diagonal $\bQ$]\label{ass:blkdiag}
  $\bQ_{(-q)} = \bzero$ for $q\in\mathbb{I}_{[1,\overline{q}]}$.
\end{assumption}
\begin{assumption}[Nonzero Rows of $\bA$]\label{ass:nzrow}
  $\bA[i,:] \neq \bzero$ for $i\in\mathcal{A}^1$.
\end{assumption}
\begin{assumption}[Block SSOSC]\label{ass:block-SSOSC}
  $\exists\gamma>0$: $ReH(\bQ_{(q)},\bA^{-}_{(q)}) \succeq \gamma \bI$ for $q\in\mathbb{I}_{[1,\overline{q}]}$.
\end{assumption}
\begin{assumption}[Block LICQ]\label{ass:block-licq}
  $\exists\beta>0$: $\bA^{+}_{(q)} (\bA^{+}_{(q)})^\top \succeq \beta \bI$ for $q\in\mathbb{I}_{[1,\overline{q}]}$.
\end{assumption}
Note that Assumption \ref{ass:blkdiag} does not assume separability of the problem; a block-diagonal structure in $\bQ$ is obtained when coupling across blocks exists only via linear constraints. This is not a restrictive assumption since any problem of the form \eqref{eqn:prob} can be reformulated into a form with linear coupling by introducing auxiliary variables (i.e., via a {\it lifting} procedure). Assumption \ref{ass:nzrow} is not difficult to satisfy. We now show that the above assumptions guarantee uSSOSC and uLICQ for the original NLP \eqref{eqn:prob}. 

\begin{lemma}\label{lem:bssosc}
  Under Assumption \ref{ass:blkdiag}, \ref{ass:block-SSOSC} we have $ReH(\bQ,\bA^0)\succeq \gamma \bI.$
\end{lemma}
\begin{proof}
  From the block diagonal structure of $\bQ$ (Assumption \ref{ass:blkdiag}), $\bx^\top\bQ\bx = \sum_{q\in\mathbb{I}_{[1,\overline{q}]}} \bx[J_{(q)}]^\top\bQ_{(q)} \bx[J_{(q)}].$ If $\bA^0\bx=\bzero$, $\bA^-_{(q)} \bx[J_{(q)}]=\bzero$ holds for $q\in\mathbb{I}_{[1,\overline{q}]}$; therefore, by Assumption \ref{ass:block-SSOSC}, the following can be obtained: if $\bA^0\bx = \bzero$,
  \begin{align}\label{eqn:block-SSOSC-1}
    \sum_{q\in\mathbb{I}_{[1,\overline{q}]}} \bx[J_{(q)}]^\top\bQ_{(q)} \bx[J_{(q)}]\geq  \sum_{q\in\mathbb{I}_{[1,\overline{q}]}}\gamma \Vert \bx[J_{(q)}]\Vert^2 = \gamma \Vert \bx\Vert^2.
  \end{align}
  Here, the last equality follows from the fact that $\{J_{(q)}\}_{q\in\mathbb{I}_{[1,\overline{q}]}}$ partitions $\mathbb{I}_{[1,\bn_{\bx}]}$. From \eqref{eqn:block-SSOSC-1}, we obtain the desired result.
\end{proof}

\begin{lemma}\label{lem:blicq}
  Under Assumption \ref{ass:nzrow}, \ref{ass:block-licq} we have $\bA^1(\bA^1)^\top \succeq \beta\bI.$
\end{lemma}
\begin{proof}
  We have that for any $\by\in\mathbb{I}_{[1,\bm]}$,
  \begin{align*}
    \by[\mathcal{A}^1]^\top \bA^1 (\bA^1)^\top \by[\mathcal{A}^1]
    &= \by[\mathcal{A}^1]^\top (\sum_{q\in\mathbb{I}_{[1,\overline{q}]}} \bA[\mathcal{A}^1,J_{(q)}] (\bA[\mathcal{A}^1,J_{(q)}])^\top) \by[\mathcal{A}^1] \\
    &= \sum_{q\in\mathbb{I}_{[1,\overline{q}]}} (\by [\mathcal{A}^+_{(q)}])^\top \bA^+_{(q)} (\bA^+_{(q)})^\top \by [\mathcal{A}^+_{(q)}].
  \end{align*}
  Here the second equality follows from $\bA[\mathcal{A}^1\setminus \mathcal{A}^+_{(q)},J_{(q)}]=\bzero$. By Assumption \ref{ass:block-licq},
  \begin{align}\label{eqn:block-licq-1}
    \by[\mathcal{A}^1]^\top \bA^1 (\bA^1)^\top \by[\mathcal{A}^1]
    &\geq \sum_{q\in\mathbb{I}_{[1,\overline{q}]}} \beta\Vert\by [\mathcal{A}^+_{(q)}]\Vert^2 \geq \beta\Vert \by[\mathcal{A}^1]\Vert^2. 
  \end{align}
  where the second inequality follows from the fact that $\bigcup_{q\in\mathbb{I}_{[1,\overline{q}]}}\mathcal{A}^+_{(q)} = \mathcal{A}^1$, which follows from Assumption \ref{ass:nzrow}. Inequality \eqref{eqn:block-licq-1} implies the desired result.
\end{proof}

We now summarize the developments in Section \ref{sec:block-0}, \ref{sec:block} in the following theorem.
\begin{theorem}\label{thm:summary-2}
  Under Assumption \ref{ass:uniform-degree}, \ref{ass:uniform-2nd}, \ref{ass:blkdiag}, \ref{ass:nzrow}, \ref{ass:block-SSOSC}, \ref{ass:block-licq}, we have
  \begin{align*}
    \oosigma_{\bH}&\leq CD^2;\quad
                    \oosigma_{\bR}\leq CD^2;\quad
                                    \uusigma_{\bH}\geq\left(\frac{2}{\gamma} +  \frac{8\overline{\mu} C^2D^4}{\gamma^3\beta} +\frac{4CD^2}{\gamma^2\beta}\right)^{-1}\left(1+\overline{\mu} CD^2\right)^{-1},
  \end{align*}
  where $\overline{\mu}$ is defined in Lemma \ref{lem:uniform-2}, and $\oosigma_{\bH},\oosigma_{\bR},\uusigma_{\bH}$ are defined in Theorem \ref{thm:main}.
\end{theorem}
\begin{proof}
  The result follows from Theorem \ref{thm:summary-1} and Lemma \ref{lem:uniform-1}, \ref{lem:bssosc}, \ref{lem:blicq}.
\end{proof}

The results in Section \ref{sec:uniform-sub}-\ref{sec:block} are useful for different problems of interest but might not be applicable to certain problem classes. For instance, it is difficult to derive uniform regularity conditions for multi-stage stochastic programs with a fixed number of children per node because the probability of a given stage decays asymptotically over time (this prevents Assumption \ref{ass:block-SSOSC} to hold). Also, we have not discussed how the sensitivity behavior changes when the discretization resolution changes; such behavior can be used to understand the sensitivity behavior of the continuous-time (infinite-dimensional) optimization problems studied in \cite{grune2019sensitivity,grune2020exponential,grune2020abstract}. We will leave specialized treatment for those problems as a topic of future work.

\subsection{Discussion}
We now illustrate the practical applicability of our results using simple examples. With the intuition obtained from the examples, we shall discuss qualitative conditions under which the problem is likely to exhibit EDS.

\begin{example}\label{eg:pde}
  We first consider a PDE optimization problem for a steady-state thin-plate system described in \cite{mathworks2020nonlinear}:
  \begin{subequations}\label{eqn:pdeo}\small
    \begin{align}
      \min_{s,u}\;& \int_{w\in \Omega_{(k)}}\frac{1}{2}a(s(w)-s_{\text{ref}}(w))^2 + \frac{1}{2}u(w)^2 dw\label{eqn:pdeo-obj}\\
      \st\;& \Delta s(w) = \frac{2h_c}{\kappa t_z} (s(w)-\overline{T})+ \frac{2\epsilon\sigma}{\kappa t_z} (s(w)^4-\overline{T}^4)- \frac{1}{\kappa t_z}(bu(w) + d(w)),\; w\in \Omega_{(k)} \label{eqn:pdeo-pde}\\
                                &\nabla{s}(w)\cdot\hat{\bn}(w) = 0,\; w\in \partial\Omega_{(k)},\label{eqn:pdeo-bc}
    \end{align}
  \end{subequations}
  where $\Omega_{(k)}=[0,k]\times[0,k]\subseteq\mathbb{R}^2$ is the 2-dimensional domain of interest; $\partial\Omega_{(k)}$ is the boundary of $\Omega_{(k)}$; $s:\Omega_{(k)}\rightarrow \mathbb{R}$ is the temperature; $u:\Omega_{(k)}\rightarrow \mathbb{R}$ is the control; $d:\Omega_{(k)}\rightarrow \mathbb{R}$ is the disturbance; $\Delta$ is the Laplacian operator; $\hat{\bn}$ is the unit normal vector; $\cdot$ is the inner product; \eqref{eqn:pdeo-pde} is the heat equation whose right-hand-side consists of convection, radiation, and forcing terms by control and disturbance; \eqref{eqn:pdeo-bc} is the Neumann boundary condition (i.e., insulated); $s_{\text{ref}}:\Omega_{(k)}\rightarrow\mathbb{R}$ is the desired temperature; $\kappa=400$, $t_z=.01$, $h_c=1$, $\epsilon=.5$, $\sigma=5.67\times 10^{-8}$, and $\overline{T}=300$ are the constant parameters. We define the variable vector to be $x(w)=[s(w);u(w)]$; the data vector to be $p(w):=[s_{\text{ref}}(w);d(w)]$ for $w\in \Omega_{(k)}$. We consider a discretized version of \eqref{eqn:pdeo} (e.g., each $[i,j]\times[i+1,j+1]$ cell for $i,j=0,\cdots,k-1$); we consider the discretization mesh as the graph $\cG$ and let $K=\mathbb{I}_{>0}$. Observe that the domain expands with $k$ but the fineness of the discretization mesh remains the same. 
\end{example}

We assume that $s_{\text{ref}}(w)=d(w)=\overline{T}$ for $w\in \Omega$; in this case, the problem admits a trivial solution $s^\star(w)=\overline{T}$ and $u^\star(w)=0$. Since there are no inequality constraints, $\bA^0_{(k)}=\bA^1_{(k)}=\bA_{(k)}$. Furthermore, there exist permutations of the Lagrangian Hessian $\bQ_{(k)}$ and of constraint Jacobian $\bA_{(k)}$ that can be expressed as:
\begin{align*}
  \widetilde{\bQ}_{(k)}=\Pi_{(k)}\bQ_{(k)}\Pi_{(k)}&=\begin{bmatrix}
    \bI\\
    &a\bI
  \end{bmatrix};\quad
      \widetilde{\bA}_{(k)}:=
      \bA_{(k)}\Pi_{(k)}=
      \begin{bmatrix}
        \bL_{(k)}+c_0 \bI &(b/b_0)\bI
      \end{bmatrix},
\end{align*}
where $\Pi_{(k)}$ is the permutation operator, $\bL_{(k)}$ is the Laplacian matrix induced by the mesh graph, $c_0=(2h_c+8\epsilon\sigma\overline{T}^3)/\kappa t_z$ and $b_0=-1/\kappa t_z$. It is easy to see that Assumptions \ref{ass:uniform-degree}, \ref{ass:uniform-2nd} hold; thus, upper bounds in \eqref{eqn:uniform-bound} are satisfied. Also, one can observe that the uniform LICQ condition holds if $b\neq 0$ from $\widetilde{\bA}_{(k)}(\widetilde{\bA}_{(k)})^\top = (\bL_{(k)}+c_0 \bI)^2+(b/b_0)^2\bI\succeq (b/b_0)^2\bI$. Furthermore, since the smallest eigenvalue of $Z^\top\widetilde{\bQ}_{(k)}Z$ with orthogonal $Z$ is always greater than or equal to the smallest eigenvalue of $\widetilde{\bQ}_{(k)}$, we have that $ReH(\widetilde{\bQ}_{(k)},\widetilde{\bA}_{(k)})\succeq aI$. Thus, if $a>0$, uniform SSOSC holds. This example demonstrates that the results in Section \ref{sec:uniform-sub} can be used to check uniform boundedness. 

\begin{example}\label{eg:ocp}
  Consider a dynamic optimization problem for energy storage:
    \begin{subequations}\label{eqn:ocp}
    \begin{align}
      \min_{\{s_i,u_i,v_i\}_{i=1}^k}\;& \sum_{i=1}^k \frac{1}{2}a(s_i)^2 + \frac{1}{2}(u_i)^2 + \pi_i v_i \\
      \st\;& s_1 = \overline{s},\;(\mu_1)\\
      &s_{i} = s_{i-1}+bu_{i-1}+ w_i,\;i\in\mathbb{I}_{[2,k]},\;(\mu_i)\\
      &v_{i} = u_{i}+d_i,\;i\in\mathbb{I}_{[1,k]},\;(\nu_i).
    \end{align}
  \end{subequations}
  Here, $s_i\in\mathbb{R}$ is the stored energy (state) at time $i$; $u_i\in\mathbb{R}$ is the charge/discharge of energy (control); $v_i\in\mathbb{R}$ are the transactions with the grid; $\overline{s}=w_1$ is the initial storage; $\pi_i\in\mathbb{R}$ is the energy price forecast; $d_i$ is the energy demand forecast; and $w_i$ is the disturbance forecast. We define $x_i:=[s_i,u_i,v_i]$ as the primal variables; $y_i:=[\mu_i,\nu_i]$ as the dual variables; $p_i=[\pi_i;w_i;d_i]$ as the data. We consider  $\cG_{(k)}=(\cV_{(k)},\cE_{(k)})$ as a linear graph that represents time domain, $\cV_{(k)}:=\mathbb{I}_{[1,k]}$ and $K=\mathbb{I}_{>0}$. The index $k$ is the length of the horizon. 
\end{example}

Given any $k\in K$, \eqref{eqn:ocp} satisfies SSOSC and LICQ, regardless of the choice of $(a,b)$, but \eqref{eqn:uniform-bound} can be violated. For example, if $b=0$, $w_1=1$, and $w_2=w_3=\cdots=0$, one can see that $s_i=1$ for $i\in\mathbb{I}_{[1,k]}$; thus, $\usigma(\bH_{(k)})\rightarrow 0$. On the other hand, if $a>0$ and $b\neq 0$, we will see below that $\usigma(\bH_{(k)})$ can be uniformly bounded below by a strictly positive number. It is easy to see that Assumptions \ref{ass:uniform-degree}, \ref{ass:uniform-2nd} hold; thus, the upper bounds in \eqref{eqn:uniform-bound} can be obtained. Since there is no inequalities, $\bA^0=\bA^1=\bA$.

It is clear that Assumptions \ref{ass:uniform-degree}, \ref{ass:uniform-2nd} hold for Example \ref{eg:ocp}, but it is difficult to check whether Assumptions \ref{ass:ussosc}, \ref{ass:ulicq} (uniform SSOSC and LICQ) hold. Specifically, we need to show that there exist some uniform constants $\gamma,\beta>0$ such that $ReH(\bQ_{(k)},\bA_{(k)})\succeq \gamma \bI$ and $\bA_{(k)}(\bA_{(k)})^\top \succeq \beta I$ hold for
\begin{align*}
  \bQ_{(k)}
  &=
    \begin{bmatrix}
      a\\
      &1\\
      &&0\\
      &&&\ddots\\
      &&&&a\\
      &&&&&1\\
      &&&&&&0\\
    \end{bmatrix}
  \bA_{(k)}
  =
  \begin{bmatrix}
    \begin{matrix}
      1\\
      &-1&1\\
      -1 &-b&&1\\
      &&&&-1&1\\
    \end{matrix}\\
    &\ddots\\
    &&\begin{matrix}
      -1 &-b&&1\\
      &&&&-1&1
    \end{matrix}
  \end{bmatrix}.
\end{align*}
  The difficulty arises from the indefinitely growing sizes of $\bQ_{(k)}$ and $\bA_{(k)}$; hence, we leverage the result in Section \ref{sec:block}. If we choose $J_{(k)(q)}=\mathbb{I}_{[3(q-1)+1,3q]}$ for $q\in \mathbb{I}_{[1,k]}$, 
\begin{align}\label{eqn:ocp-2}
  \begin{array}{ll}
    \bQ_{(k)(q)} =  
    \begin{bmatrix}
      a\\
      &1\\
      &&0
    \end{bmatrix};
    &
    \bQ_{(k)(-q)} =
    \bzero;
    \\
    \bA^{-}_{(k)(q)} =
    \begin{cases}
      \begin{bmatrix}
        1\\
        0&-1&1
      \end{bmatrix},&q=1
      \\
      \begin{bmatrix}
        0&-1&1      
      \end{bmatrix},
      &q\in\mathbb{I}_{[2,k]}
    \end{cases};
    &
    \bA^{+}_{(k)(q)} =
    \begin{cases}
      \begin{bmatrix}
        1&&\\
        &-1&1\\
        -1&-b
      \end{bmatrix} & q\in\mathbb{I}_{[1,k-1]}
      \\    
      \begin{bmatrix}
        1&&\\
        &-1&1
      \end{bmatrix}  & q= k
    \end{cases}.
  \end{array}
\end{align}
One can confirm from \eqref{eqn:ocp-2} that Assumptions \ref{ass:block-SSOSC}, \ref{ass:block-licq} are satisfied if $a>0$ and $b\neq 0$ hold. Thus, Example \ref{eg:ocp} with $a>0$ and $b\neq 0$ satisfies \eqref{eqn:uniform-bound}; accordingly, $(\Upsilon,\rho)$ are uniformly bounded. This demonstrates the usefulness of the results in Section \ref{sec:block} in checking the uniform boundedness conditions.

The above results also provide {\em qualitative} conditions under which quantities in \eqref{eqn:uniform-bound} are likely to be moderately bounded. The first is having {\em sufficient positive curvature} in the objective function (related to Assumption \ref{ass:ussosc}, \ref{ass:block-SSOSC}), and the second is having a {\em sufficient flexibility} in the constraints (related to Assumption \ref{ass:ulicq}, \ref{ass:block-licq}). Indeed, in the absence of nonlinear constraints, the SSOSC implies the strong convexity of the objective function on the null space. In addition, in many practical domains, flexibility is defined as the ability to endure and adjust to the variations in conditions \cite{grossmann1983operability,swaney1985index,nosair2015flexibility,cochran2014flexibility}. In the context of sensitivity analysis, this can be interpreted as the ability to remain feasible without changing the solution too aggressively when the system is subject to data perturbations. The justification is that the big jump in the solution may force the system to violate the constraints. Thus, this notion of flexibility is related to the smallest nonzero singular value of the active constraint Jacobian. Intuitively, the first qualitative condition helps the decay of sensitivity because positive curvature produces a direction to which the solution tends and the second qualitative condition helps the decay of sensitivity as it enables the solutions to dampen the impact of perturbations. These conditions can be related to specific properties of particular problem classes; for example, for the dynamic optimization problems analyzed in \cite{Na2019Exponential}, it can be seen that uniform LICQ is related to uniform controllability; similarly, the observability is directly related to SSOSC for state and parameter estimation problems \cite{zavala2010stability}. Establishing such connections for specific problem instances is an interesting topic of future work.

\section{Numerical Studies}\label{sec:num}
In this section, we illustrate the theoretical developments with different classes of graph-structured problems. We conduct case studies for four different classes of graph-structured optimization problems: dynamic optimization, stochastic optimization, PDE-constrained optimization, and network optimization. We are particularly interested in exploring the effect of the margin of uSSOSC (positive objective curvature) and uLICQ (flexibility) on the exponential decay of sensitivity. In what follows, we describe the particular problem instances under study. Throughout the instances, $a$ and $b$ are parameters that control the margin of uSSOSC and uLICQ, respectively, and $j$ represents the node where the data perturbation is introduced.

\subsection{Dynamic Optimization}
The study is performed for Example \ref{eg:ocp}. We set $w_i=d_i=\pi_i=0$ for $i\in \cV_{(k)}$, $k=9$, and $j=5$. Numerical sensitivity study with more sophisticated dynamic optimization problems (e.g., with nonlinearities) are provided in \cite{shin2019parallel,na2020overlapping}. 

\subsection{Stochastic Optimization}
We consider a stochastic program:
\begin{subequations}\label{eqn:sp}
  \begin{align}
    \min_{\{s_i,u_i,v_i\in\mathbb{R}\}_{i\in \cV}} \;&\sum_{i\in \cV} \pi_i\cdot(a \frac{1}{2}s_i^2 + \frac{1}{2}u_i^2+\chi_i v_i)\\
    \st\;& s_1 = \overline{s},\;(\mu_1)\\
    &s_i = s_{an(i)}+bu_{an(i)}+w_i,\;i\in \cV\setminus\{1\},\;(\mu_i),\\
    &v_i = u_i + d_i,\;i\in \cV,\;(\nu_i).
  \end{align}
\end{subequations}
Here $\cG=(\cV,\cE)$ represents the scenario tree (depicted in Figure \ref{fig:graphs}); $1\in \cV$ is the root node; $an(i)\in \cN_\cG[i]$ denotes the parent node; $\pi_i\in\mathbb{R}_{\geq 0}$ denotes the probability of node $i$; $s_i\in\mathbb{R}$ is the stored energy at node $i$; $u_i\in\mathbb{R}$ is the charge/discharge of energy at node $i$; $v_i\in\mathbb{R}$ is the transactions with grid at node $i$; $\overline{s}=w_1$ is the initial energy storage; $\chi_i\in\mathbb{R}$ is the forecast energy price at node $i$; $d_i$ is the energy demand forecast at node $i$; $w_i$ is the disturbance forecast at node $i$. We set $x_i:=[s_i,u_i,v_i]$ as the primal variable vector at node $i$; $y_i:=[\mu_i,\nu_i]$ as the dual variable vector at node $i$; $p_i=[\chi_i;w_i;d_i]$ as the data at node $i$. Furthermore, we set $|c(i)|=3$, where $c(i)\subseteq \cN_\cG[i]$ is the set of children nodes, $w_1=0$; $\{w_j\}_{j\in c(i)}=[-1;0;1]$; $d_i=\chi_i=0$ for $i\in \cV$; $k=6$, where $k$ denotes the number of stages. We choose $j=3$ (a node in the second stage).

\subsection{PDE Optimization}
The study is performed for Example \ref{eg:pde} with $k=9$ ($9\times 9$ grid). We choose $j$ to be the node at the center $(5,5)$. For the base data, we set $s_{\text{ref}}(w)=d(w)=\overline{T}$ for $w\in \Omega$. 

\subsection{Network Optimization}
We consider the alternating current (AC) optimal power flow problem:
\begin{subequations}\label{eqn:opf}\small
  \begin{align}
    \min_{\substack{\{v_i\in\mathbb{C}\}_{i\in \cV}\\
        \{s^g_{k}\in\mathbb{C}\}_{k\in \mathcal{W}}\\
        \{s_{ij}\in\mathbb{C}\}_{i,j\in \cV}}}\;&
    a \Big(\sum_{i\in \cV}(|v_i|-v_\text{ref})^2 +\sum_{\{i,j\}\in \cE} (\angle v_iv_j^*)^2\Big)+
    \sum_{k\in \mathcal{W}}c^1_{(k)} \Re(s^g_{(k)}) + c^2_{(k)} \Re(s^g_{(k)})^2 \label{eqn:opf-obj}\\
    \st\;& \angle v_i = 0,\; i\in \cV_{\text{ref}}\\
    &s^{gL}_{(k)}-b(1+\sqrt{-1})\leq s^g_{(k)}\leq s^{gU}_{(k)}+b(1+\sqrt{-1}),\;k\in \mathcal{W}\label{eqn:opf-con-gen}\\
    &\sum_{k\in \mathcal{W}_i}s^g_{k} -s^d_{i} = \sum_{j\in \cN_\cG[i]} s_{ij},\;v^L_i\leq |v_i|\leq v^U_i,\; i\in \cV\\
    &s_{ij} = (Y_{ij}+Y_{ij}^c)^* \frac{|v_i|^2}{|T_{ij}|} v_i^*-Y_{ij}\frac{v_iv_j^*}{T_{ij}} ,\;|s_{ij}| \leq s^U_{ij},\; i,j\in \cV\\
    &\theta^{\Delta L}_{ij} \leq \angle v_iv_j^*\leq \theta^{\Delta U}_{ij},\; \{i,j\}\in \cE.
  \end{align}
\end{subequations}
Here, $\cG=(\cV,\cE)$ represents the power network, $\mathbb{C}$ denotes the set of complex numbers; $\Re(\cdot)$ and $\Im(\cdot)$ denotes the real and imaginary part of the argument; $(\cdot)^*$ denotes the complex conjugate of the argument; $p\geq q \iff \Re(p)\geq \Re(q)$ and $\Im(p)\geq \Im(q)$ for $p,q\in\mathbb{C}$; $\mathcal{W}_i$ is the set of generators connected to node $i$; $\mathcal{W}:=\bigcup_{i\in \cV}\mathcal{W}_i$; $\cV_{\text{ref}}$ is the set of reference nodes; $v_i\in\mathbb{C}$ is the voltage at node $i$; $s^g_{(k)}\in\mathbb{C}$ is the power generation at generator $k$; $\{v^L_i,v^U_i,s^d_i\in\mathbb{C}\}_{i\in \cV}$ ,$\{\theta^{\Delta,L}_{ij},\theta^{\Delta,U}_{ij}\in\mathbb{R}\}_{\{i,j\}\in \cE}$, $\{s_{ij}^{U},Y_{ij},Y^c_{ij},T_{ij}\in\mathbb{C}\}_{i,j\in \cV}$, $\{c^1_{(k)},c^2_{(k)}\in\mathbb{R},s^{gL}_{(k)},s^{gU}_{(k)}\in\mathbb{C}\}_{k\in \mathcal{W}}$ are the data. The readers are pointed to the documentation of {\tt PowerModels.jl} \cite{coffrin2018powermodels} for the details of Problem \eqref{eqn:opf}. Here we modified the problem by adding the regularization term (the first term in \eqref{eqn:opf-obj}) and by introducing the additional terms in constraint \eqref{eqn:opf-con-gen} to examine the effect of positive curvature and flexibility in the constraints; the problem reduces to the original problem when $(a,b)=0$. We treat the edge variables, constraints, and the objective terms by treating them as node terms for one of the connected nodes (in particular, the one with lower node index), as explained in Remark \ref{rem:gen-1}; note that this manipulation only alters indexing and does not change the problem. We set $z_i$ as all the primal/dual variable that are associated with node $i\in \cV$ (including generator and edge variables/constraints), and we set $p_i=[\Re(s^d_i),\Im(s^d_i)]$. We choose $j=1$ as the perturbation location. We use test case {\tt pglib\_opf\_case500\_tamu} available at {\tt pglib-opf} v18.08 \cite{babaeinejadsarookolaee2019power,birchfield2016grid} (the problem data $c^1_{(k)}$, $c^2_{(k)}$, $V^L_i$, etc are available therein). The problem is modeled using modeling library {\tt PowerModels.jl}.

\subsection{Methods}
We conduct the following numerical study for each problem instance. We consider a problem $P(\bp^\star)$ with the base data $\bp^\star$. Then, we consider perturbed problems $\{P(\bp^{^{(q)}})\}_{q=1}^{\overline{q}}$ in which the data are perturbed as $\bp^{(q)}=\bp^\star+ \Delta\bp^{(q)}$, where $\Delta\bp^{(q)}$ are i.i.d samples drawn from $\Delta p_j \sim\mathcal{U}([-\sigma,\sigma]^{l_j})$, and $\Delta p_i=\bzero$ if $i\neq j$. Here, $j\in \cV$ is a selected perturbation point and $\mathcal{U}(\Omega)$ denotes the multivariate uniform distribution on $\Omega$. We choose $\sigma=10^{-3}$ and $\overline{q}=30$ for all instances. Then, the {\it empirical} sensitivity coefficients:
\begin{align*}
  \overline{C}_{ij} := \max_{q\in\mathbb{I}_{[1,\overline{q}]}} \Vert z^\ddag_i(\bp^{(q)})\Vert/\Vert \Delta \bp^{(q)}\Vert,\quad i,j\in \cV 
\end{align*}
are computed and visualized. The empirical sensitivity $\overline{C}_{ij}$ converges to $\Vert \nabla_{p_j} z^\dag_i(\bp^\star)\Vert$ as $\sigma\rightarrow 0$ and $\overline{q}\rightarrow\infty$; thus, these empirical sensitivities are suitable quantities for study of sensitivity coefficients. We recall that $(a,b)$ are the key parameters that control the positive curvature and flexibility. We vary these parameters as shown in Table \ref{tbl:etab}, and see how they affect the spread of the sensitivity coefficients. Here, Case $1$ has sufficiently large $(a,b)$; Case 2 has low $a$; Case 3 has low $b$; and Case 4 has low $(a,b)$. The results can be reproduced using the scripts provided in \url{https://github.com/zavalab/JuliaBox/tree/master/SensitivityNLP}.

\begin{table}\centering\small
  \caption{Variation of $(a,b)$ in numerical studies.}\label{tbl:etab}
  \begin{tabular}{|c|c|c|c|c|}
    \hline
    &Case 1& Case 2& Case 3& Case 4\\
    \hline
    Dynamic Optimization & $(1,1)$ & $(10^{-2},1)$ & $(1,10^{-2})$ & $(10^{-2},10^{-2})$\\
    \hline
    Stochastic Optimization & $(1,1)$ & $(10^{-2},1)$ & $(1,10^{-2})$ & $(10^{-2},10^{-2})$\\
    \hline
    PDE Optimization& $(1,1)$ & $(10^{-2},1)$ & $(1,10^{-2})$& $(10^{-2},10^{-2})$\\
    \hline
    Network Optimization &
    $(10^6,10)$ & $(0,10)$ & $(10^6,0)$& $(0,0)$\\
    \hline
  \end{tabular}  
\end{table}

\subsection{Results}
The sensitivity results are illustrated as heat maps of the empirical coefficients (Figure \ref{fig:heatmap}) and as scatter plots of the coefficients against distance $d_\cG(i,j)$ (Figure \ref{fig:scatter}). From Figure \ref{fig:heatmap} we see that, with sufficiently large $(a,b)$ (Case 1), the empirical sensitivity coefficients decay as they move away from the perturbation point. Furthermore, from Figure \ref{fig:scatter}, one can confirm that the sensitivity coefficients decay exponentially with distance. This verifies the theoretical results in Section \ref{sec:sens}. If either one or both of $(a,b)$ are not sufficiently large (Case 2, 3, 4), the decay of sensitivity is weaker or not observed (except for the PDE optimization problem). The reason that the PDE optimization problem exhibits decay of sensitivity even in the absence of positive curvature and flexibility is that the system itself has a strong dissipative property. From these results, we can confirm that it is sufficient for problems to have strong positive curvature and flexibility in the constraints to exhibit decay of sensitivity (this confirms the theoretical results in Section \ref{sec:uniform}). Notably, even though we cannot guarantee uniform boundedness of the multi-stage stochastic programs, we can observe EDS for sufficiently large $(a,b)$. 

\begin{figure}[!htb]\centering
    \begin{tikzpicture}
      \node at (0,0){
        \setlength{\tabcolsep}{0.25em}
      \begin{tabular}{cccc}
        Case 1& Case 2& Case 3& Case 4\\
        \includegraphics[width=.23\textwidth]{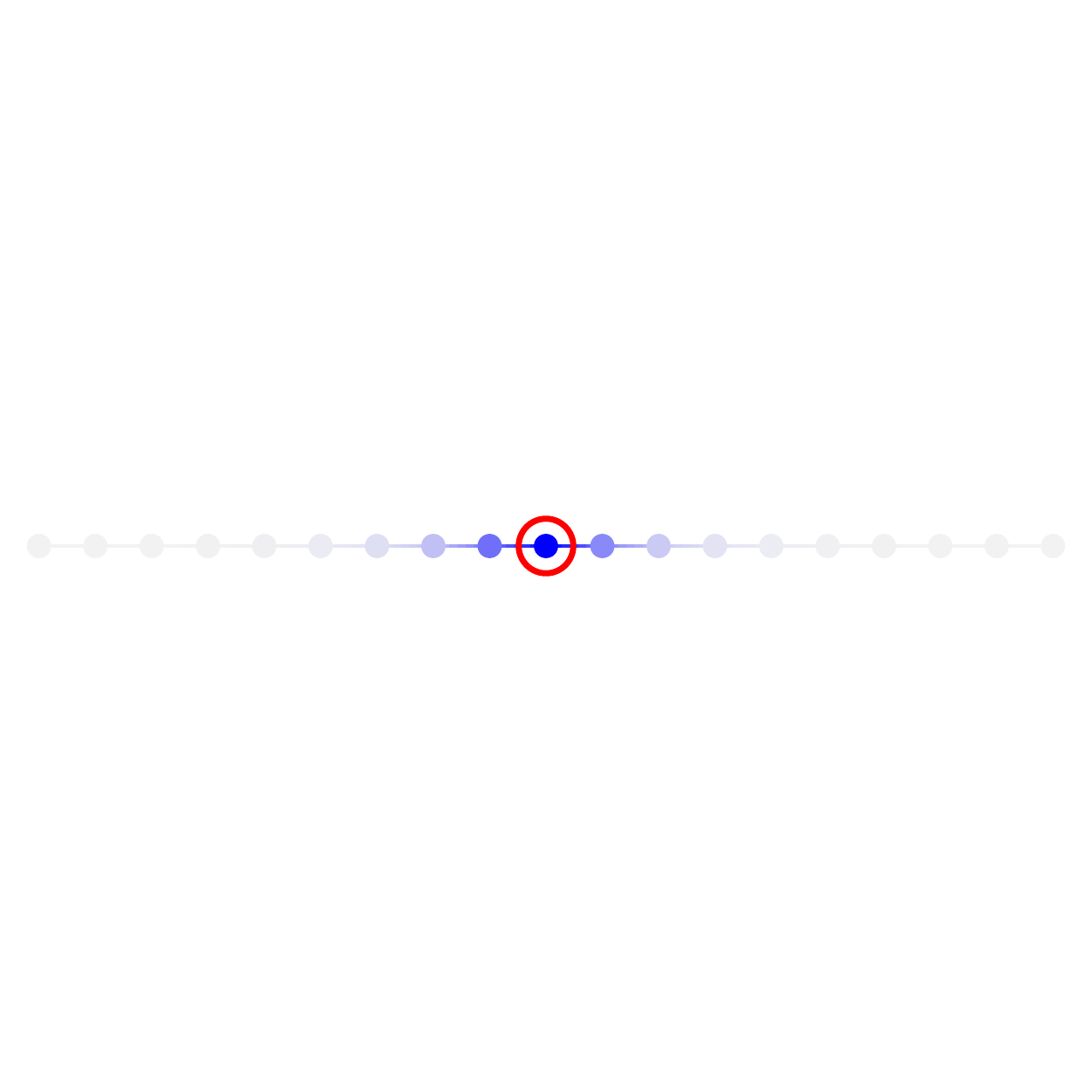}&
        \includegraphics[width=.23\textwidth]{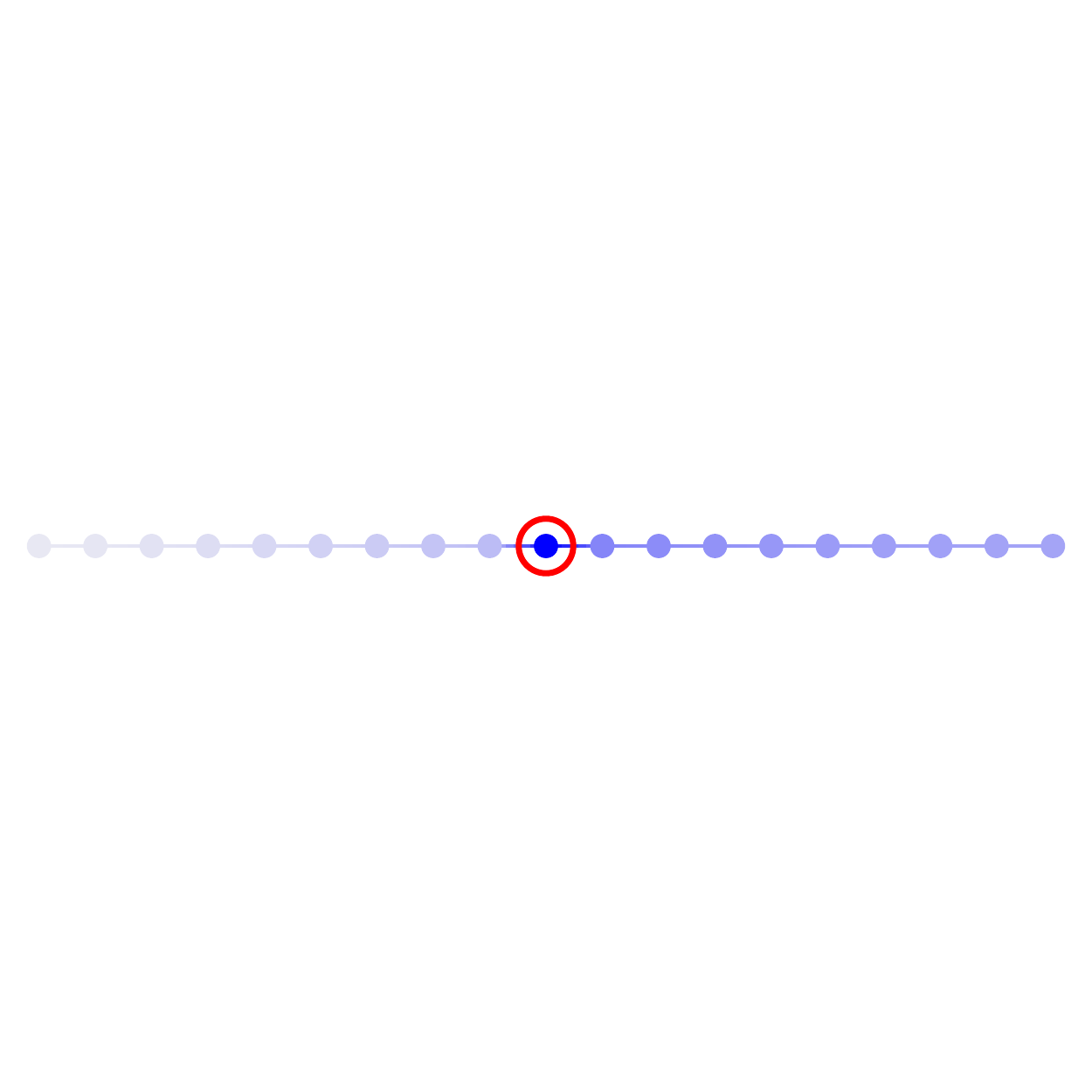}&
        \includegraphics[width=.23\textwidth]{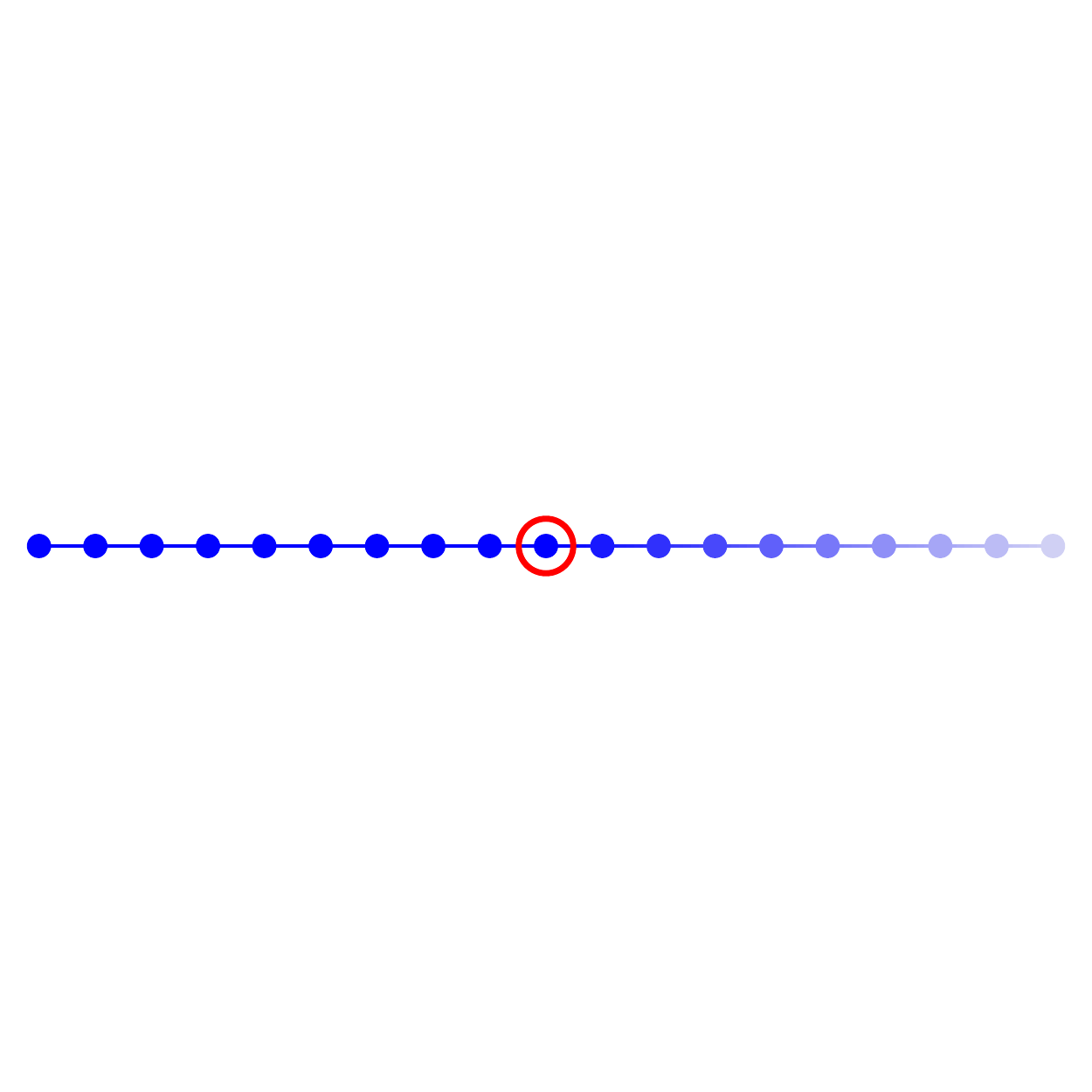}&
        \includegraphics[width=.23\textwidth]{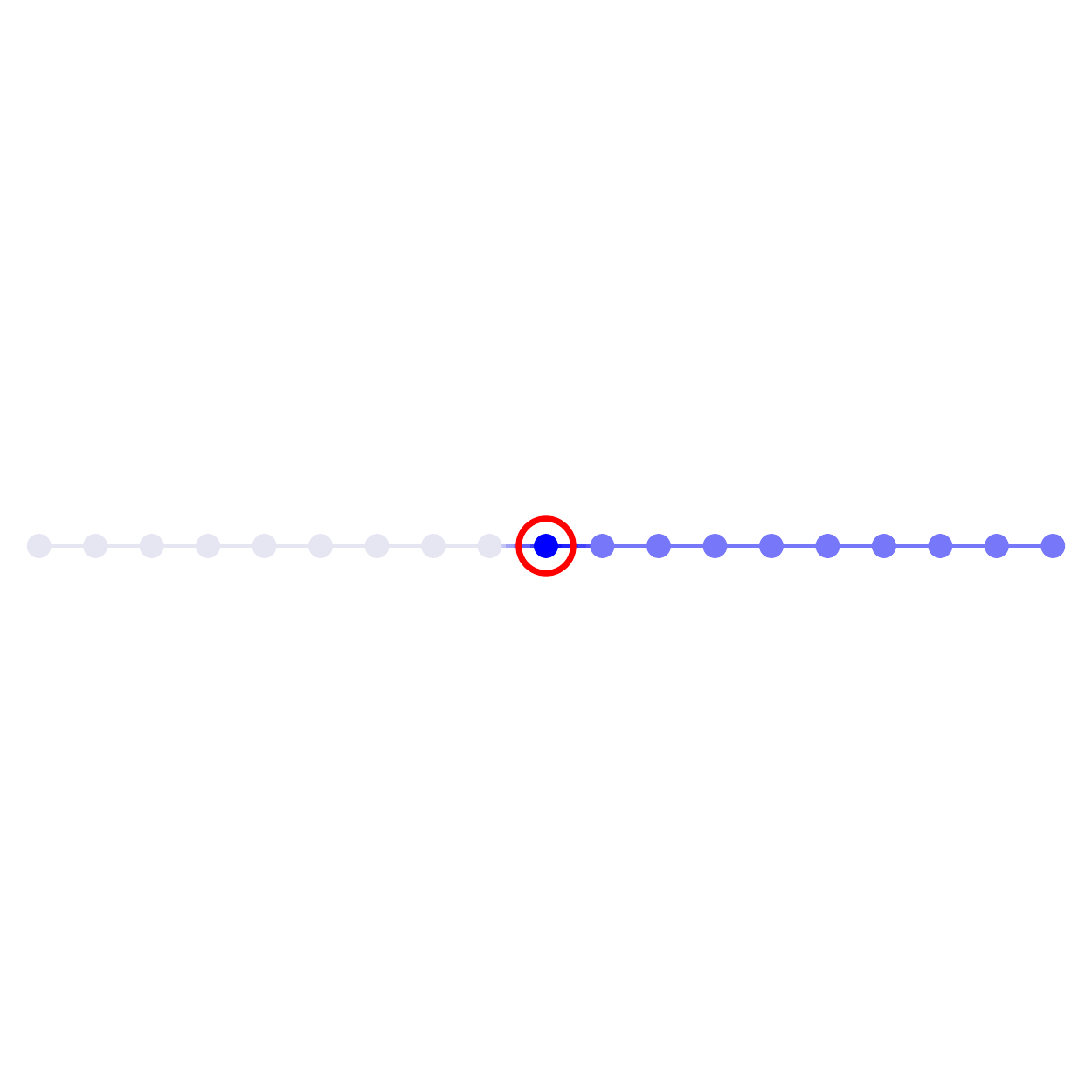}\\
        \includegraphics[width=.23\textwidth]{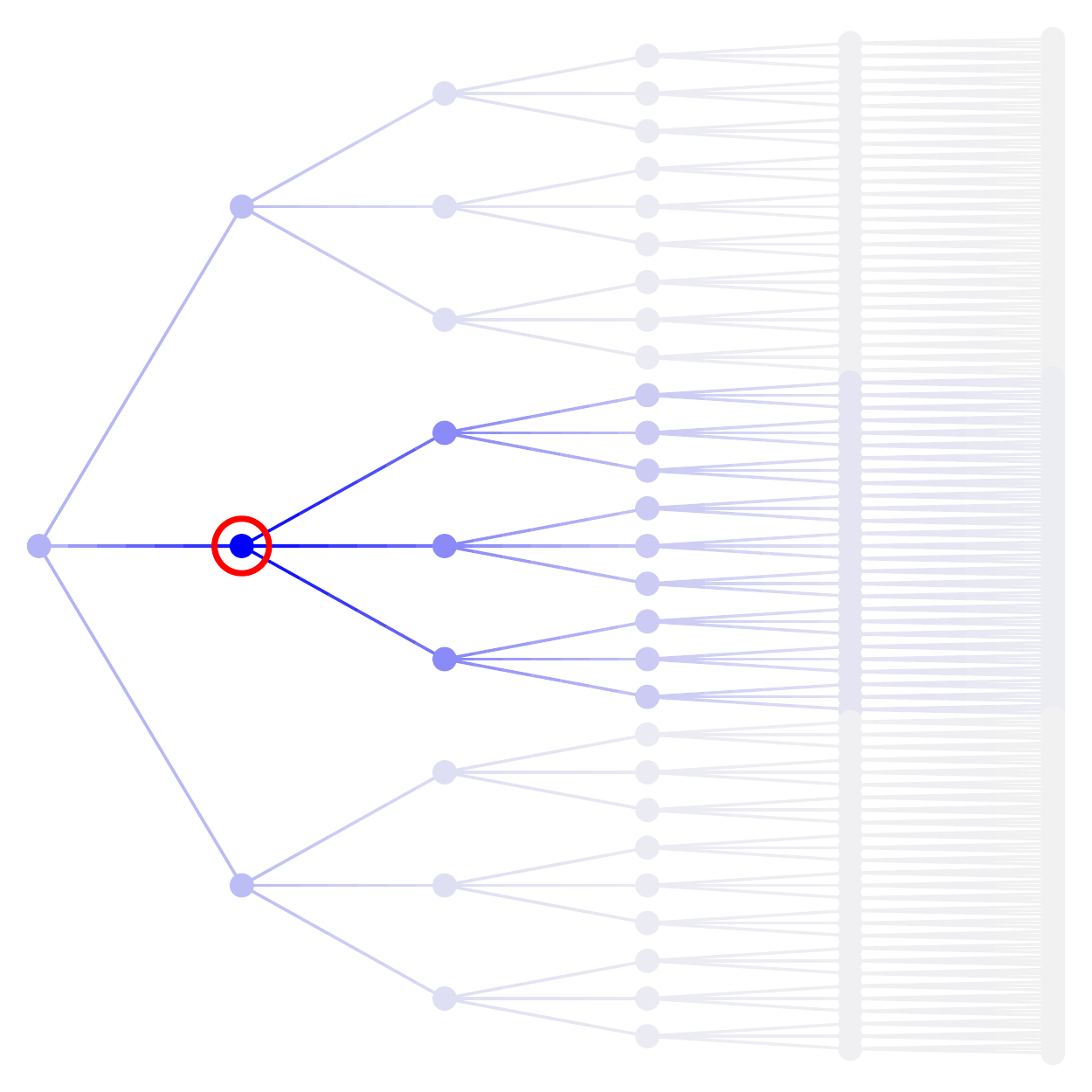}&
        \includegraphics[width=.23\textwidth]{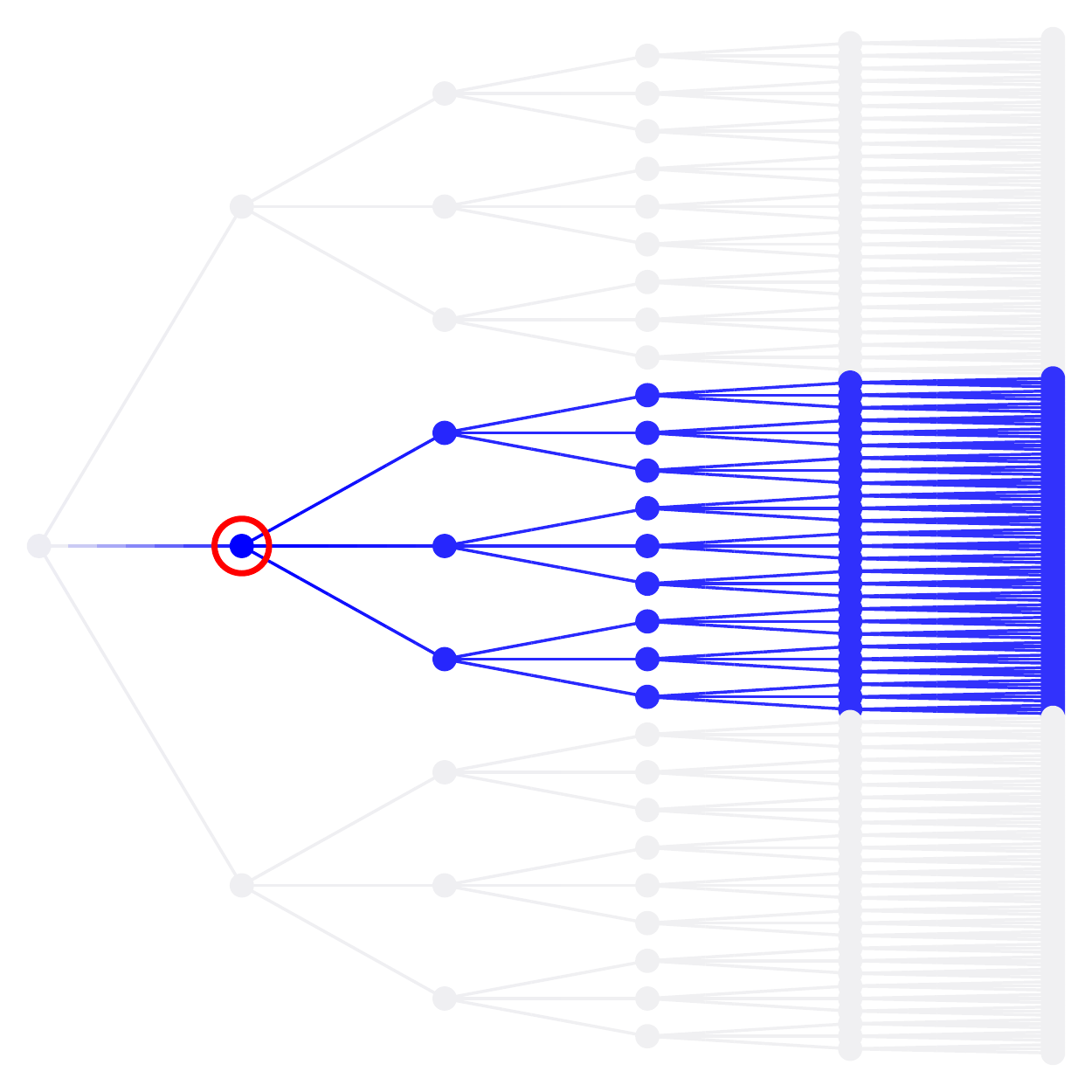}&
        \includegraphics[width=.23\textwidth]{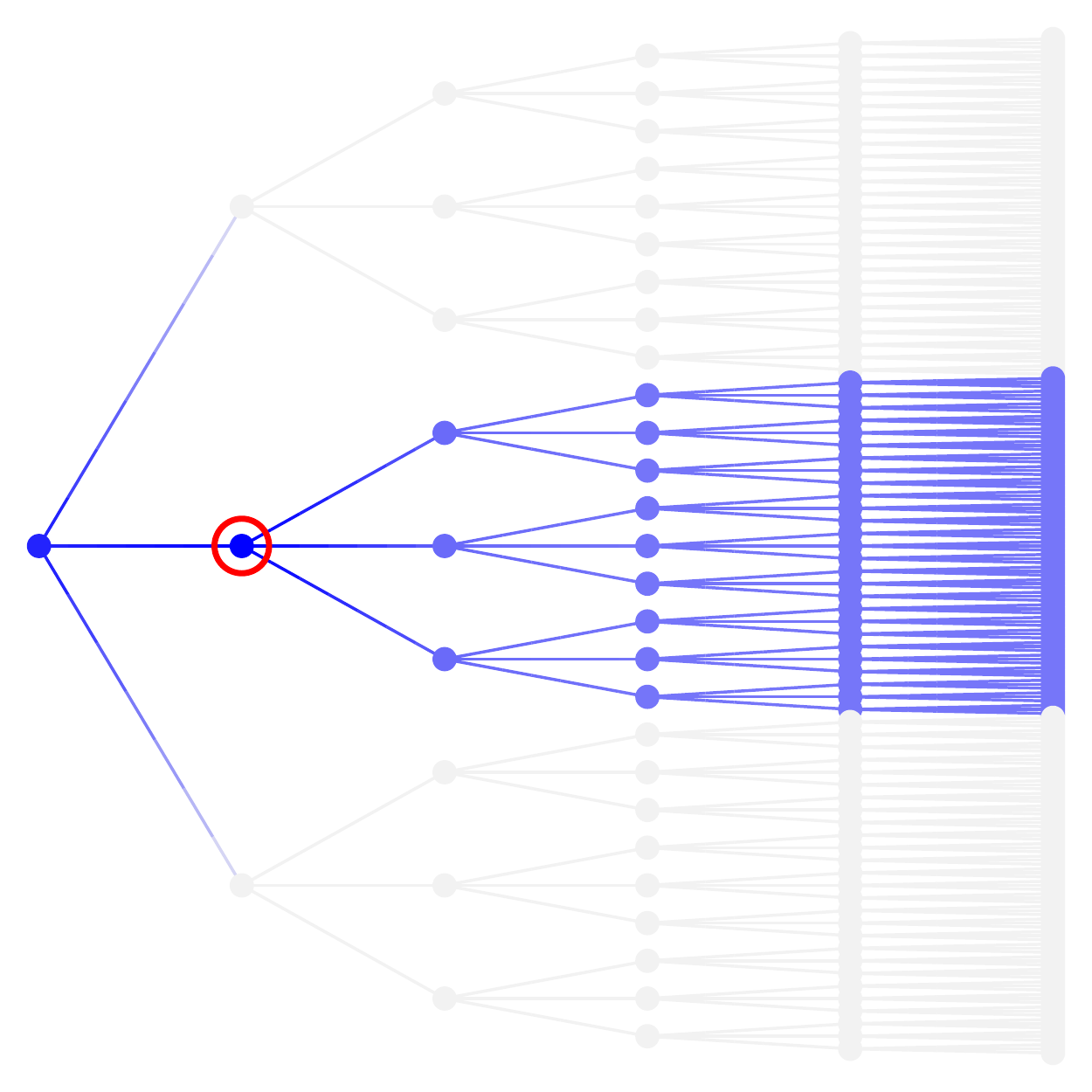}&
        \includegraphics[width=.23\textwidth]{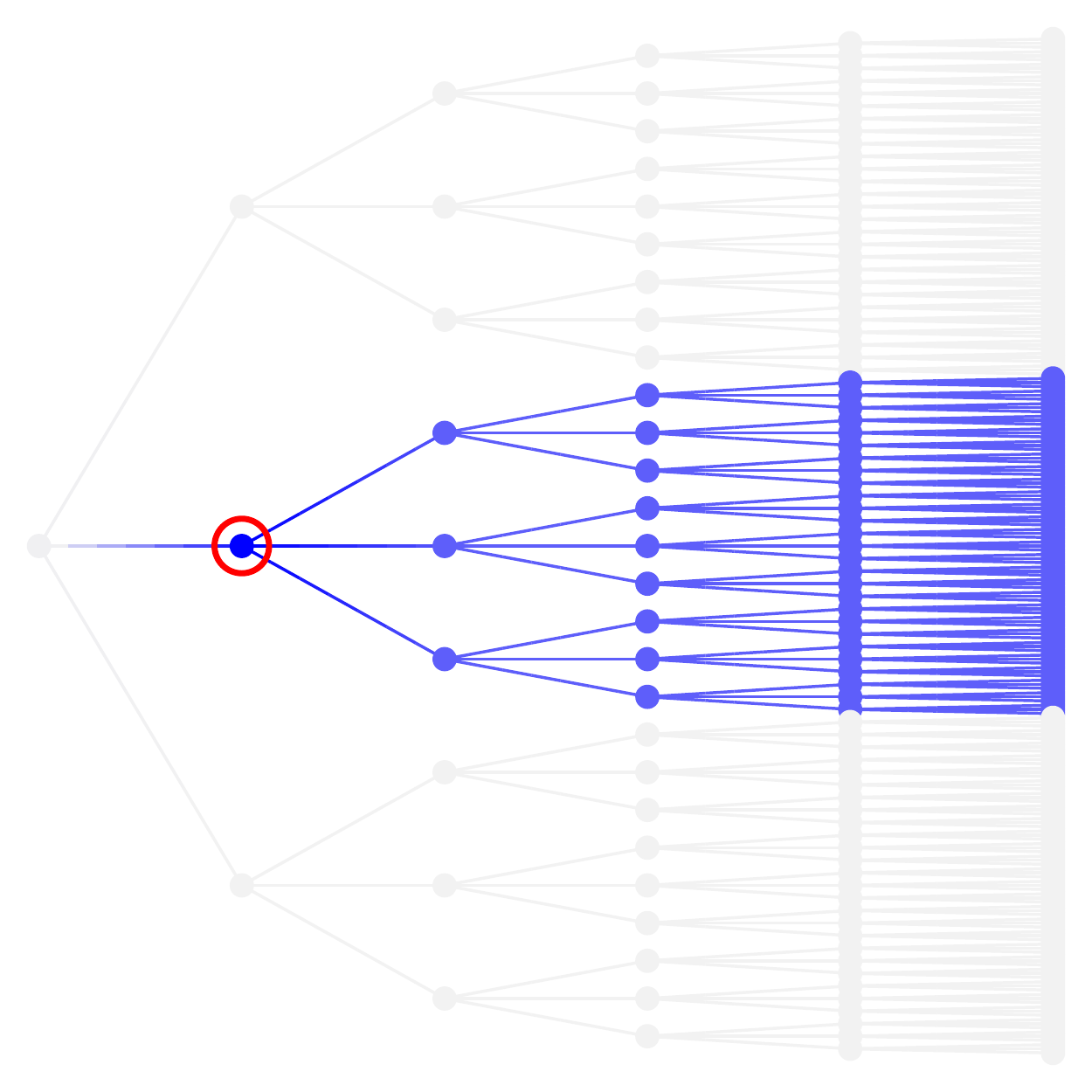}\\  
        \includegraphics[width=.23\textwidth]{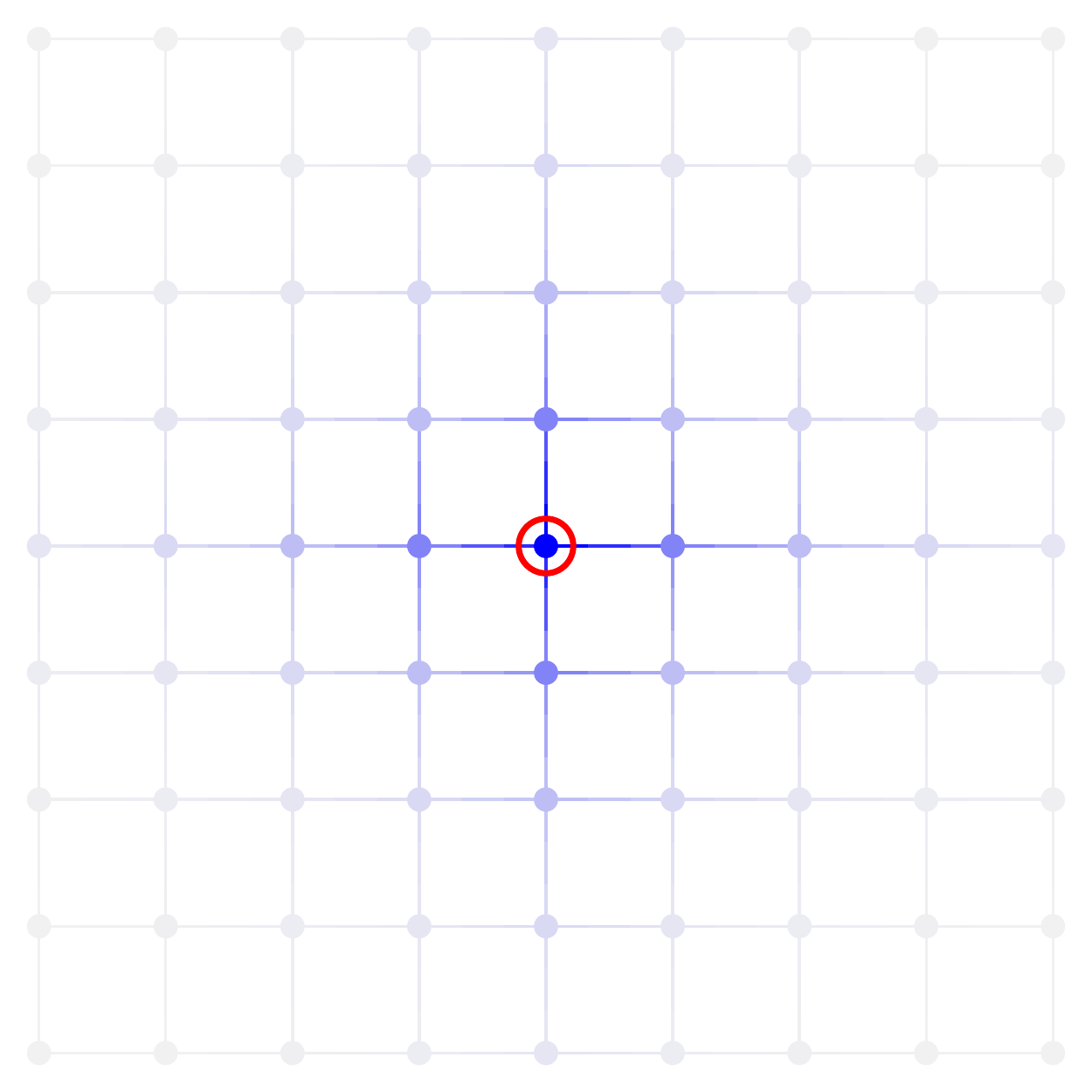}&
        \includegraphics[width=.23\textwidth]{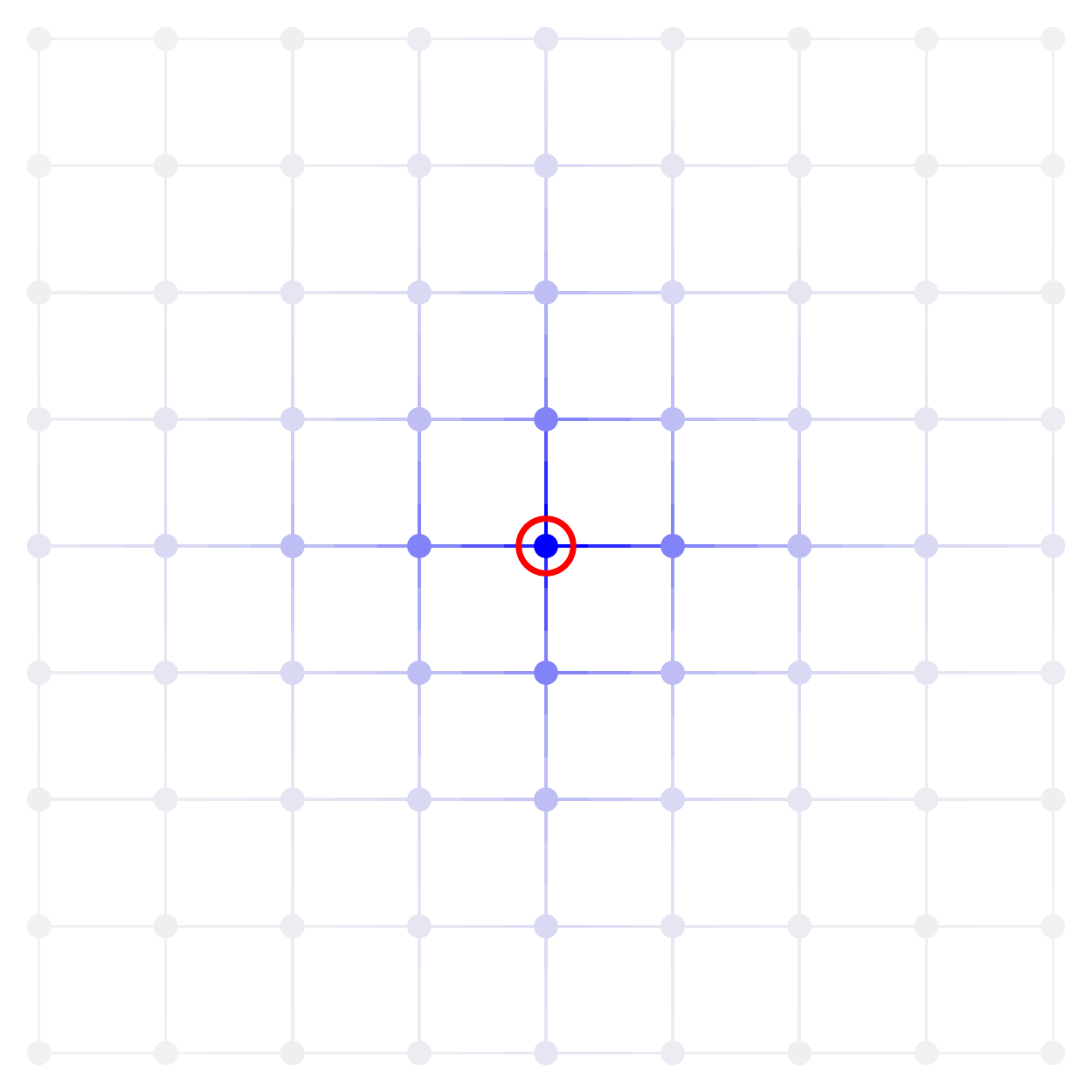}&
        \includegraphics[width=.23\textwidth]{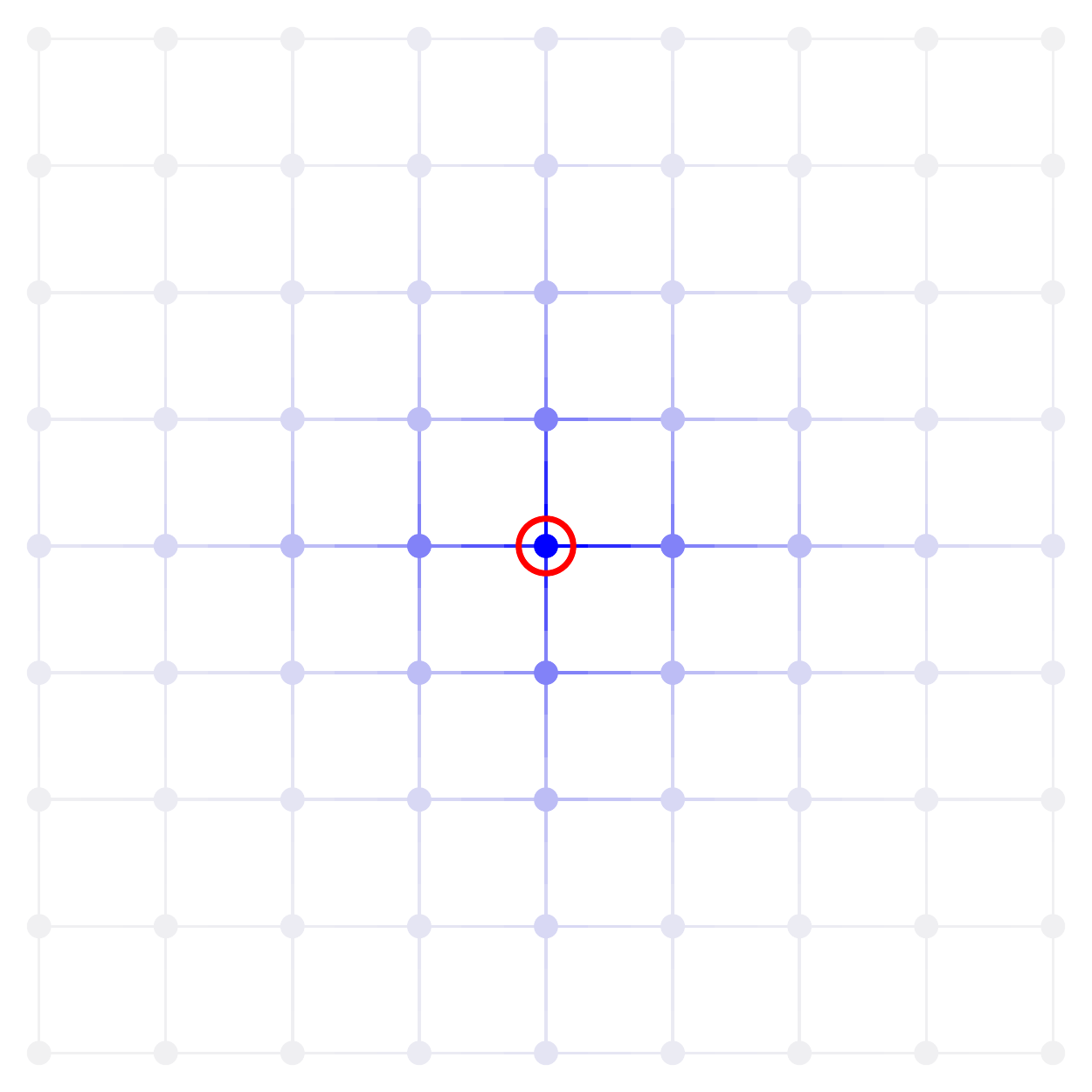}&
        \includegraphics[width=.23\textwidth]{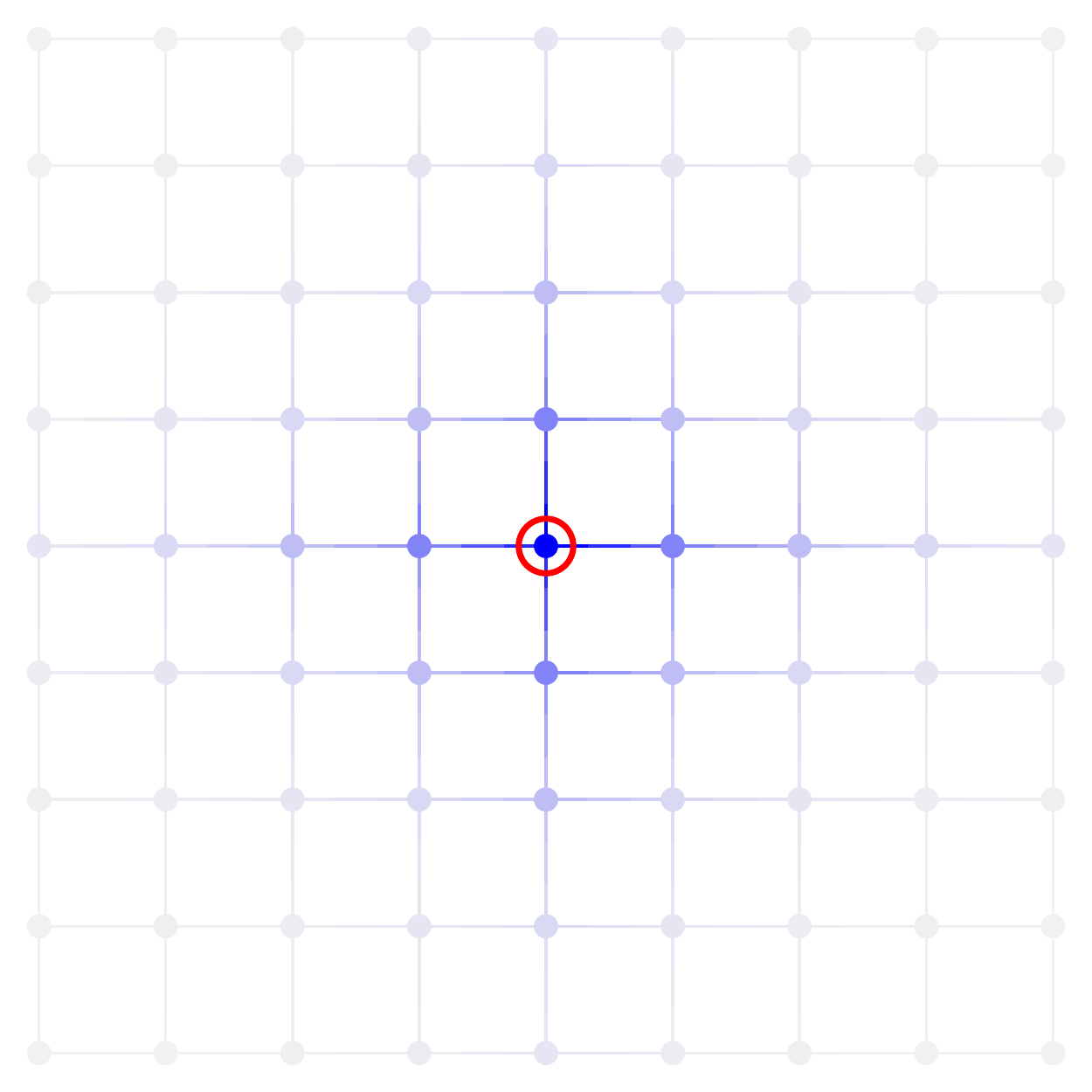}\\  
        \includegraphics[width=.23\textwidth]{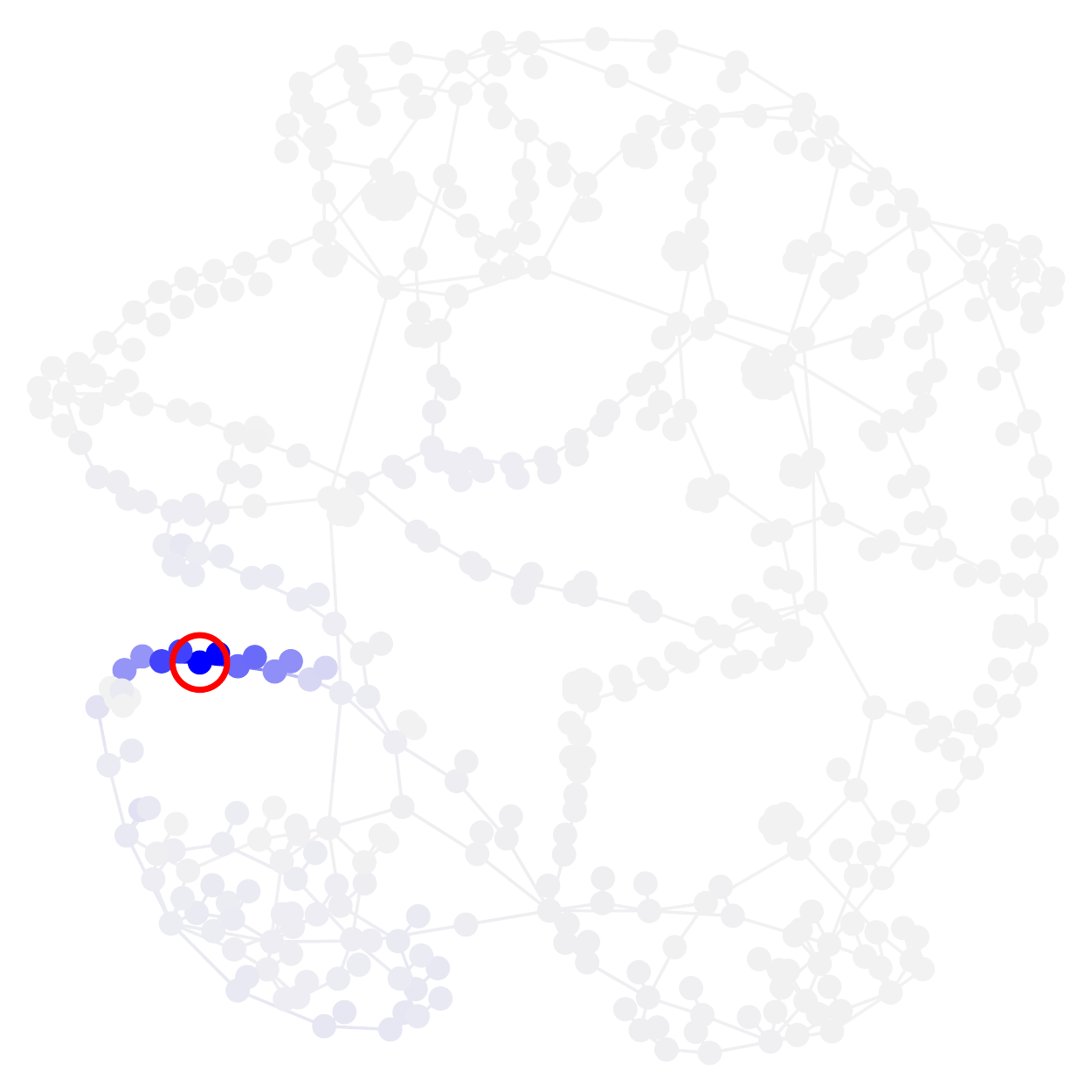}&
        \includegraphics[width=.23\textwidth]{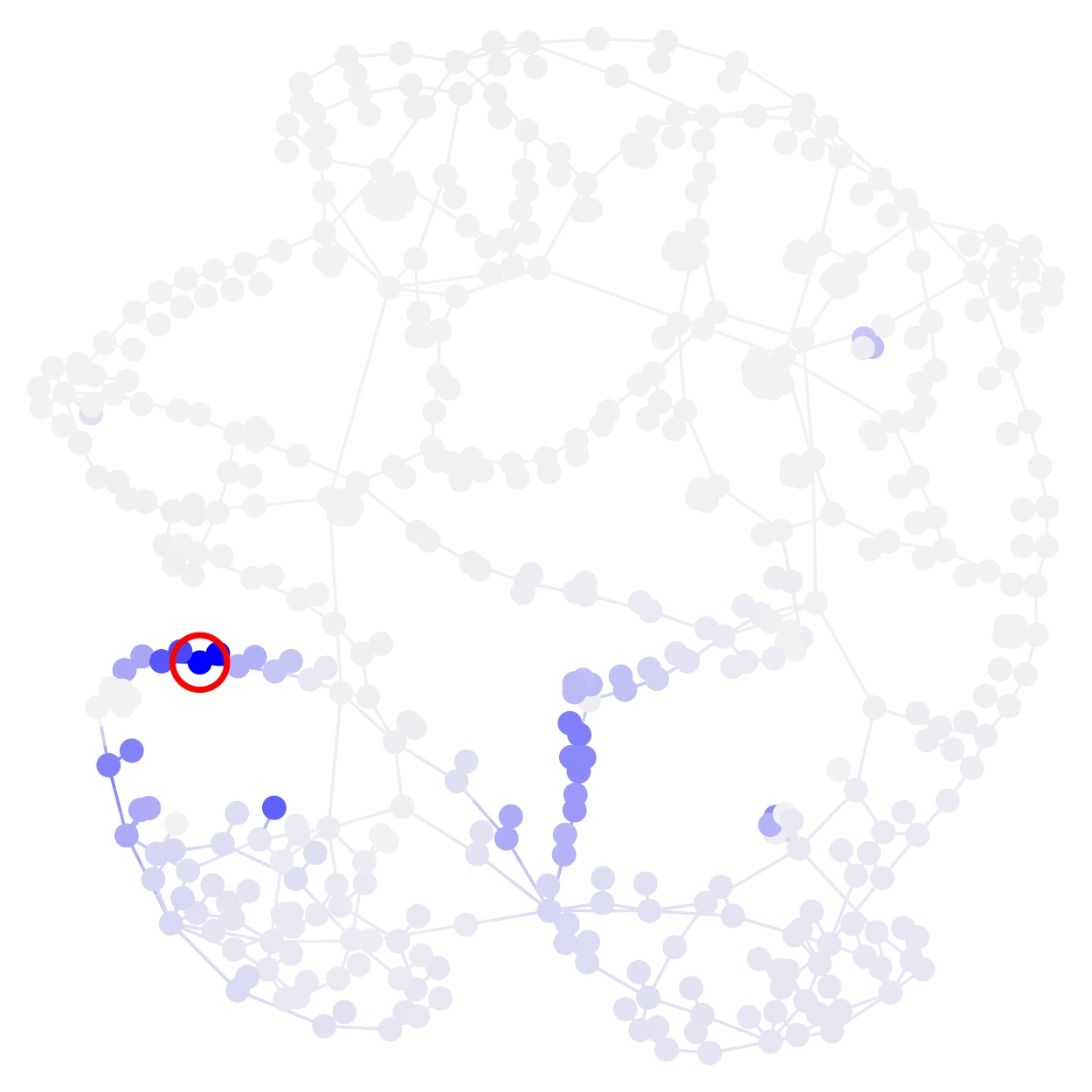}&
        \includegraphics[width=.23\textwidth]{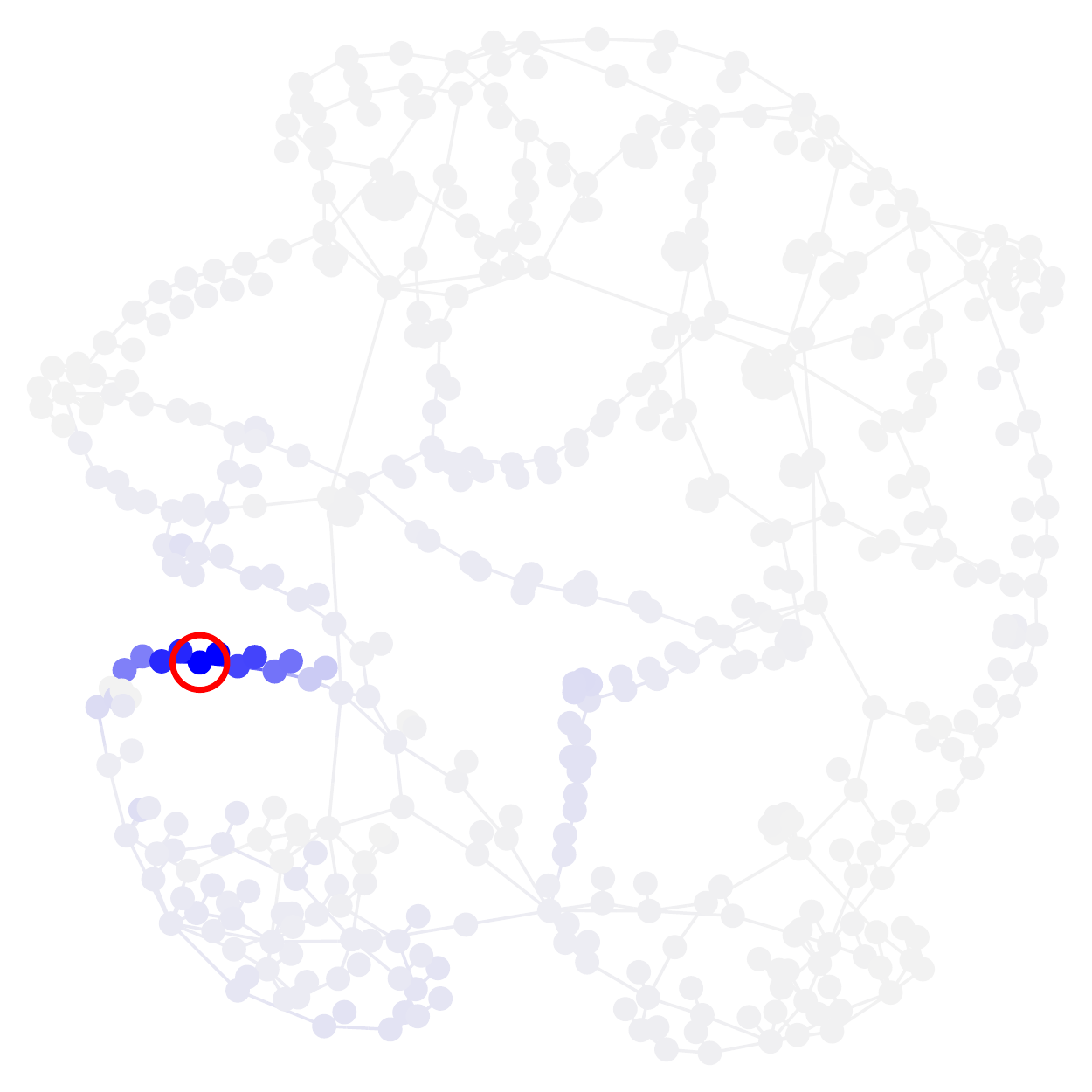}&
        \includegraphics[width=.23\textwidth]{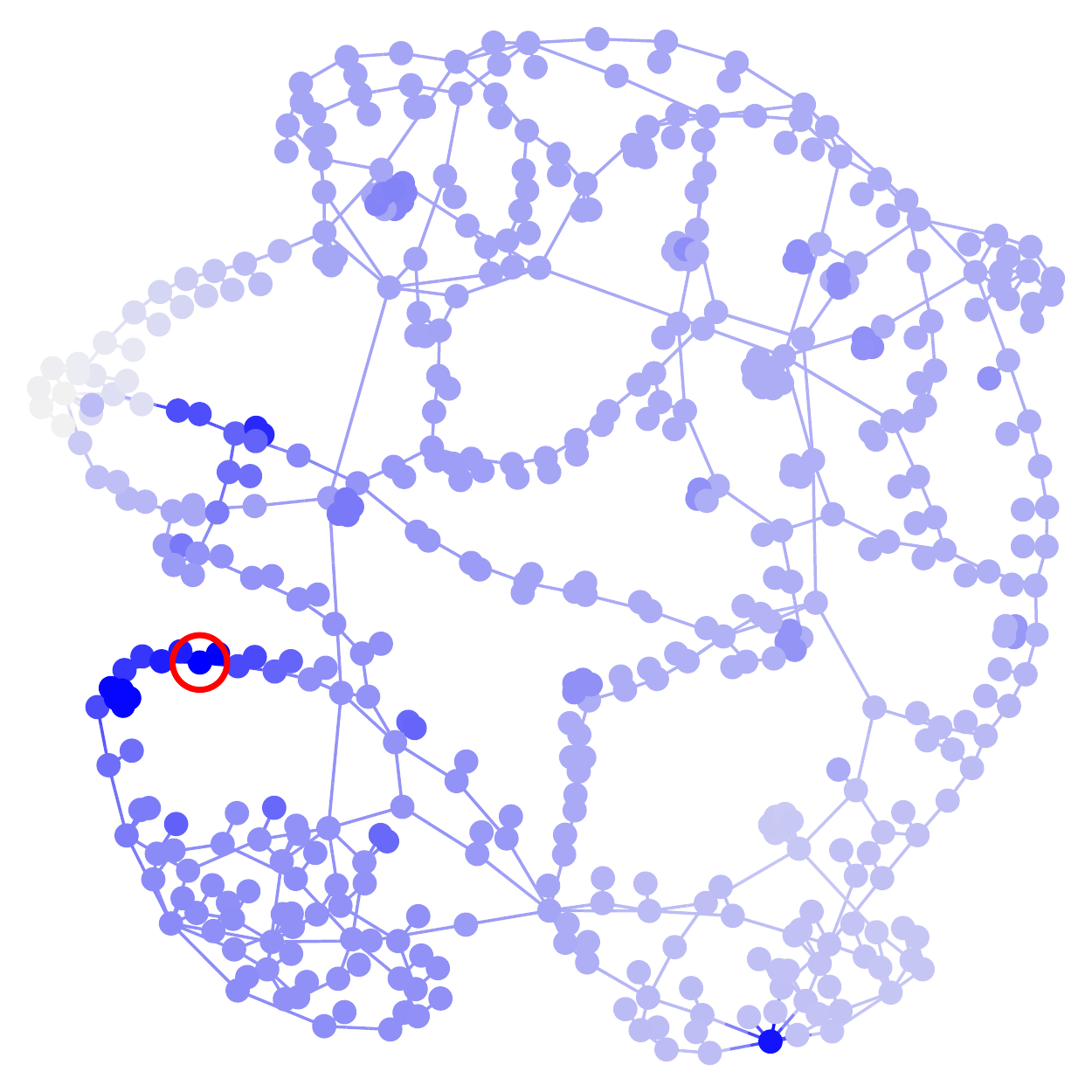}
    \end{tabular}};
    \end{tikzpicture}
    \caption{Spread of empirical sensitivity coefficients $\overline{C}_{ij}/\overline{C}_{jj}$ on $\cG$ for dynamic optimization (top), stochastic optimization (second row), PDE optimization (third row), and network optimization (bottom) problem for different values of $(a,b)$. Red circles denote perturbation point, dark blue approaches one, and white approaches zero.}\label{fig:heatmap}
\end{figure}
\begin{figure}\centering
  \begin{tikzpicture}
    \node at (0,0){
      \begin{tabular}{@{}c@{}c@{}c@{}c@{}}
        Case 1& Case 2& Case 3& Case 4\\
        \includegraphics[width=.235\textwidth]{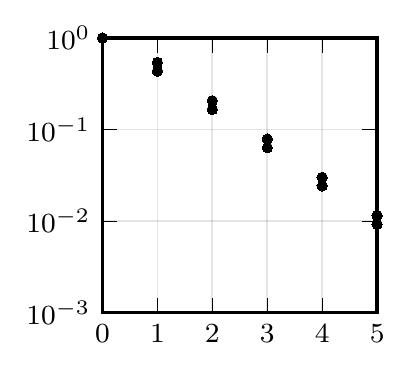}&
        \includegraphics[width=.235\textwidth]{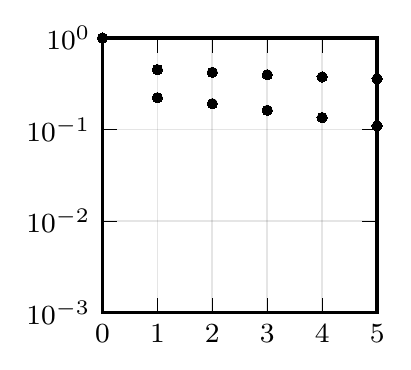}&
        \includegraphics[width=.235\textwidth]{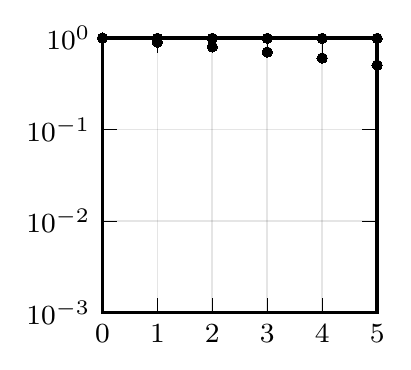}&
        \includegraphics[width=.235\textwidth]{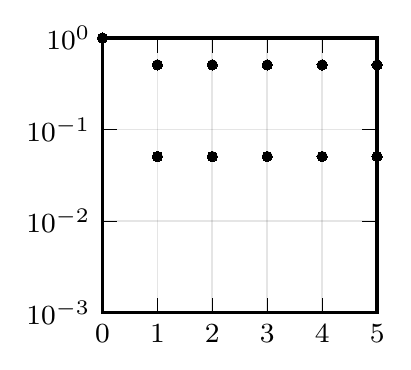}\\
        \includegraphics[width=.235\textwidth]{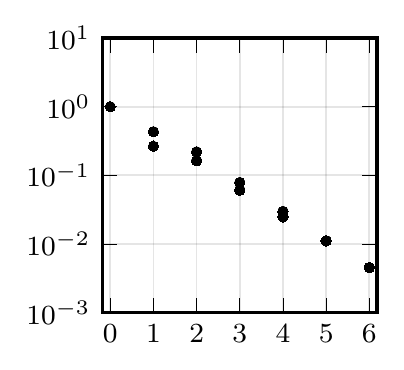}&
        \includegraphics[width=.235\textwidth]{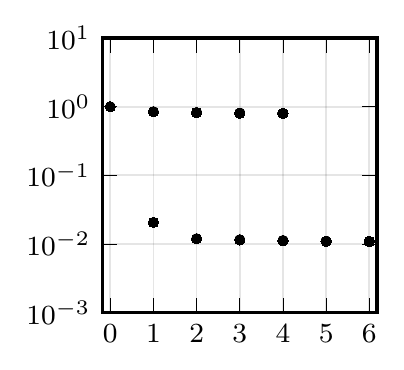}&
        \includegraphics[width=.235\textwidth]{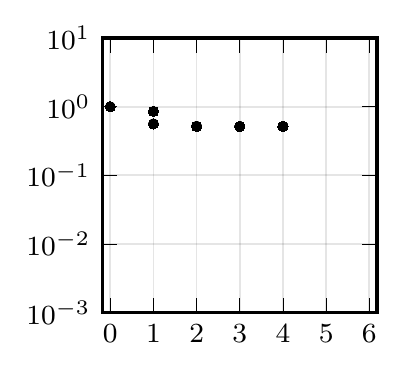}&
        \includegraphics[width=.235\textwidth]{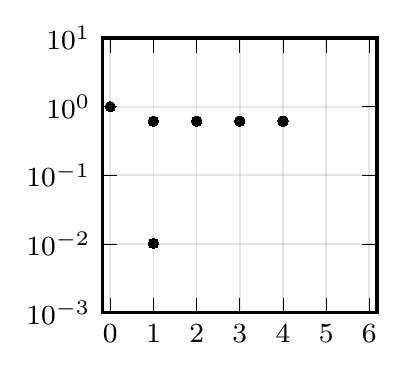}\\  
        \includegraphics[width=.235\textwidth]{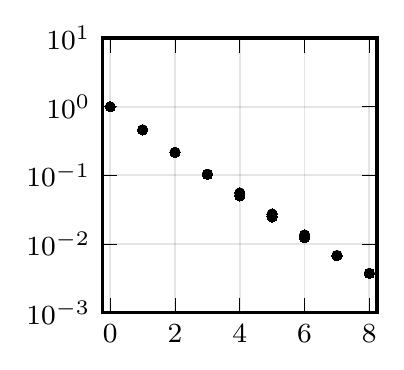}&
        \includegraphics[width=.235\textwidth]{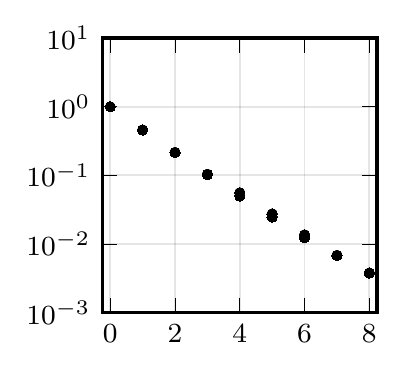}&
        \includegraphics[width=.235\textwidth]{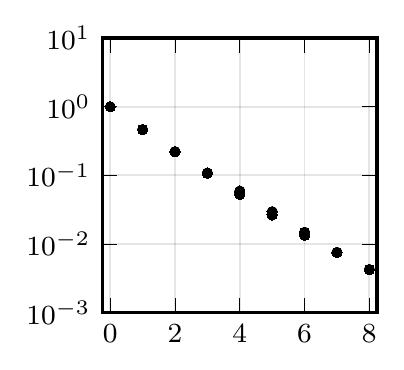}&
        \includegraphics[width=.235\textwidth]{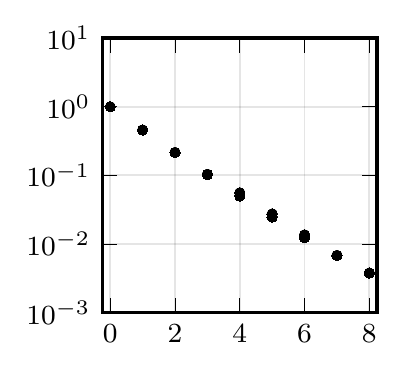}\\  
        \includegraphics[width=.235\textwidth]{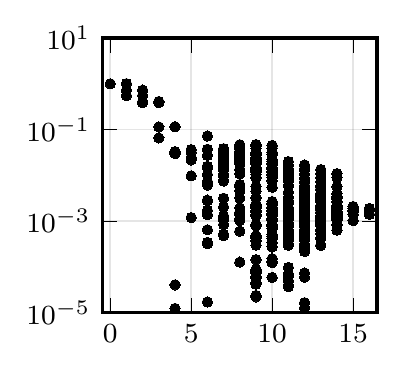}&
        \includegraphics[width=.235\textwidth]{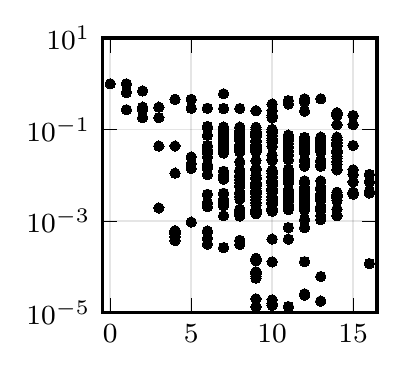}&
        \includegraphics[width=.235\textwidth]{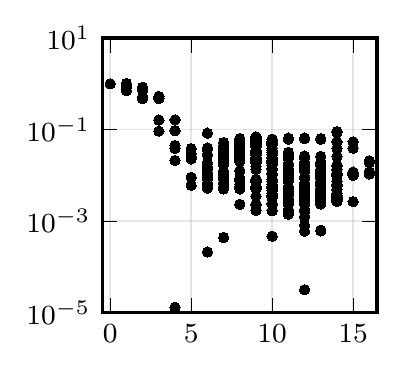}&
        \includegraphics[width=.235\textwidth]{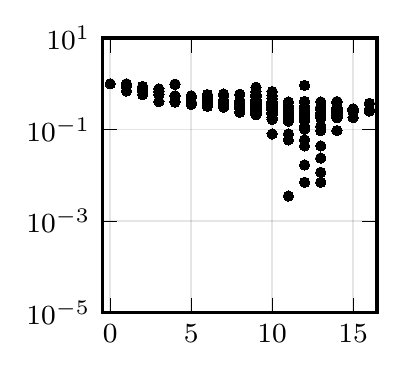}
    \end{tabular}};
    \node[rotate=90] at (-6.3,0) {$\overline{C}_{ij}/\overline{C}_{jj}$};
    \node at (0,-6.3) {$d_\cG(i,j)$};
  \end{tikzpicture}
  \caption{Scatter plots of sensitivity coefficients $\overline{C}_{ij}/\overline{C}_{jj}$ versus $d_\cG(i,j)$ for dynamic optimization (top), stochastic optimization (second row), PDE optimization (third row), and network optimization (bottom) for different values of $(a,b)$.}\label{fig:scatter}
\end{figure}

{Finally, we compare the empirically observed sensitivity decay rate with the theoretically calculated ones. We have computed $(\Upsilon,\rho)$ based on the upper bounds obtained in Theorem \ref{thm:main} with $\epsilon=0$ (note that $\epsilon$ cannot be exactly zero, but it can be chosen arbitrarily small). The result is shown in Table \ref{tbl:uprho}. Here, we have excluded network optimization, which has inequality constraints. For all cases, the decay rate $\rho$ is very close to $1$; clearly, these theoretical decay rates are much slower than the empirically observed decay rates. This implies that the theoretical upper bound of the decay rate is very conservative, and in practice, we have a much faster decay rate. Thus, our results in Section \ref{sec:sens}, \ref{sec:uniform} should be used for theoretical characterization rather than actual computations.
  \begin{table}\centering\footnotesize
    \caption{Theoretically calculated sensitivity decay parameters $(\Upsilon,1-\rho)$.}\label{tbl:uprho}
    \begin{tabular}{|c|@{}c@{}|@{}c@{}|@{}c@{}|@{}c@{}|}
      \hline
      &Case 1&Case 2&Case 3&Case 4\\
      \hline
      Dyn. &  $(26.1,0.0265)$&$(8910,8.88\times 10^{-5})$&$(59900,1.30\times10^{-5})$&$(349,0.0028)$\\
      \hline
      Stoch. & $(817000,7.09\times10^{-7})$&$(6.72\times10^6,9.01\times10^{-8})$&$(677000,1.03\times10^{-6})$&$(627000,1.20\times10^{-6})$\\
      \hline
      PDE&$(1.66,0.0275)$&$(0.721,0.0684)$&$(1.69,0.0269)$&$(0.612,0.0801)$\\
      \hline
    \end{tabular}  
  \end{table}
}

\section{Conclusions}\label{sec:conc}
We have presented sensitivity analysis for graph-structured nonlinear programs (problems whose structure is induced by a graph). Our results indicate that the sensitivity of the solution at a given location decays exponentially with respect to the distance to the perturbation point. This result holds under the strong second-order sufficiency condition (SSOSC) and the linear independence constraint qualification (LICQ). We show that sensitivity decay depends on the singular values of the submatrices of the Lagrangian Hessian matrix; as such, uniform boundedness conditions for such singular values are essential. {We have shown that the singular values can be uniformly bounded under uniform regularity conditions: uniformly bounded Lagrangian Hessian, uniform SSOSC, and uniform LICQ. Furthermore, we have shown that uniformly bounded Lagrangian Hessian can be obtained from uniformly bounded graph degree and second-order derivatives, and uniform SSOSC and LICQ can be guaranteed under uniform SSOSC and LICQ at problem blocks.} Qualitatively, these conditions can be interpreted as having sufficiently strong positive curvature in the objective and flexibility in the constraints. Our numerical studies (for dynamic optimization, stochastic optimization, network optimization, and PDE optimization) confirm that, if the graph-structured problem exhibits such properties, sensitivity indeed decays exponentially. As part of future work, we are interested in exploring the propagation of sensitivity in specific problem classes. A sensitivity decay property for continuous-time (infinite-dimensional) dynamic optimization problems has been recently established in \cite{grune2019sensitivity,grune2020exponential,grune2020abstract}. We are interested in studying the limiting behavior of our graph sensitivity results to establish a similar result for such problems. {Also, we are interested in establishing tighter local sensitivity bounds for inhomogeneous networks, which may arise in multi-scale/multi-physical systems.}

\section*{Acknowledgments}

We are grateful to the 3 anonymous referees whose comments have greatly improved the paper. This material is based upon work supported by the U.S. Department of Energy, Office of Science, Office of Advanced Scientific Computing Research (ASCR) under Contract DE-AC02-06CH11347 and by NSF through award CNS-1545046. We also acknowledge support from the Grainger Wisconsin Distinguished Graduate Fellowship.

\bibliographystyle{siamplain}
\bibliography{Sensitivity_NLP_final}
\vspace{0.1cm}
\begin{flushright}
	\scriptsize \framebox{\parbox{2.5in}{Government License: The
			submitted manuscript has been created by UChicago Argonne,
			LLC, Operator of Argonne National Laboratory (``Argonne").
			Argonne, a U.S. Department of Energy Office of Science
			laboratory, is operated under Contract
			No. DE-AC02-06CH11357.  The U.S. Government retains for
			itself, and others acting on its behalf, a paid-up
			nonexclusive, irrevocable worldwide license in said
			article to reproduce, prepare derivative works, distribute
			copies to the public, and perform publicly and display
			publicly, by or on behalf of the Government. The Department of Energy will provide public access to these results of federally sponsored research in accordance with the DOE Public Access Plan. http://energy.gov/downloads/doe-public-access-plan. }}
	\normalsize
\end{flushright}

\end{document}